\documentclass{amsart}
\usepackage{latexsym}
\usepackage{amssymb}
\usepackage{amsmath}
\usepackage{amsfonts}
\usepackage{mathrsfs}
\usepackage{color}
\newtheorem{theorem}{Theorem}[section]
\newtheorem{lemma}[theorem]{Lemma}
\newtheorem{proposition}[theorem]{Proposition}
\newtheorem{corollary}[theorem]{Corollary}

\theoremstyle{definition}
\newtheorem{definition}[theorem]{Definition}
\newtheorem{notation}[theorem]{Notation}

\newtheorem{example}[theorem]{Example}

\theoremstyle{remark}
\newtheorem{remark}[theorem]{Remark}
\newtheorem{claim}[theorem]{Claim}
\newtheorem{question}[theorem]{Question}
\newtheorem{problem}[theorem]{Problem}

\def\F{\mathscr{F}}

\def\M{\mathcal{M}}

\def\P{\mathbb{P}}
\def\GG{\mathcal{G}}

\DeclareMathOperator{\fix}{{\rm fix}}
\DeclareMathOperator{\sym}{{\rm sym}}
\DeclareMathOperator{\Orb}{{\rm Orb}}
\DeclareMathOperator{\dom}{{\rm dom}}

\DeclareMathOperator{\HS}{{\rm HS}}
\DeclareMathOperator{\Fn}{{\rm Fn}}

\DeclareMathOperator{\inter}{{\rm int}}
\DeclareMathOperator{\cuc}{{\rm cuc}}
\DeclareMathOperator{\cf}{{\rm cf}}

\begin{document}

\title[Crowded Hausdorff $P$-spaces without $\mathbf{AC}$]{Constructing crowded Hausdorff $P$-spaces in set theory\\ without the axiom of choice}

\author[E. Tachtsis]{Eleftherios Tachtsis}
\address{Department of Statistics and Actuarial-Financial Mathematics\\ University of the Aegean\\ Karlovassi 83200, Samos\\ Greece}
\email{ltah@aegean.gr}

\author[E. Wajch]{Eliza Wajch}
\address{Institute of Mathematics\\ Faculty of Exact and Natural Sciences\\ The University of Siedlce\\ 3 Maja 54, Siedlce, 08-110\\ Poland}
\email{eliza.wajch@gmail.com}

\date{\today}



\begin{abstract}
For an infinite set $X$, a closed under finite unions family $\mathcal{Z}$ with $[X]^{<\omega}\subseteq\mathcal{Z}\subseteq\mathcal{P}(X)$, and any $\mathcal{A}\subseteq\mathcal{P}(X)$, the topology $\tau_{\mathcal{A}}[\mathcal{Z}]=\{V\in\mathcal{A}: (\forall x\in V)(\exists z\in \mathcal{Z})(x\cap z=\emptyset \wedge \{y\in\mathcal{A}: x\subseteq y\subseteq X\setminus z\}\subseteq V)\}$ on $\mathcal{A}$ is investigated to give answers to the following open problem in various models of $\mathbf{ZF}$ or $\mathbf{ZFA}$: Is there a non-empty Hausdorff, crowded zero-dimensional $P$-space in the absence of the axiom of choice? Spaces of the form $\mathbf{S}(X, [X]^{\leq\omega})=\langle \mathcal{A}, \tau_{\mathcal{A}}[\mathcal{Z}]\rangle$ for $\mathcal{A}=[X]^{<\omega}$ and $\mathcal{Z}=[X]^{\leq\omega}$ are of special importance here. Among many other results, the following theorems are proved in $\mathbf{ZF}$:
\begin{enumerate}
\item If $X$ is uncountable, then $\mathbf{S}(X, [X]^{\leq\omega})$ is a crowded zero-dimensional Hausdorff space, and if $X$ is also quasi Dedekind-finite, then $\mathbf{S}(X, [X]^{\leq\omega})$ is a $P$-space;
\item  $\mathbf{S}(\omega_1, [\omega_1]^{\leq\omega})$ is a $P$-space if and only if $\omega_1$ is regular;
\item the axiom of countable choice for finite sets is equivalent to the statement ``for every infinite Dedekind-finite set $X$, $\mathbf{S}(X,[X]^{\leq\omega})$ is a $P$-space'';
\item the statement ``$\mathbb{R}$ admits a topology $\tau$ such that $\langle\mathbb{R}, \tau\rangle$ is a crowded, zero-dimensional Hausdorff $P$-space'' is strictly weaker than the axiom of countable choice for subsets of $\mathbb{R}$;
\item the statement ``there exists a non-empty, well-orderable crowded zero-dimensional Hausdorff $P$-space'' is strictly weaker than ``$\omega_1$ is regular''.
 \end{enumerate}
 
  A lot of relevant independence results are obtained via permutation and symmetric models.
\end{abstract}
\maketitle

\tableofcontents

\section{Introduction}
\label{intro}

Our notation and terminology are given in Section \ref{s2}. All topological notions not introduced here are standard and can be found in \cite{Rich} or \cite{wil}. Our set-theoretic framework is the Zermelo--Fraenkel set theory without the axiom of choice ($\mathbf{ZF}$) and $\mathbf{ZF}$ with atoms ($\mathbf{ZFA}$). As usual, $\mathbf{AC}$ denotes the axiom of choice.

A $P$-space (in the sense of Gillman--Henriksen) is a topological space in which every $G_{\delta}$-set is open (see \cite{gh} and Definition \ref{s2:d7}). Basic facts about $P$-spaces in $\mathbf{ZFC}$ (i.e. $\mathbf{ZF}+\mathbf{AC}$) can be found in \cite{gh, gj}. Recognized for their importance in topology, non-standard analysis and related areas of mathematics, $P$-spaces are accorded a dedicated entry in the Encyclopedia of Mathematics (see \cite{hart}).

In \cite{kow}, Keremedis, Olfati and Wajch
started a systematic, extensive study of $P$-spaces in the absence of $\mathbf{AC}$. In \cite{kow},
many statements concerning $P$-spaces, known to be true in $\mathbf{ZFC}$, were shown to be independent of $\mathbf{ZF}$, and their set-theoretic strength was investigated. A lot of intriguing
problems were left open in \cite{kow}. Some of these problems were solved by Tachtsis in \cite{TachRM}. However, to the
best of our knowledge, the following important problem, posed in \cite{kow}, remains unsolved.

\begin{problem}
\label{s1:q1}
(See \cite[Proposition 1.10 and Problem 1 of Section 6]{kow}.) Is there a non-discrete Tychonoff $P$-space in $\mathbf{ZF}$?
\end{problem}

It was observed in \cite[Remark 4.14]{kow} that, contrary to what happens in $\mathbf{ZFC}$, a regular $P$-space may fail to be zero-dimensional in $\mathbf{ZF}$ (see also \cite{bru} and \cite[p. 182]{hr}). Therefore, the following problems are also important, but not explicitly stated in \cite{kow}.

\begin{problem}
\label{s1:q2}
\begin{enumerate}
\item Does $\mathbf{ZF}$ prove the existence of a non-empty, crowded zero-dimensional Hausdorff $P$-space?
\item Is $\mathbf{ZF}$ $+$ $\neg\mathbf{AC}$ $+$ ``there exists a non-empty, crowded zero-dimensional Hausdorff $P$-space'' relatively consistent with $\mathbf{ZFC}$?
\item Does $\mathbf{ZF}+\cf(\aleph_{1})=\aleph_{1}$ or $\mathbf{ZF}+\mathbf{W}$, where $\mathbf{W}$ denotes some principle strictly weaker than $\cf(\aleph_{1})=\aleph_{1}$, prove the existence of a non-empty, well-orderable, crowded zero-dimensional Hausdorff $P$-space?
\end{enumerate}
\end{problem}

We \textbf{answer the second (and thus the first) part of Problem \ref{s1:q2}(3) \textbf{in the affirmative}} (and thus also \textbf{settle Problem \ref{s1:q2}(2)}, since $\mathbf{ZF}+\neg\mathbf{AC}+\cf(\aleph_1)=\aleph_1$ is relatively consistent with $\mathbf{ZFC}$), and we shed more light on Problem \ref{s1:q2}(1). To achieve this goal, in Section \ref{s3}, for an infinite set $X$, any subset $\mathcal{A}$ of the power set $\mathcal{P}(X)$ of $X$,  and a subfamily $\mathcal{Z}$ of $\mathcal{P}(X)$ such that $\mathcal{Z}$ is closed under finite unions and contains the family $[X]^{<\omega}$ of all finite subsets of $X$, we
introduce the following topology $\tau_{\mathcal{A}}[\mathcal{Z}]$ on $\mathcal{A}$:
\[
\tau_{\mathcal{A}}[\mathcal{Z}] = \{V\in\mathcal{P}(\mathcal{A}): (\forall x\in V )(\exists z\in\mathcal{Z})(x\cap z=\emptyset\wedge \{y\in\mathcal{A}: x\subseteq y\subseteq X\setminus z\}\subseteq V)\}.
\]
The triple $\langle \mathcal{A}, \tau_{\mathcal{A}}[\mathcal{Z}], \subseteq\rangle$ is a convex partially ordered topological space (see Definition \ref{s2:d9}, Remark \ref{s2:r10} and Theorem \ref{s3:t2}).  Needless to say, various convex partially ordered topological spaces have been studied by many authors for a long time. A brief introduction to the class of convex partially ordered topological spaces is included, for instance, in \cite{Rich}. Our aim is to show properties and many applications of topologies of the form $\tau_{\mathcal{A}}[\mathcal{Z}]$, with a focus on conditions under which the topological spaces $\mathbf{S}_{\mathcal{A}}(X,\mathcal{Z})=\langle \mathcal{A}, \tau_{\mathcal{A}}[\mathcal{Z}]\rangle$ are $P$-spaces; the applications are  not only to topological problems relevant to $P$-spaces in $\mathbf{ZF}$, but also to solutions of some set-theoretic problems and independence results. The paper consists of eleven sections and the Appendix. It is organized as follows.

In Section \ref{s2}, we establish our basic notation and terminology. 

In Section \ref{s3}, we show basic properties of the spaces $\mathbf{S}_{\mathcal{A}}(X, \mathcal{Z})$; in particular, of the spaces of the form $\mathbf{S}(X,\mathcal{Z})=\mathbf{S}_{\mathcal{A}}(X, \mathcal{Z})$ where $\mathcal{A}$ is the family of all finite subsets of $X$, and the spaces $\mathbf{S}_{\mathcal{P}}(X, \mathcal{Z})=\mathbf{S}_{\mathcal{P}(X)}(X, \mathcal{Z})$. All spaces $\mathbf{S}_{\mathcal{A}}(X, \mathcal{Z})$ are Hausdorff and zero-dimensional (see Theorem \ref{s3:t2}(v)). For an infinite set $X$, the space $\mathbf{S}(X, \mathcal{Z})$ is crowded if and only if $X\notin\mathcal{Z}$ (see Theorem \ref{s3:t2}(xiii)). Assuming that $\mathcal{Z}$ is a bornology on $X$(see Definition \ref{s2:d8}), we prove that the space $\mathbf{S}(X, \mathcal{Z})$ is homogeneous (see Theorem \ref{s3:t5}(b)) and, in addition, if $\mathbf{S}(X, \mathcal{Z})$ is a $P$-space, then $\mathcal{Z}$ is a $\sigma$-ideal (see Theorem \ref{s3:t6}). A central result is Theorem \ref{s3:t9} in which we provide set-theoretic necessary and sufficient conditions for $\mathbf{S}_{\mathcal{A}}(X, \mathcal{Z})$ to be a $P$-space when $\mathcal{Z}$ and $\mathcal{A}$ are bornologies such that either $\mathcal{A}\subseteq\mathcal{Z}$ or $\mathcal{A}=\mathcal{P}(X)$. In particular, Theorem \ref{s3:t9} gives necessary and sufficient conditions for $\mathbf{S}(X, [X]^{\leq\omega})$ to be a $P$-space. Significant direct applications of Theorem \ref{s3:t9} are shown in Sections \ref{s6} and \ref{s7} (especially to the proofs of Theorems \ref{s6:t4} and \ref{s7:t9}).

In Section \ref{s4}, we introduce a zero-dimensional Hausdorff topology $\tau_{2}[\mathcal{Z}]$ on the set $2^X$. Basic properties of the topological space $2^X[\mathcal{Z}]=\langle 2^X, \tau_2[\mathcal{Z}]\rangle$ are given in Theorem \ref{s4:t2}. For every infinite set $X$, the space $2^X[[X]^{<\omega}]$ is the Cantor cube. That the space $2^X[\mathcal{Z}]$ is homogeneous is shown in Theorem \ref{s4:t6}. If $\mathcal{Z}$ is a bornology on $X$, then $2^X[\mathcal{Z}]$ contains a homeomorphic copy of $\mathbf{S}(X, \mathcal{Z})$, and there is a natural continuous bijection of $\mathbf{S}_{\mathcal{P}}(X, \mathcal{Z})$ onto $2^X[\mathcal{Z}]$ (see Theorem \ref{s4:t3}). Theorem \ref{s4:t8}(ii) asserts that if  $\mathcal{Z}$ is a bornology, then $2^{X}[\mathcal{Z}]$ is a $P$-space if and only if $\mathbf{S}_{\mathcal{P}}(X, \mathcal{Z})$ is a $P$-space. In Theorem \ref{s4:t10}, we show that $\mathbf{CMC}(\aleph_0, \infty)$ (see Definition \ref{s2forms}(10)) implies that, for every infinite set $X$ and every bornology $\mathcal{Z}$ on $X$, it holds that the spaces $\mathbf{S}(X, \mathcal{Z})$ and $2^X[\mathcal{Z}]$ are $P$-spaces if and only if $\mathcal{Z}$ is a $\sigma$-ideal.

A crucial role in Sections \ref{s5}--\ref{s10} is played by the spaces of the form $\mathbf{S}(X, [X]^{\leq\omega})$ where $X$ is an infinite set, and $\mathcal{Z}=[X]^{\leq\omega}$ is the bornology of all countable subsets of $X$; however, the spaces $\mathbf{S}_{\mathcal{P}}(X, [X]^{\leq\omega})$, $2^X[[X[^{\leq\omega}]$,  $\mathbf{S}(X, [X]^{<\omega})$, $\mathbf{S}_{\mathcal{P}}(X, [X]^{<\omega})$ are also taken into consideration. Several results of the paper that require a deep insight into various permutation models of $\mathbf{ZFA}$ (including new ones) are presented in Sections \ref{s6}--\ref{s10}. Some of these results are purely set-theoretic, but all have been inspired by the general question about conditions under which the spaces  $\mathbf{S}_{\mathcal{A}}(X, \mathcal{Z})$ and $2^X[\mathcal{Z}]$ can be $P$-spaces. But besides permutation models, from the perspective of the general question, we also provide (in Sections \ref{s9} and \ref{s10}) novel, non-trivial results about two famous, and very important, symmetric models of $\mathbf{ZF}$, namely the Basic Cohen Model (\cite[Model $\mathcal{M}1$]{hr}) and the Feferman--L\'{e}vy Model (\cite[Model $\mathcal{M}9$]{hr}); see Theorems \ref{thm:Models_M1_N3} and \ref{s10:t1}.

Section \ref{s5} is mainly about the spaces $\mathbf{S}_{\mathcal{P}}(\kappa, [\kappa]^{\leq\omega})$, $2^{\kappa}[[\kappa]^{\leq\omega}]$ and $\mathbf{S}(\kappa, [\kappa]^{\leq\omega})$, where $\kappa$ is an uncountable well-ordered cardinal. In particular, it is shown in Theorem \ref{s5:t4} that the spaces $\mathbf{S}_{\mathcal{P}}(\omega_1, [\omega_1]^{\leq\omega})$, $2^{\omega_1}[[\omega_1]^{\leq\omega}]$ and $\mathbf{S}(\omega_1, [\omega_1]^{\leq\omega})$ are $P$-spaces if and only if $\omega_1$ is regular.

In Section \ref{s6}, we establish in Theorem \ref{s6:t1} that the Principle $\mathbf{CAC}$ of Countable Choice (see Definition \ref{s2forms}(1)) is equivalent to the statement ``For every infinite set $X$, $\mathbf{S}_{\mathcal{P}}(X, [X]^{\leq\omega})$ is a $P$-space''. We prove in $\mathbf{ZF}$ that, for every infinite set $X$, the space $\mathbf{S}_{\mathcal{P}}(X, [X]^{<\omega})$ is not a $P$-space and, therefore, Cantor cubes are not $P$-spaces in $\mathbf{ZF}$ (see Theorem \ref{s6:t2}). In Theorem \ref{s6:t4}, we establish, without invoking any form of choice, that if $X$ is an infinite quasi Dedekind-finite set, then $\mathbf{S}(X,[X]^{\leq\omega})$ is a $P$-space, and we deduce that it is relatively consistent with $\mathbf{ZF}$ that there exists an infinite set $X$ for which $\mathbf{S}(X, [X]^{\leq\omega})$ is a $P$-space, while $\mathbf{S}_{\mathcal{P}}(X, [X]^{\leq\omega})$ is not a $P$-space. Theorem \ref{s6:t4} is one of the \emph{major} theorems of this paper which serves as a basis for producing some significant positive and independence results, namely Theorem \ref{s7:t5}, Theorem \ref{thm:Models_M1_N3}, Theorem \ref{s9:t10}(1), and Theorem \ref{s9:t12}((2), (6)). 

In Section \ref{s7}, we introduce new principles $\mathbf{PS}_0-\mathbf{PS}_5$ in Definition \ref{s7:d1}, and four \textbf{new weak choice principles} in Definition \ref{s7:d7}, all relevant to $\mathbf{PS}_0$ (and one of them shown to be equivalent to $\mathbf{PS}_{0}$). Let us discuss in brief only the following statements from Definition \ref{s7:d1}:  
\begin{itemize}
\item $\mathbf{PS}_0$: For every uncountable set $X$, $\mathbf{S}(X,[X]^{\leq\omega})$ is a $P$-space.

\item  $\mathbf{PS}_1$: For every uncountable set $X$, there exists an uncountable set $Y\subseteq X$ such that $\mathbf{S}(Y,[Y]^{\leq\omega})$ is a $P$-space.

\item $\mathbf{PS}_2$: For every infinite set $X$, there exists an infinite set $Y\subseteq X$ such that $\mathbf{S}(Y,[Y]^{\leq\omega})$ is a $P$-space.
\end{itemize}

\noindent In Theorem \ref{s7:t5}, we show that each of $\mathbf{PS}_2$ and ``for every infinite Dedekind-finite set $X$, $\mathbf{S}(X,[X]^{\leq\omega})$ is a $P$-space'' is equivalent to $\mathbf{CAC_{fin}}$ (see Definition \ref{s2forms}(5)). In Theorem \ref{s7:t4}, we observe that $\mathbf{PS}_0\rightarrow\mathbf{PS}_1\rightarrow\mathbf{CUC}$, where $\mathbf{CUC}$ is the Countable Union Theorem (see Definition \ref{s2forms}(2)).
The question whether or not the latter implications are reversible is \emph{still open}. However, we are able to provide \emph{non-trivial partial answers} to this question in Theorems \ref{s7:t9} and \ref{s7:t12}. In the proof of Theorem \ref{s7:t10}, concerning the principles from Definition \ref{s7:d7}, but giving also more light into $\mathbf{PS}_0$, a new permutation model is constructed.

In Section \ref{s8}, we begin to investigate the set-theoretic strength of the statement ``$(\forall X)(\mathbf{CAC}(X)\rightarrow\mathbf{CAC}(X^{\omega}))$'' (the formula $\mathbf{CAC}(X)$ is defined in Definition \ref{s2forms}(29)). In Theorem \ref{s8:t2}, we show that, for every uncountable set $X$, $\mathbf{CAC}(X^{\omega})$ implies that  
$\mathbf{S}(X, [X]^{\leq\omega})$ is a $P$-space. Since, in every permutation model, for every well-orderable set $X$, $\mathbf{CAC}(X^{\omega})$ holds, we conclude that it is true in every permutation model that if $X$ is an infinite well-orderable set, then $\mathbf{S}(X, [X]^{\leq\omega})$ is a $P$-space (see Corollary \ref{s8:c8}). The following weak choice principle, introduced in Definition \ref{s8:d9}, is worth mentioning here:
\begin{itemize}
\item $\mathbf{\Pi}$: There exists an infinite set $X$ such that $\mathbf{CAC}(X)$ holds, but $\mathbf{S}(X, [X]^{\leq\omega}))$ is not a $P$-space.
\end{itemize}
\noindent It is shown in Theorem \ref{s8:t10} that $\mathbf{\Pi}$ is relatively consistent with $\mathbf{ZF}$ and it follows from $\mathbf{\Pi}$ that the statement ``$(\forall X)\mathbf{CAC}(X)\rightarrow \mathbf{CAC}(X^{\omega}))$'' is false. In the proof of Theorem \ref{s8:t11}, it is shown that $\mathbf{\Pi}$ is true in Howard's Model I (Model $\mathcal{N}18$ in \cite{hr}). In Remark \ref{s8:r15}, a new permutation model in which $\mathbf{\Pi}$ holds is described. 

In Theorem \ref{thm:general_criterion}, we provide a \textbf{general criterion} for the validity of $(\forall X)(\mathbf{CAC}(X) \rightarrow \mathbf{CAC}(X^{\omega}))$ \textbf{in permutation models}. Using this criterion, we prove that in the Second Fraenkel Model (Model $\mathcal{N}2$ in \cite{hr}), the conjunction $((\forall X) (\mathbf{CAC}(X)\rightarrow \mathbf{CAC}(X^{\omega})))\wedge\neg\mathbf{PS}_2$ is true (see the proof of Theorem \ref{s8:t20}). On the other hand, in Theorem \ref{s9:t10}, we show that $\mathbf{PS}_2$ does not imply ``$(\forall X)(\mathbf{CAC}(X)\rightarrow\mathbf{CAC}(X^{\omega}))$'' in $\mathbf{ZF}$. 
 
In Section \ref{s9}, among many other facts, we give solutions to some of the problems posed in the paper. For instance, Corollary \ref{s8:c4} states that $\mathbf{CAC}(\mathbb{R})$ implies that there exists a Hausdorff, zero-dimensional topology $\tau$ on the real line $\mathbb{R}$ such that $\langle\mathbb{R}, \tau\rangle$ is a crowded $P$-space, so it is natural to ask if the implication from Corollary \ref{s8:c4} is reversible in $\mathbf{ZF}$. In the proof of Theorem \ref{thm:Models_M1_N3}, and on the basis of Theorem \ref{s6:t4}, we show that $\mathbf{PS}_0$ is true in the Basic Cohen Model (Model $\mathcal{M}1$ of \cite{hr}), so $\mathbf{S}(\mathbb{R}, [\mathbb{R}]^{\leq\omega})$ is a $P$-space in $\mathcal{M}1$, although $\mathbf{CAC}(\mathbb{R})$ is false in $\mathcal{M}1$.

In Problem \ref{s4:q12}, we ask if $\mathbf{CMC}(\aleph_0, \infty)$ follows from the statement: ``For every infinite set $X$, if $[X]^{\leq\omega}$ is a $\sigma$-ideal, then $\mathbf{S}(X, [X]^{\leq\omega}$ is a $P$-space''. In Theorem \ref{thm:Models_M1_N3}(3), we give a strong negative answer to this question by proving  that $\mathbf{BPI}\wedge \mathbf{PS}_0$ does not imply $\mathbf{CAC}(\mathbb{R})\vee\mathbf{CMC}(\aleph_0, \infty)$ in $\mathbf{ZF}$, where $\mathbf{BPI}$ is the Boolean Prime Ideal Theorem (see Definition \ref{s2forms}(25)). Furthermore, in Theorem \ref{thm:Models_M1_N3}(4), we show, on the basis of Theorem \ref{s6:t4}, that $\mathbf{PS}_0$ is true in the Mostowski Linearly Ordered Model $\mathcal{N}3$ of \cite{hr}. For the readers' further insight and deeper understanding, in the Appendix, we give an alternative, direct proof that $\mathbf{PS}_0$ is true in $\mathcal{N}3$, using only the properties of this model rather than Theorem \ref{s6:t4}.

In Theorem \ref{s9:t2}, we show that $\mathbf{{CAC}_{DLO}}$ (see Definition \ref{s2forms}(18)) implies ``For every infinite linearly orderable set $X$, the space $\mathbf{S}(X, [X]^{\leq\omega})$ is a $P$-space'', and the latter statement implies $\mathbf{{CUC}_{DLO}}$ (see Definition \ref{s2forms}(19)). In the proof of Theorem \ref{thm:CMC_notCUCDLO}, we show that the Good--Tree--Watson Model $\mathcal{N}53$ of \cite{hr} is a permutation model in which $\mathbf{CMC}(\aleph_0, \infty)\wedge\neg \mathbf{{CUC}_{DLO}}$ is true. In Remark \ref{rem:W_aleph}, it is explained why $\mathbf{{CAC}_{DLO}}$ does not imply $\mathbf{PS}_1$ in $\mathbf{ZFA}$. By Theorem \ref{s9:t12}(6), $\mathbf{PS}_0$ does not imply $\mathbf{{CAC}_{DLO}}$ in $\mathbf{ZF}$. The non-trivial proof of Theorem \ref{s9:t15} showing, among other facts, that $\mathbf{PS}_1$ does not imply $\mathbf{UT}(\aleph_0, \mathbf{WO}, \mathbf{WO})$ (see Definition \ref{s2forms}(4)) in $\mathbf{ZFA}$ is also remarkable.

In Section \ref{s10}, we \textbf{fully answer Problem \ref{s1:q2}(3), and thus Problem \ref{s1:q2}(2)}. In particular, denoting by $\mathbf{G}(\cf(\aleph)>\aleph_0)$ the statement ``There exists an aleph of uncountable cofinality'' (see Definition \ref{s2forms}(15)) and by $\mathbf{P}(WO, 0-dim, T_2, crow)$ the statement ``There exists a non-empty well-orderable, zero-dimensional, Hausdorff, crowded $P$-space'' (see Definition \ref{s2forms}(23)), we first observe, on the basis of Theorem \ref{s5:t1}, that $\mathbf{G}(\cf(\aleph)>\aleph_0)$ implies $\mathbf{P}(WO, 0-dim, T_2, crow)$. Then we show that $\mathbf{G}(\cf(\aleph)>\aleph_0)$ (and thus $\mathbf{P}(WO, 0-dim, T_2, crow)$) is true in the Feferman-L\'evy Model $\mathcal{M}9$ of \cite{hr} in which $\aleph_{1}$ is singular. We conclude that $\mathbf{P}(WO, 0-dim, T_2, crow)$ is strictly weaker than $\cf(\aleph_1)=\aleph_1$ in $\mathbf{ZF}$, but $\mathbf{P}(WO, 0-dim, T_2, crow)$ is true in every permutation model (see Theorem \ref{s10:t1} and Corollary \ref{s10:c2}). Furthermore, in Corollary \ref{s10:c2}(6), we show that $\mathbf{P}(WO, 0-dim, T_2, crow)$ is true in the Basic Cohen Model $\mathcal{M}1$ of \cite{hr}; this provides an affirmative answer to Problem \ref{s1:q2}(2). Section \ref{s10} is a very important part of this paper, providing several directions and insights for further deep study in this area. In fact, the evolution of the topics in this paper started from Problem \ref{s1:q2}((2), (3)) and Section \ref{s10}.  

Section \ref{s11} contains a list of open problems relevant to the main results of the paper. 

Concluding the Introduction, let us mention here that it is our firm belief that this paper will be of special interest to set theorists, topologists and, in general, to researchers in infinite combinatorics. It will also be appealing to a broader audience of mathematicians with an awareness of the central role of $\mathbf{AC}$ in every branch of pure mathematics. It should be noted that while several of our results are accessible for a general mathematical audience, many proofs involve subtle (infinite) combinatorics, demanding arguments in models of set theory without $\mathbf{AC}$, and a novel criterion to check the validity of certain (new) fragments of $\mathbf{AC}$ in permutation models.

In our opinion, the paper provides a significant advance in the knowledge of constructing zero-dimensional crowded Hausdorff $P$-spaces in choiceless set theory and it will stimulate new interest and research activity in this intriguing theme in set-theoretic topology.
  
\section{Notation and terminology}
\label{s2}

\begin{notation}
\label{s2:n1}
\begin{enumerate}
\item $\mathbf{ZF}$ denotes the Zermelo--Fraenkel set theory without the Axiom of Choice ($\mathbf{AC}$).

\item $\mathbf{ZFC}$ denotes the $\mathbf{ZF+AC}$ set theory.

\item $\mathbf{ZFA}$ denotes the Zermelo-Fraenkel set theory with atoms (with the Axiom of Extensionality modified to allow the existence of atoms).

\end{enumerate}
\end{notation}

\begin{notation}
\label{def:comparecard}
\begin{enumerate}
\item Let $X$ and $Y$ be two sets. We write:
\begin{enumerate}

\item $|X|\leq |Y|$ if there is an injection $f:X\rightarrow Y$;

\item $|X|=|Y|$ if there is a bijection $f:X\rightarrow Y$ and, in this case, we say that $X$ is \emph{equipotent} to $Y$, or that $X$ and $Y$ are equipotent;

\item $|X|<|Y|$ if $|X|\leq |Y|$ and $|X|\ne|Y|$.
\end{enumerate}

\item Let $X$ be a set and let $\kappa$ be a well-ordered cardinal number (equivalently, an initial ordinal number).
\begin{enumerate}
\item[(d)] $\mathcal{P}(X)$ denotes the power set of $X$;

\item[(e)] $[X]^{<\kappa}:=\{Y\in\mathcal{P}(X):|Y|<|\kappa|\}$;

\item[(f)] $[X]^{\leq\kappa}:=\{Y\in\mathcal{P}(X):|Y|\leq|\kappa|\}$;

\item[(g)] $[X]^{\kappa}:=\{Y\in\mathcal{P}(X):|Y|=|\kappa|\}$;

\item[(h)] if $|X|=|\kappa|$, we write $|X|=\kappa$ and call $\kappa$ the \emph{cardinality} of $X$;

\item[(i)] $\sym(X)$ denotes the group of all permutations of $X$ (with composition of functions as the group operation);

\item[(j)] $\cf(\kappa)$ denotes the cofinality of $\kappa$.

\end{enumerate}

\item Traditionally, for an ordinal $\alpha$, when $\omega_{\alpha}$ is considered as a well-ordered cardinal, we denote it by $\aleph_{\alpha}$, call $\aleph_{\alpha}$ the \emph{cardinality} of $\omega_{\alpha}$ and write $|\omega_{\alpha}|=\aleph_{\alpha}$. In particular, as usual, $\omega$ denotes the set of all natural numbers (finite ordinal numbers),
$\omega_0=\omega$,  and $\aleph_{0}$ denotes the cardinality of $\omega$; $\omega_{1}$ denotes the first uncountable ordinal number, and $\aleph_{1}$ denotes the cardinality of $\omega_{1}$. 

For our convenience, we put $\mathbb{N}=\omega\setminus\{0\}$.

\item If $G$ is a group and $H$ is a subgroup of $G$, then $(G:H)$ denotes the index of $H$ in $G$.

\item For sets $I$ and $J$, $\Fn(I,J)$ denotes the set of all finite partial functions from $I$ into $J$.

\item For sets $X$ and $Y$, $Y^{X}$ denotes the set of all functions from $X$ into $Y$.
\end{enumerate}
\end{notation}

\begin{definition}
\label{def:a}
(\textbf{Notions of finiteness}) A set $X$ is called:
\begin{enumerate}
\item \emph{finite} if $|X|=n$ for some $n\in\omega$, and \emph{infinite} otherwise;

\item \emph{denumerable} if $|X|=\aleph_{0}$; 

\item \emph{countable} if $|X|\leq\aleph_{0}$, and \emph{uncountable} otherwise; 

\item \emph{Dedekind-finite} if $\aleph_{0}\nleq |X|$, and \emph{Dedekind-infinite} otherwise;

\item \emph{weakly Dedekind-finite} if $\mathcal{P}(X)$ is Dedekind-finite, and \emph{weakly Dedekind-infinite} otherwise;

\item \emph{quasi Dedekind-finite} if $[X]^{<\omega}$ is Dedekind-finite, and \emph{quasi Dedekind-infinite} otherwise;

\item \emph{amorphous} if it is infinite and cannot be partitioned into two infinite subsets. (In other words, $X$ is amorphous if it is infinite and its only subsets are the finite and the cofinite ones.) 
\end{enumerate} 
\end{definition}

\begin{definition}
\label{s2forms}
(\textbf{Weak choice forms})
\begin{enumerate}
\item $\mathbf{CAC}$ (the \emph{Axiom of Countable Choice}, \cite[Form 8]{hr}): Every denumerable family of non-empty sets has a choice function.

\item $\mathbf{CUC}$ (the \emph{Countable Union Theorem}, \cite[Form 31]{hr}): Countable unions of countable sets are countable sets.

\item $\mathbf{UT}(\mathbf{WO},\mathbf{WO},\mathbf{WO})$ (\cite[Form 231]{hr}): Well-orderable unions of well-orderable sets are well-orderable sets.

\item $\mathbf{UT}(\aleph_{0},\mathbf{WO},\mathbf{WO})$: Countable unions of well-orderable sets are well-orderable sets.

\item $\mathbf{CAC_{fin}}$ (\cite[Form 10]{hr}): Every denumerable family of non-empty finite sets has a choice function.

\item $\mathbf{WOAC_{fin}}$ (\cite[Form 122]{hr}): Every well-ordered family of non-empty finite sets has a choice function.

\item $\mathbf{CAC}(\mathbb{R})$ (\cite[Form 94]{hr}): Every denumerable family of non-empty subsets of $\mathbb{R}$ has a choice function.

\item $\mathbf{MC}$ (the \emph{Axiom of Multiple Choice}, \cite[Form 67]{hr}): Every family $\mathcal{A}$ of non-empty sets has a multiple choice function (that is, a function $f$ on $\mathcal{A}$ such that, for every $x\in\mathcal{A}$, $f(x)\in [x]^{<\omega}\setminus\{\emptyset\}$).

\item $\mathbf{CMC}$ (the \emph{Axiom of Countable Multiple Choice}, \cite[Form 126]{hr}): Every countable family $\mathcal{A}$ of non-empty sets has a multiple choice function.

\item $\mathbf{CMC}(\aleph_0, \infty)$ (\cite[Form 131]{hr}): For every countable family $\mathcal{A}$ of non-empty sets, there exists a function $f$ on $\mathcal{A}$ such that, for every $x\in\mathcal{A}$, $f(x)\in [x]^{\leq\omega}\setminus\{\emptyset\}$.

\item $\mathbf{Z}(\omega)$ (\cite[Form 214]{hr}): For every family $\mathcal{A}$ of infinite sets, there is a function $f$ on $\mathcal{A}$ such that, for every $x\in\mathcal{A}$, $f(x)\in [x]^{\omega}$.

\item $\mathbf{IDI}$ (\cite[Form 9]{hr}): Every infinite set is Dedekind-infinite.

\item $\mathbf{IWDI}$ (\cite[Form 82]{hr}): Every infinite set is weakly Dedekind-infinite.

\item $\mathbf{IQDI}$ (\cite{kw}): Every infinite set is quasi Dedekind-infinite.

\item $\mathbf{G}(\cf(\aleph)>\aleph_0)$ (\cite[Form 182]{hr}): There exists an aleph of uncountable cofinality.

\item $\cf(\aleph_1)=\aleph_1$ (\cite[Form 34]{hr}): $\aleph_1$ is regular.

\item $\mathbf{AC_{DLO}}$ (\cite{phjr}): For every non-empty set $J$ and every family $\{\langle X_j,\leq_j\rangle: j\in J\}$ of non-empty linearly ordered sets, the family $\{X_j: j\in J\}$ has a choice function.

\item $\mathbf{CAC_{DLO}}$ (\cite{Tach2019}): For every family $\{\langle X_n, \leq_n\rangle: n\in\omega\}$ of non-empty linearly ordered sets, the family $\{X_n: n\in\omega\}$ has a choice function. (In \cite{Tach2019}, $\mathbf{CAC_{DLO}}$ was denoted by $\mathsf{AC_{DLO}^{\aleph_{0}}}$.)

\item $\mathbf{CUC_{DLO}}$: For every family $\{\langle X_n, \leq_n\rangle: n\in\omega\}$ of non-empty linearly ordered countable sets, $\bigcup_{n\in\omega}X_n$ is countable.

\item $\mathbf{vDCP}(\omega)$ (\emph{van Douwen's Countable Choice Principle}, \cite[Form 119]{hr}): For every family $\{\langle X_n, \leq_n\rangle: n\in\omega\}$ of linearly ordered sets such that, for every $n\in\omega$, $\langle X_n, \leq_n\rangle$ is isomorphic with the set $\mathbb{Z}$ of integers equipped with the standard order, the family $\{X_n: n\in\omega\}$ has a choice function.

\item $\mathbf{CAC_{LO}}$: Every denumerable family of non-empty linearly orderable sets has a choice function.

\item $\mathbf{P}(0-dim, T_2, crow)$: There exists a non-empty zero-dimensional, Hausdorff, crowded $P$-space.

\item $\mathbf{P}(WO, 0-dim, T_2, crow)$: There exists a non-empty well-orderable, zero-dimensional, Hausdorff, crowded $P$-space.

\item $\mathbf{PW}$ (\cite[Form 91]{hr}): For every well-orderable set $X$, $\mathcal{P}(X)$ is well orderable.

\item $\mathbf{BPI}$ (\cite[Form 14]{hr}): Every Boolean algebra has a prime ideal.

\item $\mathbf{WOAM}$ (\cite[Form 133]{hr}): Every set is either well orderable or has an amorphous subset.

\item $\mathbf{NAS}$ (\cite[Form 64]{hr}): There are no amorphous sets.

\item Let $\alpha$ be an ordinal. $\mathbf{W}_{\aleph_{\alpha}}$ (\cite[Form 71($\alpha$)]{hr}): For every set $X$, $|X|\leq\aleph_{\alpha}$ or $\aleph_{\alpha}\leq |X|$. 

\item Given a set $X$, we define the following forms:
\begin{enumerate}
\item $\mathbf{CAC}(X)$: Every countable family of non-empty subsets of $X$ has a choice function.

\item $\mathbf{CMC}(X)$: Every countable family $\mathcal{A}$ of non-empty subsets of $X$ has a multiple choice function.
\end{enumerate} 
\end{enumerate}
\end{definition} 

Other forms we are most interested in are introduced in Definitions \ref{s7:d1}, \ref{s7:d7} and \ref{s8:d9}.

\begin{definition}
\label{s2:d7}
Let $\mathbf{X}=\langle X,\tau\rangle$ be a topological space.
\begin{enumerate}
\item A point $x\in X$ is called a $P$-\emph{point} of $\mathbf{X}$ if, for every $G_{\delta}$-set $O$ in $\mathbf{X}$ (i.e. for every set $O$ which is a countable intersection of open sets in $\mathbf{X}$) such that $x\in O$, it holds that $x\in\inter_{\mathbf{X}}(O)$, where $\inter_{\mathbf{X}}(O)$ denotes the interior of $O$ in $\mathbf{X}$.

\item $\mathbf{X}$ is called:
\begin{enumerate}
\item a $P$-\emph{space} if every point of $X$ is a $P$-point of $\mathbf{X}$;

\item \emph{homogeneous} if, for every pair $x_1, x_2$ of distinct points of $X$, there exists an autohomeomorphism $h$ of $\mathbf{X}$ such that $h(x_1)=x_2$;

\item \emph{zero-dimensional} if the family of all clopen subsets of $\mathbf{X}$ is a base of $\mathbf{X}$;

\item \emph{crowded} (or \emph{dense-in-itself}) if, for every $x\in X$, $\{x\}\not\in\tau$, i.e. if $\mathbf{X}$ has no isolated points;

\item \emph{iso-dense} if the set of all isolated points of $\mathbf{X}$ is dense in $\mathbf{X}$.
\end{enumerate}
\end{enumerate}
\end{definition}

The term ``iso-dense space'' was introduced in \cite{ktw}.

\begin{definition}
\label{s2:d8}
(See \cite{Hogbe}.) A \emph{bornology} on a set $X$ is an ideal $\mathcal{I}$ subsets of $X$ such that $\bigcup\mathcal{I}=X$.
\end{definition}  

\begin{definition}
\label{s2:d9}
(See \cite[Definition 10.1.2]{Rich}.) Let $\langle X, \preceq\rangle$ be a partially ordered set, and let $\tau$ be a topology on $X$.
\begin{enumerate}
\item  A subset $C$ of $X$ is called \emph{convex} in $\langle X, \preceq\rangle$ or $\preceq$-convex if it satisfies the following condition:
\[ (\forall a, b\in C)(\forall x\in X)(a\preceq x\preceq b\rightarrow x\in C).\]
\item The topology $\tau$ is called \emph{convex} with respect to $\preceq$ (or $\preceq$-\emph{convex}) if it has a base consisting of $\preceq$-convex sets.
\item The triple $\langle X, \tau, \preceq\rangle$ is called a \emph{partially ordered topological space}. If the topology $\tau$ is $\preceq$-convex, the triple $\langle X, \tau, \preceq\rangle$ is called a \emph{convex partially ordered topological space}.
\end{enumerate} 
\end{definition}

\begin{remark}
\label{s2:r10}
Every family of sets is partially ordered by inclusion $\subseteq$. Therefore, given a set $X$, a family $\mathcal{A}$ of subsets of $X$, and a topology $\tau$ on $\mathcal{A}$, we obtain the partially ordered topological space $\langle \mathcal{A}, \tau, \subseteq\rangle$. 
\end{remark}

\begin{definition}
\label{s2:d11}
A topology $\tau$ on a family $\mathcal{A}$ of subsets of a set $X$ is called a \emph{convex topology} if $\langle\mathcal{A}, \tau, \subseteq\rangle$ is a convex partially ordered topological space.
\end{definition}

\section{Special convex topologies on subsets of power sets}
\label{s3}

Throughout this section, we assume that $X$ is an infinite set, and $\mathcal{Z}$ is a family of subsets of $X$ satisfying the following conditions: 
\begin{enumerate}
\item[(A)] $[X]^{<\omega}\subseteq\mathcal{Z}$,
\item[(B)] $\mathcal{Z}$ is closed under finite unions.
\end{enumerate}

Under our assumptions about $\mathcal{Z}$, we have that $\mathcal{Z}$ is a bornology on $X$ if and only if $\mathcal{Z}$ is an ideal of subsets of $X$.

In what follows in this section, for an arbitrary family $\mathcal{A}$ of subsets of $X$, we introduce a special convex topology $\tau_{\mathcal{A}}[\mathcal{Z}]$ on $\mathcal{A}$, and we show several basic properties of the topological space $\langle \mathcal{A}, \tau_{\mathcal{A}}[\mathcal{Z}]\rangle$. 

For applications in the forthcoming sections, we emphasize on the properties of $\langle \mathcal{A}, \tau_{\mathcal{A}}[\mathcal{Z}]\rangle$ when $[X]^{<\omega}\subseteq\mathcal{A}$ and $\mathcal{A}$ is closed under finite unions.

\begin{definition}
\label{s3:d1}
Let $\mathcal{A}\subseteq \mathcal{P}(X)$.
\begin{enumerate}
\item[(i)] For every $x\in \mathcal{P}(X)$, let $\mathcal{Z}(x)=\{z\in\mathcal{Z}: x\cap z=\emptyset\}$.
\item[(ii)]  For every $x\in \mathcal{A}$ and $z\in\mathcal{Z}(x)$, let
\[ B_{x,z}^{\mathcal{A}}=\{y\in \mathcal{A}: x\subseteq y\subseteq X\setminus z\}.\]
\item[(iii)] For every $x\in \mathcal{A}$, let $\mathcal{B}_{\mathcal{Z}}^{\mathcal{A}}(x)=\{B_{x,z}^{\mathcal{A}}: z\in\mathcal{Z}(x)\}$.
\item[(iv)] $\tau_{\mathcal{A}}[\mathcal{Z}]=\{V\in\mathcal{P}(\mathcal{A}): (\forall x\in V)(\exists U\in\mathcal{B}_{\mathcal{Z}}^{\mathcal{A}}(x))\quad U\subseteq V\}$.
\item[(v)] $\mathbf{S}_{\mathcal{A}}(X, \mathcal{Z})=\langle \mathcal{A}, \tau_{\mathcal{A}}[\mathcal{Z}]\rangle$.
\end{enumerate}

If $\mathcal{A}=\mathcal{P}(X)$ or $\mathcal{A}=[X]^{<\omega}$, then our notation is simplified as follows.
\begin{enumerate}
\item[(vi)] For every $x\in\mathcal{P}(X)$ (respectively, $x\in [X]^{<\omega}$) and every $z\in\mathcal{Z}(x)$, 
$$B_{x, z}^{\mathcal{P}}=B_{x,z}^{\mathcal{P}(X)} \quad(\text{respectively}, B_{x,z}=B_{x,z}^{[X]^{<\omega}}).$$
\item[(vii)] For every $x\in \mathcal{P}(X)$ (respectively, $x\in [X]^{<\omega}$), 
$$\mathcal{B}_{\mathcal{Z}}^{\mathcal{P}}(x)=B_{\mathcal{Z}}^{\mathcal{P}(X)}(x) \quad(\text{respectively}, \mathcal{B}_{\mathcal{Z}}(x)=B_{\mathcal{Z}}^{[X]^{<\omega}}(x)).$$

\item[(viii)] $\tau_{\mathcal{P}}[\mathcal{Z}]=\tau_{\mathcal{P}(X)}[\mathcal{Z}]$ and $\tau[\mathcal{Z}]=\tau_{[X]^{<\omega}}[\mathcal{Z}]$.

\item[(ix)] $\mathbf{S}_{\mathcal{P}}(X, \mathcal{Z})=\mathbf{S}_{\mathcal{P}(X)}(X, \mathcal{Z})$ and $\mathbf{S}(X, \mathcal{Z})=\mathbf{S}_{[X]^{<\omega}}(X, \mathcal{Z}).$
\end{enumerate}
\end{definition}

In the following theorem, we show some basic properties of $\mathbf{S}_{\mathcal{A}}(X, \mathcal{Z})$.

\begin{theorem}
\label{s3:t2}
Given a non-empty subfamily $\mathcal{A}$ of $\mathcal{P}(X)$, the following conditions are satisfied.
\begin{enumerate}
\item[(i)] For every $x\in\mathcal{A}$ and every $z\in\mathcal{Z}(x)$, $B_{x,z}^{\mathcal{A}}=\mathcal{A}\cap B_{x,z}^{\mathcal{P}}$.

\item[(ii)] $\tau_{\mathcal{A}}[\mathcal{Z}]$ is a topology on $\mathcal{A}$.

\item[(iii)] For every $x\in \mathcal{A}$, the family $\mathcal{B}_{\mathcal{Z}}^{\mathcal{A}}(x)$ is a base of neighborhoods of $x$ in $\mathbf{S}_{\mathcal{A}}(X,\mathcal{Z})$, and all members of $\mathcal{B}_{\mathcal{Z}}^{\mathcal{A}}(x)$ are clopen sets in $\mathbf{S}_{\mathcal{A}}(X,\mathcal{Z})$. The triple $\langle \mathcal{A}, \tau_{\mathcal{A}}[\mathcal{Z}], \subseteq\rangle$ is a convex partially ordered topological space.

\item[(iv)] $\mathbf{S}_{\mathcal{A}}(X, \mathcal{Z})$ is a topological subspace of $\mathbf{S}_{\mathcal{P}}(X, \mathcal{Z})$.

\item[(v)] The topological space $\mathbf{S}_{\mathcal{A}}(X, \mathcal{Z})$ is zero-dimensional and Hausdorff.

\item[(vi)] If $x\in\mathcal{A}$ and $X\setminus x\in\mathcal{Z}$, then $x$ is an isolated point of $\mathbf{S}_{\mathcal{A}}(X, \mathcal{Z})$. Every $\subseteq$-maximal element of $\mathcal{A}$ is an isolated point of $\mathbf{S}_{\mathcal{A}}(X, \mathcal{Z})$. In particular, if $X\in\mathcal{A}$, then $X$ is an isolated point of $\mathbf{S}_{\mathcal{A}}(X, \mathcal{Z})$.

\item[(vii)] If $[X]^{<\omega}\subseteq\mathcal{A}$ and $\mathcal{A}$ is closed under finite unions, then $\{x\in\mathcal{A}: X\setminus x\in\mathcal{Z}\}$ is the set of all isolated points of $\mathbf{S}_{\mathcal{A}}(X, \mathcal{Z})$. In particular, $\{X\setminus z: z\in\mathcal{Z}\}$ is the set of all isolated points of $\mathbf{S}_{\mathcal{P}}(X, \mathcal{Z})$, and $\{x\in [X]^{<\omega}: X\setminus x\in\mathcal{Z}\}$ is the set of all isolated points of $\mathbf{S}(X, \mathcal{Z})$.

\item[(viii)] Suppose that $[X]^{<\omega}\subseteq\mathcal{A}$. Then $\emptyset$ is an isolated point of $\mathbf{S}_{\mathcal{A}}(X, \mathcal{Z})$ if and only if $X\in\mathcal{Z}$. Moreover, if $\mathcal{A}$ is also closed under finite unions, $\mathcal{A}\subseteq\mathcal{Z}$, and $\mathbf{S}_{\mathcal{A}}(X, \mathcal{Z})$ has an isolated point, then $\emptyset$ is an isolated point of $\mathbf{S}_{\mathcal{A}}(X, \mathcal{Z})$.

\item[(ix)] Suppose that $\mathcal{A}$ is a bornology on $X$ such that $\mathcal{A}\subseteq\mathcal{Z}$. Suppose that $x\in \mathcal{A}$ and some point of $B_{x,\emptyset}^{\mathcal{A}}$ is isolated in $\mathbf{S}_{\mathcal{A}}(X, \mathcal{Z})$. Then $x$ is an isolated point of $\mathbf{S}_{\mathcal{A}}(X, \mathcal{Z})$ and, in consequence, every $y\in \mathcal{A}$ with $y\subseteq x$ is also an isolated point of $\mathbf{S}_{\mathcal{A}}(X, \mathcal{Z})$.

\item[(x)] If $\{X\setminus z: (\exists x\in\mathcal{A}) z\in \mathcal{Z}(x)\}\subseteq\mathcal{A}$, then the space $\mathbf{S}_{\mathcal{A}}(X, \mathcal{Z})$ is iso-dense. In consequence, the space $\mathbf{S}_{\mathcal{P}}(X, \mathcal{Z})$ is iso-dense.

\item[(xi)] The topological space $\mathbf{S}_{\mathcal{P}}(X, \mathcal{Z})$ is discrete if and only if $\mathcal{Z}=\mathcal{P}(X)$.

\item[(xii)] Suppose that $\mathcal{Z}$ is a bornology on $X$, and $[X]^{<\omega}\subseteq\mathcal{A}$. Then the topological space $\mathbf{S}_{\mathcal{A}}(X, \mathcal{Z})$ is discrete if and only if $X\in\mathcal{Z}$ (equivalently, if and only if $\mathcal{Z}=\mathcal{P}(X)$).

\item[(xiii)] Suppose that the family $\mathcal{A}$ is closed under finite unions, and $[X]^{<\omega}\subseteq\mathcal{A}\subseteq\mathcal{Z}$. Then the topological space $\mathbf{S}_{\mathcal{A}}(X, \mathcal{Z})$ is crowded if and only if $X\notin\mathcal{Z}$. In particular, the space $\mathbf{S}(X, \mathcal{Z})$ is crowded if and only if $X\notin\mathcal{Z}$.
\end{enumerate}
\end{theorem}

\begin{proof}
(i) This follows from the definitions of $B_{x,z}^{\mathcal{A}}$ and $B_{x,z}^{\mathcal{P}}$.
\smallskip

(ii) Let $x\in \mathcal{A}$. Since $\emptyset\in\mathcal{Z}(x)$, we have $\mathcal{B}_{\mathcal{Z}}^{\mathcal{A}}(x)\neq\emptyset$. It is obvious that, for every $z\in\mathcal{Z}(x)$, $x\in B_{x,z}^{\mathcal{A}}$. 

Let $z_1, z_2\in\mathcal{Z}(x)$. Since $\mathcal{Z}$ is closed under finite unions, $z_1\cup z_2\in\mathcal{Z}$. Furthermore, $x\cap(z_1\cup z_2)=\emptyset$. This shows that $z_1\cup z_2\in\mathcal{Z}(x)$. It is easily seen that $B_{x, z_1\cup z_2}^{\mathcal{A}}\subseteq B_{x, z_1}^{\mathcal{A}}\cap B_{x, z_2}^{\mathcal{A}}$. All this taken together implies that $\tau_{\mathcal{A}}[\mathcal{Z}]$ is a topology on $\mathcal{A}$ generated by the family $\bigcup\{\mathcal{B}_{\mathcal{Z}}^{\mathcal{A}}(x): x\in\mathcal{A}\}$.
\smallskip

(iii) Suppose that $x\in\mathcal{A}$ and $z\in\mathcal{Z}(x)$. Consider any $t_0\in B_{x,z}^{\mathcal{A}}$. Then $x\subseteq t_0\subseteq X\setminus z$, so $z\in\mathcal{Z}(t_0)$. It is easy to notice that $B_{t_0,z}^{\mathcal{A}}\subseteq B_{x,z}^{\mathcal{A}}$. This, together with (ii), implies that $\mathcal{B}_{\mathcal{Z}}^{\mathcal{A}}(x)$ is a local base at $x$ in the topological space $\mathbf{S}_{\mathcal{A}}(X, \mathcal{Z})$. It is easily seen that all members of $\mathcal{B}_{\mathcal{Z}}^{\mathcal{A}}(x)$ are convex with respect to $\subseteq$. This implies that $\langle \mathcal{A}, \tau_{\mathcal{Z}}[\mathcal{Z}], \subseteq\rangle$ is a convex partially ordered topological space.

To prove that the set $B_{x,z}^{\mathcal{A}}$ is clopen in $\mathbf{S}_{\mathcal{A}}(X, \mathcal{Z})$, assuming that $B_{x,z}^{\mathcal{A}}\neq\mathcal{A}$, we fix $y\in \mathcal{A}\setminus B_{x,z}^{\mathcal{A}}$. Then either $x\setminus y\neq\emptyset$ or $y\cap z\neq\emptyset$. 

If $x\setminus y\neq\emptyset$, we fix $t_1\in x\setminus y$; next, we notice that $\{t_1\}\in\mathcal{Z}(y)$ and $B_{y, \{t_1\}}^{\mathcal{A}}\cap B_{x,z}^{\mathcal{A}}=\emptyset$. 

Suppose that $x\setminus y=\emptyset$. Then $x\subseteq y$ and $y\cap z\neq\emptyset$. In this case, $B_{y,\emptyset}^{\mathcal{A}}\cap B_{x,z}^{\mathcal{A}}=\emptyset$. Therefore, $\mathcal{A}\setminus B_{x,z}^{\mathcal{A}}\in\tau_{\mathcal{A}}[\mathcal{Z}]$, so $B_{x,z}^{\mathcal{A}}$ is clopen in $\mathbf{S}_{\mathcal{A}}(X, \mathcal{Z})$. This completes the proof of (iii).
\smallskip

(iv) This follows from (i)--(iii).
\smallskip

(v) By (iii), the space $\mathbf{S}_{\mathcal{A}}(X, \mathcal{Z})$ is zero-dimensional. Let us show that it is also Hausdorff. 

Let $x_0,y_0\in \mathcal{A}$ and $x_0\neq y_0$. Then either $x_0\setminus y_0\neq\emptyset$ or $y_0\setminus x_0\neq\emptyset$. Without loss of generality, we may assume that $x_0\setminus y_0\neq\emptyset$.  We fix $x_1\in x_0\setminus y_0$. Then $\{x_1\}\in \mathcal{Z}(y_0)$, and  $B_{x_0, \emptyset}^{\mathcal{A}}\cap B_{y_0, \{x_1\}}^{\mathcal{A}}=\emptyset$. Hence $\mathbf{S}_{\mathcal{A}}(X, \mathcal{Z})$ is a Hausdorff space.
\smallskip 

(vi) Let $x\in\mathcal{A}$. Assuming that $X\setminus x\in\mathcal{Z}$, we have $B_{x, X\setminus x}^{\mathcal{A}}=\{x\}$. Assuming that $x$ is $\subseteq$-maximal in $\mathcal{A}$, we have $B_{x, \emptyset}^{\mathcal{A}}=\{x\}$. Hence, by  (iii), if either $X\setminus x\in\mathcal{Z}$ or $x$ is $\subseteq$-maximal in $\mathcal{A}$, then $x$ is an isolated point of $\mathbf{S}_{\mathcal{A}}(X, \mathcal{Z})$.  This shows that (vi) holds.
\smallskip

(vii) Suppose that $\mathcal{A}$ is closed under finite unions and $[X]^{<\omega}\subseteq \mathcal{A}$. Let $x\in \mathcal{A}$ be an isolated point of $\mathbf{S}_{\mathcal{A}}(X, \mathcal{Z})$. Then, by (iii), there exists $z\in\mathcal{Z}(x)$ with $B_{x,z}^{\mathcal{A}}=\{x\}$. Then $z\subseteq  X\setminus x$. Suppose that $z\neq X\setminus x$. Then we fix $t\in (X\setminus x)\setminus z$. Since $[X]^{<\omega}\subseteq\mathcal{A}$, we have $\{t\}\in\mathcal{A}$. Thus, since $\mathcal{A}$ is closed under finite unions, we have $x\cup\{t\}\in\mathcal{A}$. We notice that $x\cup\{t\}\in B_{x, z}^{\mathcal{A}}$ and $x\cup\{t\}\neq x$. This contradicts the equality $B_{x,z}^{\mathcal{A}}=\{x\}$. Hence $X\setminus x=z$. This, together with (vi), shows that (vii) holds.
\smallskip

(viii) Assume that $[X]^{<\omega}\subseteq\mathcal{A}$. If $X\in\mathcal{Z}$, then it follows from (vi) that $\emptyset$ is an isolated point of $\mathbf{S}_{\mathcal{A}}(X, \mathcal{Z})$. Suppose that $\emptyset$ is an isolated point of $\mathbf{S}_{\mathcal{A}}(X, \mathcal{Z})$. Then, by (iii), we can fix $z\in\mathcal{Z}$ such that $B_{\emptyset, z}^{\mathcal{A}}=\{\emptyset\}$. If $X\neq z$, then we can fix $t\in X\setminus z$. Since $[X]^{<\omega}\subseteq\mathcal{A}$, we have $\{t\}\in\mathcal{A}$. Then $\{t\}\in B_{\emptyset, z}^{\mathcal{A}}$, which is impossible.  Hence $X=z$, so $X\in\mathcal{Z}$. 

Now, we also assume that $\mathcal{A}$ is closed under finite unions, $\mathcal{A}\subseteq\mathcal{Z}$, and $x_0$ is an isolated point of $\mathbf{S}_{\mathcal{A}}(X, \mathcal{Z})$. Then, by (vii), $X\setminus x_0\in\mathcal{Z}$. Since $\mathcal{Z}$ is closed under finite unions and $\mathcal{A}\subseteq\mathcal{Z}$, we have $X=(X\setminus x_0)\cup x_0\in\mathcal{Z}$.  Thus, since $\emptyset\in\mathcal{A}$, it follows from (vi) that $\emptyset$ is an isolated point of $\mathbf{S}_{\mathcal{A}}(X, \mathcal{Z})$. Hence (viii) holds.
\smallskip

(ix) Let $x_0\in B_{x,\emptyset}^{\mathcal{A}}$ be an isolated point of $\mathbf{S}_{\mathcal{A}}(X, \mathcal{Z})$. By (vii), $X\setminus x_0\in\mathcal{Z}$. Since $\mathcal{A}$ is a bornology and $x_0\in\mathcal{A}$, we have $x_0\setminus x\in\mathcal{A}$.  Since $\mathcal{A}\subseteq\mathcal{Z}$, we have $x_0\setminus  x\in\mathcal{Z}$. We notice that since $x\subseteq x_0$, we have $X\setminus x=(X\setminus x_0)\cup (x_0\setminus x)$. We know that both the sets $X\setminus x_0$ and $x_0\setminus x$ belong to $\mathcal{Z}$, and the family $\mathcal{Z}$ is closed under finite unions. Therefore, $X\setminus x\in\mathcal{Z}$. This, by (vi), implies that $x$ is an isolated point of $\mathbf{S}_{\mathcal{A}}(X, \mathcal{Z})$. Hence (ix) holds.
\smallskip

(x) Suppose that $\{X\setminus z: (\exists x\in\mathcal{A}) z\in \mathcal{Z}(x)\}\subseteq\mathcal{A}$. Consider any point $x_0\in\mathcal{A}$ and any $z_0\in\mathcal{Z}(x_0)$. By our assumption about $\mathcal{A}$, $X\setminus z_0\in\mathcal{A}$. It follows from (vi) that $X\setminus z_0$ is an isolated point of $\mathbf{S}_{\mathcal{A}}(X, \mathcal{Z})$. It is easily seen that $X\setminus z_0\in B_{x_0, z_0}^{\mathcal{A}}$. This, together with (iii), completes the proof of (x).
\smallskip

(xi) It follows from (vii) that $\mathbf{S}_{\mathcal{P}}(X, \mathcal{Z})$ is a discrete space if and only if $\{X\setminus z: z\in\mathcal{Z}\}=\mathcal{P}(X)$. The latter equality is true if and only if $\mathcal{Z}=\mathcal{P}(X)$.
\smallskip

(xii) To prove (xii), we assume that $\mathcal{Z}$ is a bornology on $X$, and $[X]^{<\omega}\subseteq\mathcal{A}$. Suppose that $X\in\mathcal{Z}$. Then $\mathcal{Z}=\mathcal{P}(X)$ and it follows from (vi) that $\mathbf{S}_{\mathcal{A}}(X, \mathcal{Z})$ is discrete. 

On the other hand, if $\mathbf{S}_{\mathcal{A}}(X, \mathcal{Z})$ is discrete, then, since $\emptyset\in\mathcal{A}$, we deduce from (viii) that $X\in\mathcal{Z}$. Hence (xii) holds.
\smallskip

(xiii) To prove (xiii), we assume that $\mathcal{A}$ is closed under finite unions, and $[X]^{<\omega}\subseteq\mathcal{A}\subseteq\mathcal{Z}$. Since $\emptyset\in\mathcal{A}$, if $\mathbf{S}_{\mathcal{A}}(X, \mathcal{Z})$ is crowded, it follows from (viii) that $X\notin\mathcal{Z}$.

Now, let us assume that $X\notin\mathcal{Z}$. If $\mathbf{S}_{\mathcal{A}}(X, \mathcal{Z})$ is not crowded, there exists $x_0\in\mathcal{A}$ such that $x_0$ is an isolated point of $\mathbf{S}_{\mathcal{A}}(X, \mathcal{Z})$. It follows from (vii) that $X\setminus x_0\in\mathcal{Z}$. Since $\mathcal{A}\subseteq\mathcal{Z}$, we have $x_0\in\mathcal{Z}$. Therefore, $X=(X\setminus x_0)\cup x_0\in\mathcal{Z}$ because $\mathcal{Z}$ is closed under finite unions.
\end{proof}

We have the following corollaries to Theorem \ref{s3:t2}.

\begin{corollary}
\label{s3:c3}
Let $\mathcal{A}$ be a family of subsets of $X$ such that $\mathcal{A}$ is closed under finite unions, and $[X]^{<\omega}\subseteq\mathcal{A}\subseteq\mathcal{Z}$. Then the following conditions are equivalent:
\begin{enumerate}
\item[(i)] $\mathbf{S}_{\mathcal{A}}(X,\mathcal{Z})$ is crowded;
\item[(ii)] $\mathbf{S}(X, \mathcal{Z})$ is crowded;
\item[(iii)] $\emptyset$ is not an isolated point of $\mathbf{S}_{\mathcal{P}}(X, \mathcal{Z})$.
\end{enumerate}
\end{corollary}
\begin{proof}
This follows directly from items (viii) and (xiii) of Theorem \ref{s3:t2}.
\end{proof}

\begin{corollary}
\label{s3:c4}
For every infinite set $X$, the following conditions are satisfied.
\begin{enumerate}
\item[(a)] The space $\mathbf{S}(X, [X]^{\leq\omega})$ is crowded if and only if $X$ is uncountable.
\item[(b)] The space $\mathbf{S}(X, [X]^{\leq\omega})$ is discrete if and only if $X$ is denumerable.
\end{enumerate}
\end{corollary}
\begin{proof}
Condition (a) follows from item (xiii) of Theorem \ref{s3:t2}. Condition (b) is a consequence of item (xii) of Theorem \ref{s3:t2}.
\end{proof}

It is worthwhile to turn attention to the problem of when $\mathbf{S}_{\mathcal{A}}(X, \mathcal{Z})$ can be homogeneous. 

\begin{theorem}
\label{s3:t5}
\begin{enumerate}
\item[(a)] The space $\mathbf{S}_{\mathcal{P}}(X, \mathcal{Z})$ is homogeneous if and only if $\mathcal{Z}=\mathcal{P}(X)$.
\item[(b)] If $\mathcal{Z}$ and $\mathcal{A}$ are bornologies on $X$ such that $\mathcal{A}\subseteq\mathcal{Z}$, then the space $\mathbf{S}_{\mathcal{A}}(X, \mathcal{Z})$ is homogeneous.
\item[(c)] If $\mathcal{Z}$ is a bornology on $X$, then the space $\mathbf{S}(X, \mathcal{Z})$ is homogeneous.
\end{enumerate}
\end{theorem}

\begin{proof}
(a) By Theorem \ref{s3:t2}(xii), if $\mathcal{Z}=\mathcal{P}(X)$, then the space $\mathbf{S}_{\mathcal{P}}(X, \mathcal{Z})$ is discrete, so also homogeneous. On the other hand, if $\mathcal{Z}\neq\mathcal{P}(X)$, then, by Theorem \ref{s3:t2}(vi)--(vii), the space $\mathbf{S}_{\mathcal{P}}(X, \mathcal{Z})$ has an isolated point and an accumulation point, so it cannot be homogeneous.
\smallskip

(b) Now, we assume that $\mathcal{Z}$ and $\mathcal{A}$ are bornologies on $X$ such that $\mathcal{A}\subseteq\mathcal{Z}$. If $X\in\mathcal{Z}$, then, by Theorem \ref{s3:t2}(xii), $\mathbf{S}_{\mathcal{A}}(X, \mathcal{Z})$ is a discrete space. Since every discrete space is homogeneous, without loss of generality, we may assume that $X\notin\mathcal{Z}$.  We fix $x\in \mathcal{A}$ with $x\neq\emptyset$. Clearly, $\emptyset\in\mathcal{A}$. To prove that $S_{\mathcal{A}}(X,\mathcal{Z})$ is homogeneous, it suffices to show an autohomeomorphism $h$ of $\mathbf{S}_{\mathcal{A}}(X, \mathcal{Z})$ such that $h(\emptyset)=x$.

We notice that, since $\mathcal{A}\subseteq\mathcal{Z}$, we have $x\in\mathcal{Z}$. Furthermore, $B_{\emptyset,x}^{\mathcal{A}}\cap B_{x,\emptyset}^{\mathcal{A}}=\emptyset$.  By Theorem \ref{s3:t2}(iii), the sets $B_{\emptyset, x}^{\mathcal{A}}$ and $B_{x, \emptyset}^{\mathcal{A}}$ are both clopen in $\mathbf{S}_{\mathcal{A}}(X, \mathcal{Z})$. Since $\mathcal{A}$ is a bornology, and $x\in\mathcal{A}$, we have that, for every $y\in \mathcal{A}$, $y\setminus x\in\mathcal{A}$ and $y\cup x\in\mathcal{A}$. Therefore, we can define the function $h: \mathcal{A}\to \mathcal{A}$ by stipulating, for every $y\in\mathcal{A}$,
\[
h(y)=\begin{cases} y\cup x &\text{ if } y\in B_{\emptyset, x}^{\mathcal{A}},\\
y\setminus x  &\text{ if } y\in B_{x, \emptyset}^{\mathcal{A}},\\
y &\text{ if } y\in \mathcal{A}\setminus (B_{\emptyset, x}^{\mathcal{A}}\cup B_{x, \emptyset}^{\mathcal{A}}).\end{cases}
\]
\noindent We notice that if $y\in B_{\emptyset, x}^{\mathcal{A}}$, then $y\cup x\in B_{x, \emptyset}^{\mathcal{A}}$, and if $y\in B_{x,\emptyset}^{\mathcal{A}}$, then $y\setminus x\in B_{\emptyset, x}^{\mathcal{A}}$. Hence, it is easily seen that $h$ is a bijection. Clearly, $h(\emptyset)=x$. Furthermore, for every $y\in \mathcal{A}$, we have $h\circ h(y)= y$, which implies that $h^{-1}=h$. In consequence, to complete the proof of (b), it suffices to prove that $h$ is 
$\langle\tau_{\mathcal{A}}[\mathcal{Z}], \tau_{\mathcal{A}}[\mathcal{Z}]\rangle$-continuous at every point $y\in \mathcal{A}$.

Let us fix $y\in \mathcal{A}$. Since the set $\mathcal{A}\setminus( B_{\emptyset,x}^{\mathcal{A}}\cup B_{x,\emptyset}^{\mathcal{A}})$ is open in $\mathbf{S}_{\mathcal{A}}(X, \mathcal{Z})$, $h$ is $\langle\tau_{\mathcal{A}}[\mathcal{Z}], \tau_{\mathcal{A}}[\mathcal{Z}]\rangle$-continuous at $y$ if $y\in \mathcal{A}\setminus( B_{\emptyset,x}^{\mathcal{A}}\cup B_{x,\emptyset}^{\mathcal{A}})$.

Suppose that $y\in B_{\emptyset, x}^{\mathcal{A}}$. Then $h(y)=y\cup x\in B_{x,\emptyset}^{\mathcal{A}}$. Consider any $z\in \mathcal{Z}(y\cup x)$. We notice that $B_{y\cup x, z}^{\mathcal{A}}\subseteq B_{y\cup x, \emptyset}^{\mathcal{A}}\subseteq B_{x, \emptyset}^{\mathcal{A}}$. Since $x\cup z\in\mathcal{Z}(y)$, we have the neighborhood $B_{y,x\cup z}^{\mathcal{A}}$ of $y$ in $\mathbf{S}_{\mathcal{A}}(X, \mathcal{Z})$. Clearly, $B_{y,  x\cup z}^{\mathcal{A}}\subseteq B_{\emptyset, x}^{\mathcal{A}}$. Let $t\in B_{y, x\cup z}^{\mathcal{A}}$. Then $h(t)=t\cup x$. We have $y\subseteq t$ and $t\cap (x\cup z)=\emptyset=x\cap z$. Then $y\cup x\subseteq t\cup x$ and $(t\cup x)\cap z=\emptyset$; thus $t\in B_{y\cup x, z}^{\mathcal{A}}$. This shows that $h[B_{y,z\cup x}^{\mathcal{A}}]\subseteq B_{y\cup x, z}^{\mathcal{A}}$. Hence, if $y\in B_{\emptyset, x}^{\mathcal{A}}$, then $h$ is $\langle\tau_{\mathcal{A}}[\mathcal{Z}], \tau_{\mathcal{A}}[\mathcal{Z}]\rangle$-continuous at $y$ by Theorem \ref{s3:t2}(iii).

Now, suppose that $y\in B_{x,\emptyset}^{\mathcal{A}}$. Then $h(y)=y\setminus x\in B_{\emptyset, x}^{\mathcal{A}}$, $x\in\mathcal{Z}(y\setminus x)$ and $B_{y\setminus x, x}^{\mathcal{A}}\subseteq B_{\emptyset, x}^{\mathcal{A}}$. Let us fix any $z_1\in\mathcal{Z}(y\setminus x)$ such that $B_{y\setminus x, z_1}^{\mathcal{A}}\subseteq B_{\emptyset, x}^{\mathcal{A}}$. Since $\mathcal{Z}$ is a bornology on $X$ and $z_1\in\mathcal{Z}$, we have $z_1\setminus x\in\mathcal{Z}$. Since $(y\setminus x)\cap z_1=\emptyset$, we have $y\cap (z_1\setminus x)=\emptyset$. Hence $z_1\setminus x\in\mathcal{Z}(y)$. We can consider the neighborhood $B_{y, z_1\setminus x}^{\mathcal{A}}$ of $y$ in $\mathbf{S}_{\mathcal{A}}(X, \mathcal{Z})$. We fix an arbitrary $t_1\in B_{y, z_1\setminus x}^{\mathcal{A}}$. Then $x\subseteq y\subseteq t_1$, so  $t_1\in B_{x, \emptyset}^{\mathcal{A}}$ and, therefore, $h(t_1)=t_1\setminus x$. Since $t_1\cap (z_1\setminus x)=\emptyset$, we have $(t_1\setminus x)\cap z_1=\emptyset$. In consequence, $t_1\setminus x\in B_{y\setminus x, z_1}^{\mathcal{A}}$, so $h[B_{y, z_1\setminus x}^{\mathcal{A}}]\subseteq B_{y\cap x, z_1}^{\mathcal{A}}$. This, together with Theorem \ref{s3:t2}(iii), shows that $h$ is $\langle \tau_{\mathcal{A}}[\mathcal{Z}], \tau_{\mathcal{A}}[\mathcal{Z}]\rangle$-continuous at $y$. In this way, we have proved that $h$ is an autohomeomorphism of $\mathbf{S}_{\mathcal{A}}(X, \mathcal{Z})$ such that $h(\emptyset)=x$.
Hence (b) is true. It follows from (b) that (c) is also true.
\end{proof}

We are most interested in conditions under which $\mathbf{S}_{\mathcal{A}}(X, \mathcal{Z})$ is a $P$-space. Let us state our preliminary theorem concerning such conditions.

\begin{theorem}
\label{s3:t6}
Let $\mathcal{A}$ be a family of subsets of $X$ such that $\emptyset\in \mathcal{A}$ and every singleton of $X$ belongs to $\mathcal{A}$. Then the following conditions are satisfied.
\begin{enumerate}
\item[(i)] If $\mathcal{Z}$ is a bornology on $X$ such that $\mathbf{S}_{\mathcal{A}}(X, \mathcal{Z})$ is a $P$-space, then $\mathcal{Z}$ is a $\sigma$-ideal.
\item[(ii)] If both $\mathcal{Z}$ and $\mathcal{A}$ are bornologies on $X$ with $\mathcal{A}\subseteq\mathcal{Z}$, then the following conditions are equivalent:
\begin{enumerate}
\item[(a)] $\mathbf{S}_{\mathcal{A}}(X, \mathcal{Z})$ is a $P$-space;
\item[(b)] $\mathbf{S}_{\mathcal{A}}(X, \mathcal{Z})$ has a $P$-point;
\item[(c)] $\emptyset$ is a $P$-point of $\mathbf{S}_{\mathcal{A}}(X, \mathcal{Z})$.
\end{enumerate}
\item[(iii)] If $\mathcal{Z}=\mathcal{P}(X)$ and $[X]^{<\omega}\subseteq\mathcal{A}$, then $\mathbf{S}_{\mathcal{A}}(X, \mathcal{Z})$ is a $P$-space.
\item[(iv)] If $\mathcal{Z}$ is a bornology on $X$, then $\mathbf{S}_{\mathcal{P}}(X, \mathcal{Z})$ is a $P$-space if and only if $\emptyset$ is a $P$-point of $\mathbf{S}_{\mathcal{P}}(X, \mathcal{Z})$.
\end{enumerate}
\end{theorem}

\begin{proof} Let us assume that $\mathcal{Z}$ is a bornology on $X$.

(i) Assuming that $\mathbf{S}_{\mathcal{A}}(X, \mathcal{Z})$ is a $P$-space, to show that $\mathcal{Z}$ is a $\sigma$-ideal, we consider any family $\{t_n: n\in\omega\}$ of members of $\mathcal{Z}$. We notice that $\{t_n: n\in\omega\}\subseteq\mathcal{Z}(\emptyset)=\mathcal{Z}$. Thus, since $\emptyset\in \mathcal{A}$, it follows from Theorem \ref{s3:t2}(iii) that, for every $n\in\omega$, we have the open neighborhood $B_{\emptyset, t_n}^{\mathcal{A}}$ of $\emptyset$ in $\mathbf{S}_{\mathcal{A}}(X, \mathcal{Z})$. Let $U=\bigcap_{n\in\omega}B_{\emptyset, t_n}^{\mathcal{A}}$. Since $\mathbf{S}_{\mathcal{A}}(X, \mathcal{Z})$ is a $P$-space and $\emptyset\in U$, we deduce from Theorem \ref{s3:t2}(iii) that there exists $z_0\in\mathcal{Z}$ such that  $B_{\emptyset, z_0}^{\mathcal{A}}\subseteq U$. We fix $n\in\omega$ and consider any $y\in X\setminus z_0$. Then $z_0\in\mathcal{Z}(\{y\})$. Since every singleton of $X$ belongs to $\mathcal{A}$, we have $\{y\}\in\mathcal{A}$. Then $\{y\}\in B_{\emptyset, z_0}^{\mathcal{A}}\subseteq B_{\emptyset, t_n}^{\mathcal{A}}$. Hence $\{y\}\cap t_n=\emptyset$. This shows that $X\setminus z_0\subseteq X\setminus t_n$, which implies that $t_n\subseteq z_0$. Therefore, $\bigcup_{n\in\omega}t_n\subseteq z_0$. Since $\mathcal{Z}$ is an ideal of subsets of $X$, we infer that $\bigcup_{n\in\omega}t_n\in\mathcal{Z}$. This concludes the proof of (i).

(ii)  Now, we assume that $\mathcal{A}$ and $\mathcal{Z}$ are bornologies on $X$ such that $\mathcal{A}\subseteq\mathcal{Z}$. Then $\emptyset\in\mathcal{A}$ and, by Theorem \ref{s3:t5}, the space $\mathbf{S}_{\mathcal{A}}(X, \mathcal{Z})$ is homogeneous. This implies that (ii) holds.
\smallskip

(iii) Since every discrete space is a $P$-space, it follows from Theorem \ref{s3:t2}(xii) that (iii) holds.
\smallskip

(iv) Assume that $\mathcal{Z}$ is a bornology on $X$, and $\emptyset$ is a $P$-point of $\mathbf{S}_{\mathcal{P}}(X, \mathcal{Z})$. Let $\{W_n: n\in\omega\}$ be a family of open sets of $\mathbf{S}_{\mathcal{P}}(X, \mathcal{Z})$ such that $\bigcap_{n\in\omega}W_n\neq\emptyset$. We fix $x\in\bigcap_{n\in\omega}W_n$. For every $n\in\omega$, let 
\[
A_n=\{z\in \mathcal{Z}(x): B^{\mathcal{P}}_{x, z}\subseteq W_n\}\quad\text{and}\quad V_n=\bigcup\Big\{B^{\mathcal{P}}_{\emptyset, z}: z\in A_n\Big\}.
\]
By Theorem \ref{s3:t2}(iii), for every $n\in\omega$, $A_n\neq\emptyset$, the set $V_n$ is open in $\mathbf{S}_{\mathcal{P}}(X, \mathcal{Z})$ and $\emptyset\in V_n$. Since $\emptyset$ is a $P$-point of $\mathbf{S}_{\mathcal{P}}(X, \mathcal{Z})$, it folows from Theorem \ref{s3:t2}(iii) that there exists $z_0\in\mathcal{Z}$ such that $B^{\mathcal{P}}_{\emptyset, z_0}\subseteq\bigcap_{n\in\omega}V_n$. Since $\mathcal{Z}$ is a bornology on $X$, we have $z_0\setminus x\in\mathcal{Z}(x)$. Thus, we can consider the neighborhood $B^{\mathcal{P}}_{x, z_0\setminus x}$ of $x$ in $\mathbf{S}_{\mathcal{P}}(X, \mathcal{Z})$. To show that $B^{\mathcal{P}}_{x, z_0\setminus x}\subseteq \bigcap_{n\in\omega}W_n$, we fix $n\in\omega$ and $y\in B^{\mathcal{P}}_{x, z_0\setminus x}$. Then $(y\setminus x)\cap z_0=\emptyset$ and, therefore, $y\setminus x\in B^{\mathcal{P}}_{\emptyset, z_0}$. Since $B^{\mathcal{P}}_{\emptyset, z_0}\subseteq V_n$, we have $y\setminus x\in V_n$. There exists $z_y\in A_n$ such that $y\setminus x\in B^{\mathcal{P}}_{x, z_y}$. Since $(y\setminus x)\cap z_y=\emptyset$ and $x\cap z_y=\emptyset$, and $y=(y\setminus x)\cup (y\cap x)$, we have $y\cap z_y=\emptyset$. This, together with the inclusion $x\subseteq y$, implies that $y\in B^{\mathcal{P}}_{x, z_y}\subseteq W_n$. Hence $B^{\mathcal{P}}_{x, z_0\setminus x}\subseteq\bigcap_{n\in\omega}W_n$. Thus, we can conclude that (iv) holds.
\end{proof}

\begin{remark}
\label{s3:r7}
Assume that $\mathcal{A}$ and $\mathcal{Z}$ are bornologies on $X$ such that $\mathcal{A}\subseteq\mathcal{Z}$. Under this assumption, not  involving the homogeneity of $\mathcal{S}_{\mathcal{A}}(X, \mathcal{Z})$, let us give below an alternative direct proof that conditions (a)--(c) of item (ii) of Theorem \ref{s3:t6} are equivalent. To some extent, the arguments given below are similar to the proof of item (iv) of Theorem \ref{s3:t6}, but there are also important differences.
\smallskip

Suppose that $x_0$ is a $P$-point of $\mathbf{S}_{\mathcal{A}}(X, \mathcal{Z})$. We fix any $x\in\mathcal{A}$ with $x\neq x_0$. To show that $x$ is also a $P$-point of $\mathbf{S}_{\mathcal{A}}(X, \mathcal{Z})$, we consider any family $\{W_n: n\in\omega\}$ of open sets in $\mathbf{S}_{\mathcal{A}}(X, \mathcal{Z})$ such that $x\in\bigcap_{n\in\omega}W_n$. For every $n\in\omega$, let $A_n=\{z\in\mathcal{Z}(x): B_{x, z}^{\mathcal{A}}\subseteq W_n\}$. As in the proof of item (iv) of Theorem \ref{s3:t6}, we notice that, by Theorem \ref{s3:t2}(iii), the sets $A_n$ are all non-empty. Since $\mathcal{Z}$ is a bornology on $X$, it follows that, for every $z\in\mathcal{Z}$, we have $z\setminus x_0\in\mathcal{Z}(x_0)$. For every $n\in\omega$, let $V_n=\bigcup\{B_{x_0, z\setminus x_0}^{\mathcal{A}}: z\in A_n\}$. Then, by Theorem \ref{s3:t2}(iii), the sets $V_n$ are all open in $\mathbf{S}_{\mathcal{A}}(X, \mathcal{Z})$. Furthermore, $x_0\in\bigcap_{n\in\omega}V_n$. Since $x_0$ is a $P$-point of $\mathbf{S}_{\mathcal{A}}(X, \mathcal{Z})$, it follows from Theorem \ref{s3:t2}(iii) that there exists $z_0\in\mathcal{Z}(x_0)$ such that $B_{x_0, z_0}^{\mathcal{A}}\subseteq\bigcap_{n\in\omega}V_n$. Using the assumption that $\mathcal{Z}$ is a bornology with $\mathcal{A}\subseteq\mathcal{Z}$, we infer that $(x_0\cup z_0)\setminus x\in\mathcal{Z}(x)$.  Our aim is to check that $B_{x, (x_0\cup z_0)\setminus x}^{\mathcal{A}}\subseteq\bigcap_{n\in\omega}W_n$.

We fix $n\in\omega$ and $y\in B_{x, (x_0\cup z_0)\setminus x}^{\mathcal{A}}$. Let $y_1=(y\setminus x)\cup x_0$. Since $\mathcal{A}$ is a bornology, we have $y_1\in\mathcal{A}$. We notice that $z_0\in\mathcal{Z}(y_1)$ and $y_1\in B_{x_0, z_0}^{\mathcal{A}}\subseteq V_n$. Hence, there exists $z_1\in A_n$ such that $y_1\in B_{x_0, z_1\setminus x_0}^{\mathcal{A}}$. Since $x\subseteq y$, to show that $y\in B_{x, z_1}^{\mathcal{A}}$, it suffices to check that $y\cap z_1=\emptyset$.

Suppose that $t\in y\cap z_1$. Since $x\cap z_1=\emptyset$, we have $t\in y\setminus x$, which implies that $t\in y_1\cap z_1$. Hence $t\in y\cap (x_0\setminus x)$. This is impossible because $y\cap (x_0\setminus x)=\emptyset$. The contradiction obtained shows that $y\cap z_1=\emptyset$. Therefore, we infer that $y\in B_{x, z_1}^{\mathcal{A}}\subseteq W_n$. This implies that $B_{x, (x_0\cup z_0)\setminus x}^{\mathcal{A}}\subseteq\bigcap_{n\in\omega}W_n$, so $x$ is a $P$-point of $\mathbf{S}_{\mathcal{A}}(X, \mathcal{Z})$ by Theorem \ref{s3:t2}(iii). In this way, we have proved that, in Theorem \ref{s3:t6}(ii), conditions (b) and (c) are equivalent, and (b) implies (a). It is straightforward that (a) implies (b).
\end{remark} 

The following illuminating example shows that the assumption that $\mathcal{Z}$ is a bornology is essential in Theorems \ref{s3:t2}(xii),  \ref{s3:t5}(b)--(c) and \ref{s3:t6}(ii). It is worthwhile to compare this example with the forthcoming Theorem \ref{s6:t4}.

\begin{example}
\label{s3:e8}
Let $X$ be an infinite set, and $c_0$ its fixed finite subset. Let 
$$\mathcal{Z}=[X]^{<\omega}\cup\{(X\setminus c_0)\cup c: c\in [X]^{<\omega}\}.$$
Then the family  $\mathcal{Z}$ is closed under finite unions, but $\mathcal{Z}$ is not an ideal of subsets of $X$. We consider the space $\mathbf{S}(X, \mathcal{Z})$. It follows from Theorem \ref{s3:t2}(vii) that only subsets of $c_0$ are isolated points of $\mathbf{S}(X, \mathcal{Z})$. Hence $\mathcal{S}(X, \mathcal{Z})$ is not a discrete space and it is not homogeneous. 

Let us fix $x_0\in [X]^{<\omega}$ such that $x_0\setminus c_0\neq\emptyset$. Assuming that the set $X$ is quasi Dedekind-infinite, let us show that $x_0$ is not a $P$-point of $\mathbf{S}(X, \mathcal{Z})$.

Since $X$ is quasi Dedekind-infinite, there exists a family $\{t_n: n\in\omega\}$ of pairwise distinct members of $[X]^{<\omega}$. We may assume that, for every $n\in\omega$, $x_0\cap t_n=\emptyset$. Then, for every $n\in\omega$, we have the open neighborhood $B_{x_0, t_n}$ of $x_0$ in $\mathbf{S}(X, \mathcal{Z})$. Suppose that $x_0$ is a $P$-point of $\mathbf{S}(X, \mathcal{Z})$. Then there exists $z_0\in\mathcal{Z}(x_0)$ such that $B_{x_0, z_0}\subseteq\bigcap_{n\in\omega}B_{x_0, t_n}$. Now, let us modify the proof of item (i) of Theorem \ref{s3:t6}. Given $n\in\omega$ and an arbitrary $y\in X\setminus (x_0\cup z_0)$, we have $x_0\cup\{y\}\in B_{x_0, z_0}\subseteq B_{x_0, t_n}$, so $y\in X\setminus t_n$. This shows that, for every $n\in\omega$, $t_n\subseteq x_0\cup z_0$. Hence $\bigcup_{n\in\omega}t_n\subseteq x_0\cup z_0$. Since $\bigcup_{n\in\omega}t_n$ is an infinite set, while the set $x_0$ is finite, we infer that $z_0$ is infinite. Thus, there exists $c\in [X]^{<\omega}$ such that $z_0=(X\setminus c_0)\cup c$. This implies that $x_0\setminus c_0\subseteq z_0$. This is impossible because $x_0\cap z_0=\emptyset\neq x_0\setminus c_0$. The contradiction obtained shows that $x_0$ is not a $P$-point of $\mathbf{S}(X, \mathcal{Z})$. By Theorem \ref{s3:t2}(iv), 
$\mathbf{S}(X, \mathcal{Z})$ is a topological subspace of $\mathbf{S}_{\mathcal{P}}(X, \mathcal{Z})$, so $x_0$ is not a $P$-point of $\mathbf{S}_{\mathcal{P}}(X, \mathcal{Z})$.

In particular, if $\mathcal{Z}=[\mathbb{R}]^{<\omega}$, then, for $\mathcal{Z}=[\mathbb{R}]^{<\omega}\cup\{(\mathbb{R}\setminus\{0,1\})\cup c: c\in [\mathbb{R}]^{<\omega}\}$, we obtain that only  $\emptyset$, $\{0\}$, $\{1\}$ and $\{0, 1\}$ are $P$-points of $\mathbf{S}(\mathbb{R},\mathcal{Z})$, and the space $\mathbf{S}(\mathbb{R},\mathcal{Z})$ is neither discrete nor homogeneous.
\end{example}

In the following general theorem, we show purely set-theoretic necessary and sufficient conditions for $\mathbf{S}_{\mathcal{A}}(X, \mathcal{Z})$ to be a $P$-space when $\mathcal{Z}$ and $\mathcal{A}$ are bornologies such that either $\mathcal{A}\subseteq\mathcal{Z}$ or $\mathcal{A}=\mathcal{P}(X)$. In particular, this theorem gives useful set-theoretic necessary and sufficient conditions for $\mathbf{S}(X, [X]^{\leq\omega})$ to be a $P$-space. Their significant direct applications will be shown in the forthcoming Sections \ref{s6}--\ref{s7}.

\begin{theorem}
\label{s3:t9}
Suppose that $\mathcal{Z}$ and $\mathcal{A}$ are bornologies on an infinite set $X$ and either $\mathcal{A}\subseteq\mathcal{Z}$ or $\mathcal{A}=\mathcal{P}(X)$. Then  the following conditions are equivalent:
\begin{enumerate}
\item[(i)] $\mathbf{S}_{\mathcal{A}}(X, \mathcal{Z})$ is a $P$-space;
\item[(ii)] for every family $\{A_n: n\in\omega\}$ satisfying the following conditions:
\begin{enumerate}
\item[($P_1$)] $(\forall n\in\omega)\quad \emptyset\neq A_n\subseteq\mathcal{A}\setminus\{\emptyset\}$,
\item[($P_2$)] $(\forall n\in\omega)\quad \emptyset\neq\{z\in \mathcal{Z}: (\forall x\in A_n)\quad  x\cap z\neq\emptyset\}$,
\end{enumerate}
\noindent it holds that $\emptyset\neq\bigcap_{n\in\omega}\big\{z\in \mathcal{Z}: (\forall x\in A_n)\quad  x\cap z\neq\emptyset\big\}$;
\item[(iii)] $\mathcal{Z}$ is a $\sigma$-ideal and, for every family $\{A_n: n\in\omega\}$ satisfying conditions ($P_1$) and ($P_2$) of item (ii), it holds that the family 
\[ 
\{\{x\in\mathcal{Z}: (\forall z\in A_n)\quad x\cap z\neq\emptyset\}: n\in\omega\}
\]
has a choice function. 
\end{enumerate}
\end{theorem}
\begin{proof}

(i) $\rightarrow$ (ii) Assume (i). Let $\{A_n: n\in\omega\}$ be a family satisfying conditions ($P_1$) and ($P_2$) of item (ii). For every $n\in\omega$, let
\[
C_n=\{ z\in\mathcal{Z}: (\forall x\in A_n)\quad x\cap z\neq\emptyset\}.
\]

For every $n\in\omega$, put
\[
W_n=\bigcup\big\{ B^{\mathcal{A}}_{\emptyset, z}: z\in C_n\big\}.
\]
The sets $W_n$ are all open in $\mathbf{S}_{\mathcal{A}}(X, \mathcal{Z})$, and $\emptyset\in\bigcap_{n\in\omega}W_n$. Since $\mathbf{S}_{\mathcal{A}}(X, \mathcal{Z})$ is a $P$-space, it follows that there exists $z^{*}\in\mathcal{Z}$ such that 
\[
B^{\mathcal{A}}_{\emptyset, z^{*}}\subseteq\bigcap_{n\in\omega} W_n.
\]
We fix $n\in\omega$ and $x\in A_n$. Suppose that $x\cap z^{*}=\emptyset$. Then $x\in B^{\mathcal{A}}_{\emptyset, z^{*}}\subseteq W_n$, so there exists $z\in C_n$ such that $x\in B^{\mathcal{A}}_{\emptyset,z}$. For such a $z$, we have $x\cap z=\emptyset$, which contradicts the definition of $C_n$. Thus, for every $x\in A_n$, we have $x\cap z^{*}\neq\emptyset$, which shows that $z^{*}\in\bigcap_{n\in\omega}C_n$. Hence (i) implies (ii).
\smallskip

(ii) $\rightarrow$ (i)  Assume that (ii) holds. By Theorem 3.6, it suffices to prove that $\emptyset$ is a $P$-point of $\mathbf{S}_{\mathcal{A}}(X, \mathcal{Z})$. To this aim, consider any family $\{V_n: n\in\omega\}$ of open sets of $\mathbf{S}_{\mathcal{A}}(X, \mathcal{Z})$ such that $\emptyset\in\bigcap_{n\in\omega} V_n$. Without loss of generality, we may assume that $V_0=\mathcal{A}$ and, for every $n\in\omega$, $V_{n+1}\subseteq V_n$. If the sequence $\langle V_n\rangle_{n\in\omega}$ has a finite range, then the set $\bigcap_{n\in\omega}V_n$ is open in $\mathbf{S}_{\mathcal{A}}(X, \mathcal{Z})$. Consider the non-trivial case when the range of $\langle V_n\rangle_{n\in\omega}$ is infinite. For simplicity and without loss of generality, we may assume that, for every $n\in\omega$, $V_{n+1}\neq V_n$. For every $n\in\omega$, we define 
\[
A_n=V_n\setminus V_{n+1}\quad\text{and}\quad  C_n=\{z\in \mathcal{Z}: (\forall x\in A_n)\quad x\cap z\neq\emptyset\}.
\]
 We fix $n_0\in\omega$. Since $\emptyset\in V_{n_0+1}$, and the set $V_{n_0+1}$ is open in $\mathbf{S}_{\mathcal{A}}(X, \mathcal{Z})$, there exists $z_0\in\mathcal{Z}$ such that $B^{\mathcal{A}}_{\emptyset, z_0}\subseteq V_{n_0+1}$. Let $x\in A_{n_0}$. Since $x\notin V_{n_0+1}$, we have $x\notin B^{\mathcal{A}}_{\emptyset, z_0}$, which implies that $x\cap z_0\neq\emptyset$. Therefore, $z_0\in C_{n_0}$. Hence, for every $n\in\omega$, $C_n\neq\emptyset$. By (ii), we can fix $z^{*}\in\bigcap_{n\in\omega}C_n$. Let us prove that $B^{\mathcal{A}}_{\emptyset, z^{*}}\subseteq\bigcap_{n\in\omega}V_n$. 
 
 Suppose that $y\in\mathcal{A}\setminus\bigcap_{n\in\omega}V_n$. Then there exists $n_y\in\omega$ such that $y\in A_{n_y}$. Since $z^{*}\in C_{n_y}$, we have $y\cap z^{*}\neq\emptyset$, which implies that $y\notin B^{\mathcal{A}}_{\emptyset, z^{*}}$. Hence $B^{\mathcal{A}}_{\emptyset,z^{*}}\subseteq\bigcap_{n\in\omega} V_n$, and $\emptyset$ is a $P$-point of $\mathbf{S}_{\mathcal{A}}(X, \mathcal{Z})$. Therefore, conditions (i) and (ii) are equivalent.
 
 (i) $\rightarrow$ (iii) Assume that $\mathbf{S}_{\mathcal{A}}(X, \mathcal{Z})$ is a $P$-space. Then, by Theorem 3.6, $\mathcal{Z}$ is a $\sigma$-ideal. This, together with (ii), gives that (i) implies (iii).
 
 (iii) $\rightarrow$ (ii) Assume that (iii) holds. Let $\{A_n: n\in\omega\}$ be any family satisfying conditions ($P_1$) and ($P_2$) of item (ii). By (iii), there exists a family $\{z_n: n\in\omega\}$ such that, for every $n\in\omega$, $z_n\in\{z\in \mathcal{Z}: (\forall x\in A_n)\quad x\cap z\neq\emptyset\}$. Now, let $z^{*}=\bigcup_{n\in\omega}z_n$. Since $\mathcal{Z}$ is a $\sigma$-ideal, we have $z^{*}\in\mathcal{Z}$. It is obvious that $z^{*}\in\bigcap_{n\in\omega}\big\{z\in\mathcal{Z}: (\forall x\in A_n)\quad x\cap z\neq\emptyset\big\}$. Hence (iii) implies (ii).
\end{proof}

\begin{remark}
\label{s3:r10}
Suppose that $Y$ is an infinite subset of $X$, and $\emptyset\neq\mathcal{A}\subseteq\mathcal{P}(X)$ is such that, for every $x\in \mathcal{A}$, $x\cap Y\in\mathcal{A}$. Let $\mathcal{Z}_Y=\{z\cap Y: z\in\mathcal{Z}\}$ and $\mathcal{A}_Y=\{x\cap Y: x\in\mathcal{A}\}$. Then the space $\mathbf{S}_{\mathcal{A}_Y}(Y, \mathcal{Z}_Y)$ is a topological subspace of $\mathbf{S}_{\mathcal{A}}(X, \mathcal{Z})$. This is a simple consequence of Theorem \ref{s3:t2}(iii) and the fact that, for all $x\in \mathcal{A}_Y$ and $z\in\mathcal{Z}$, the following equality holds: 
\[\{y\in \mathcal{A}_Y: x\subseteq y\wedge y\cap z=\emptyset\}=\mathcal{A}_Y\cap\{y\in \mathcal{A}: x\subseteq y\wedge y\cap z=\emptyset\}.\]

In particular, for every infinite subset $Y$ of $X$, the space $\mathbf{S}(Y, [Y]^{\leq\omega})$ is a topological subspace  of $\mathbf{S}(X, [X]^{\leq\omega})$.
\end{remark}

\begin{remark}
\label{s3:r11}
By mimicking Definition \ref{s3:d1}, one can also define $\mathbf{S}_{\mathcal{P}}(X, \mathcal{Z})$ and $\mathbf{S}(X, \mathcal{Z})$ for a finite set $X$. However, if $X$ is a finite set, then $\mathbf{S}_{\mathcal{P}}(X, \mathcal{Z})=\mathbf{S}(X, \mathcal{Z})$, and the space $\mathbf{S}_{\mathcal{P}}(X, \mathcal{Z})$  is not especially interesting, for it is a finite discrete space.
\end{remark}

\section{An alternative construction}
\label{s4}

Throughout this section, we assume that $X$ is an infinite set, $\mathcal{Z}$ is a family of subsets of $X$ such that $[X]^{<\omega}\subseteq \mathcal{Z}$, and $\mathcal{Z}$ is closed under finite unions. 

In what follows, $2=\{0, 1\}$,  $\mathbf{2}=\langle 2, \mathcal{P}(2)\rangle$, and $\mathbf{2}^X$ stands for the Cantor cube. 

Since the space $\mathbf{S}_{\mathcal{P}}(X, \mathcal{Z})$ is interesting in itself, in this section, we show a relationship between the topology $\tau_{\mathcal{P}}[\mathcal{Z}]$ and a special topology on the product $2^X$, finer than the topology of the Cantor cube $\mathbf{2}^X$. 

\begin{definition}
\label{s4:d1}
\begin{enumerate}
\item[(i)] For every $z\in\mathcal{Z}\setminus\{\emptyset\}$ and $p\in 2^z$, let
\[ [p]_X=\{f\in 2^X: p\subseteq f\}.\]

\item[(ii)] For every $f\in 2^X$, let $\mathcal{B}_2(f)=\{[f\upharpoonright z]_X: z\in\mathcal{Z}\setminus\{\emptyset\}\}$.

\item[(iii)] $\tau_2[\mathcal{Z}]=\{V\subseteq 2^X: (\forall f\in V)(\exists z\in\mathcal{Z}\setminus\{\emptyset\})(\exists p\in 2^z) f\in [p]_X\subseteq V\}.$

\item[(iv)] $2^X[\mathcal{Z}]=\langle 2^X, \tau_2[\mathcal{Z}]\rangle.$
\end{enumerate}
\end{definition}

In the following theorem, we establish several basic properties of the space $2^X[\mathcal{Z}]$. Although their proofs are very simple, we include them for completeness.

\begin{theorem}
\label{s4:t2}
\begin{enumerate}
\item[(i)] $\tau_2[\mathcal{Z}]$ is a topology on $2^X$. 

\item[(ii)] For every $f\in 2^X$, $\mathcal{B}_2(f)$ is a base of neighborhoods of $f$ in $2^X[\mathcal{Z}]$, and all members of $\mathcal{B}_2(f)$ are clopen sets in $2^X[\mathcal{Z}]$.

\item[(iii)] The topology $\tau_2[\mathcal{Z}]$ is finer than the topology of the Cantor cube $\mathbf{2}^X$. Moreover, $\tau_2[\mathcal{Z}]$ is the topology of the Cantor cube $\mathbf{2}^X$ if and only if $\mathcal{Z}=[X]^{<\omega}$.

\item[(iv)] The space $2^{X}[\mathcal{Z}]$ is a zero-dimensional Hausdorff space.

\item[(v)] The space $2^X[\mathcal{Z}]$ is discrete if and only if $X\in\mathcal{Z}$. 

\item[(vi)] The space $2^X[\mathcal{Z}]$ is crowded if and only if $X\notin\mathcal{Z}$.
\end{enumerate}
\end{theorem}

\begin{proof}
(i) Let $f\in 2^X$ and $x\in X$. Then $\{x\}\in\mathcal{Z}$ and $[f\upharpoonright\{x\}]_X\in\mathcal{B}_2(f)$. Hence $\mathcal{B}_2(f)\neq\emptyset$. Of course, for every $z\in\mathcal{Z}\setminus\{\emptyset\}$, we have $f\in [f\upharpoonright z]_X$.

Let $U_1, U_2\in\mathcal{B}_2(f)$. There exist $z_1, z_2\in\mathcal{Z}\setminus\{\emptyset\}$ such that $U_1=[f\upharpoonright z_1]_X$ and $U_2=[f\upharpoonright z_2]_X$. For $z=z_1\cup z_2$, we have $[f\upharpoonright z]_X\subseteq U_1\cap U_2$.

All this taken together shows that $\tau_2[\mathcal{Z}]$ is a topology on $2^X$ generated by the family $\bigcup\{\mathcal{B}_2(f): f\in 2^X\}$.
\smallskip
 
(ii) Let us fix $f\in 2^X$ and $z\in\mathcal{Z}\setminus\{\emptyset\}$. We notice that if $g\in [f\upharpoonright z]_X$, then $[g\upharpoonright z]_X=[f\upharpoonright z]_X$. This shows that $[f\upharpoonright z]_X\in\tau_2[\mathcal{Z}]$. Therefore, $\mathcal{B}_2(f)$ is a base of neighborhooods of $f$ in $2^X[\mathcal{Z}]$.

To show that $[f\upharpoonright z]_X$ is closed in $2^X[\mathcal{Z}]$, it suffices to notice that, for every $g\in 2^X\setminus [f\upharpoonright z]_X$, we have $[g\upharpoonright z]_X\cap [f\upharpoonright z]_X=\emptyset$. 
\smallskip

(iii) Since $[X]^{<\omega}\subseteq\mathcal{Z}$, that (iii) holds follows directly from (ii) and the standard definition of the topology of the Cantor cube $\mathbf{2}^X$.
\smallskip

(iv) It follows from (ii) that $2^X[\mathcal{Z}]$ is zero-dimensional. Since the Cantor cube $\mathbf{2}^X$ is a Hausdorff space, the space $2^X[\mathcal{Z}]$ is also Hausdorff by (iii). 
\smallskip

(v) If $X\in\mathcal{Z}$, then, for every $f\in 2^X$, we have $\{f\}=[f\upharpoonright X]_X\in\mathcal{B}_2(f)$, which, by (ii), implies that $2^X[\mathcal{Z}]$ is discrete.

Now, suppose that the space $2^X[\mathcal{Z}]$ is discrete, and consider any $f\in 2^X$. Since $f$ is an isolated point of $2^X[\mathcal{Z}]$, there exists $z_f\in\mathcal{Z}\setminus\{\emptyset\}$ such that $\{f\}=[f\upharpoonright z_f]_X$. This is possible only when $z_f=X$. Hence (v) holds.
\smallskip

(vi) If the space $2^X[\mathcal{Z}]$ is crowded, then we deduce from (v) that $X\notin\mathcal{Z}$. 

Now, let us assume that $X\notin\mathcal{Z}$, and prove that $2^{X}[\mathcal{Z}]$ is crowded. To this aim, consider any $z\in\mathcal{Z}\setminus\{\emptyset\}$ and $f\in 2^X$. Since $z\neq X$, we can fix $x_0\in X\setminus z$. We define a function $g\in 2^X$ as follows: $g(x_0)=1-f(x_0)$ and, for every $x\in X\setminus\{x_0\}$, $g(x)=f(x)$. Then $g\in [f\upharpoonright z]_X$ and $g\neq f$. This, together with (ii), completes the proof of (vi) and of the theorem.
\end{proof}

Our next theorem shows a connection between $\mathbf{S}_{\mathcal{P}}(X, \mathcal{Z})$ and $2^X[\mathcal{Z}]$.

\begin{theorem}
\label{s4:t3}
Let $\mathcal{Z}$ be a bornology on $X$. Suppose that $\mathcal{A}$ is a subfamily of $\mathcal{P}(X)$ such that $[X]^{<\omega}\subseteq\mathcal{A}\subseteq\mathcal{Z}$. For every $x\in \mathcal{P}(X)$, let $\chi_x\in 2^X$ be the indicator function defined by: $\chi_x[x]\subseteq\{1\}$ and $\chi_x[X\setminus x]\subseteq\{0\}$. Let $\psi: \mathcal{P}(X)\to 2^X$ be defined by: 
\[(\forall x\in \mathcal{P}(X)) \psi(x)=\chi_x.\]
Then the following conditions are satisfied.
\begin{enumerate}
\item[(i)] $\psi$ is a $\langle \tau_{\mathcal{P}}[\mathcal{Z}], \tau_2[\mathcal{Z}]\rangle$-continuous bijection.

\item[(ii)] $\psi$ is a $\langle \tau_{\mathcal{P}}[\mathcal{Z}], \tau_2[\mathcal{Z}]\rangle$-homeomorphism if and only if $X\in\mathcal{Z}$. The spaces $\mathbf{S}_{\mathcal{P}}(X, \mathcal{Z})$ and $2^X[\mathcal{Z}]$ are homeomorphic if and only if $X\in\mathcal{Z}$.

\item[(iii)] $\psi\upharpoonright \mathcal{A}$ is a homeomorphic embedding of $\mathbf{S}_{\mathcal{A}}(X, \mathcal{Z})$ onto the subspace 
$$\mathbf{S}_{\mathcal{A}}=\{f\in 2^X: f^{-1}[\{1\}]\in\mathcal{A}\}$$
\noindent of $2^X[\mathcal{Z}]$.

\item[(iv)] $\psi\upharpoonright [X]^{<\omega}$ is a homeomorphic embedding of $\mathbf{S}(X, \mathcal{Z})$ onto the subspace 
$$\mathbf{S}=\{f\in 2^X: f^{-1}[\{1\}]\in [X]^{<\omega}\}$$
\noindent of $2^X[\mathcal{Z}]$.

\item[(v)] If $X$ contains no infinite Dedekind-finite sets and $[X]^{\leq\omega}\subseteq\mathcal{Z}$, then $\mathbf{S}$ is a closed subspace of $2^X[\mathcal{Z}]$.

\item[(vi)] If $X$ contains no infinite quasi Dedekind-finite sets and $\mathcal{Z}$ is a $\sigma$-ideal, then $\mathbf{S}$ is a closed subspace of $2^X[\mathcal{Z}]$.
\end{enumerate}
\end{theorem}

\begin{proof}
(i) It is obvious that $\psi$ is a bijection.  To show that $\psi$ is $\langle \tau_{\mathcal{P}}[\mathcal{Z}], \tau_2[\mathcal{Z}]\rangle$-continuous, we consider any $t_0\in \mathcal{P}(X)$ and any $V\in\tau_2[\mathcal{Z}]$ with $\psi(t_0)\in V$. There exist $z_0\in\mathcal{Z}\setminus\{\emptyset\}$ and $p_0\in 2^{z_0}$, such that $\psi(t_0)\in [p_0]_X\subseteq V$. Since $\mathcal{Z}$ is an ideal, we have $z_0\setminus t_0\in\mathcal{Z}(t_0)$. We notice that $\psi[B_{t_0, z_0\setminus t_0}^{\mathcal{P}}]\subseteq [p_0]_X$. Indeed, let us assume that $y\in B_{t_0, z_0\setminus t_0}^{\mathcal{P}}$. We want to show that $\psi(y)\in [p_0]_X$. To this aim, we consider any $t\in z_0$. We have $t_0\subseteq y$ and $y\cap (z_0\setminus t_0)=\emptyset$. If $t\in t_0\cap z_0$, then $p_0(t)=\chi_{t_0}(t)=1$ and $\psi(y)(t)=\chi_y(t)=1$. If $t\in z_0\setminus t_0$, then $p_0(t)=\psi(t_0)(t)=\chi_{t_0}(t)=0$, and $\psi(y)(t)=\chi_y(t)=0$. Therefore, for every $y\in B_{t_0, z_0\setminus t_0}$, we have $\psi(y)\in [p_0]_X$. This completes the proof that $\psi$ is $\langle \tau_{\mathcal{P}}[\mathcal{Z}], \tau_2[\mathcal{Z}]\rangle$-continuous. 
\smallskip

(ii) First, we recall that, by Theorem \ref{s3:t2}(vi), $X$ is an isolated point of $\mathbf{S}_{\mathcal{P}}(X, \mathcal{Z})$. 

Suppose that $X\notin\mathcal{Z}$. Then, by Theorem \ref{s4:t2}(v), the space $2^X[\mathcal{Z}]$ is crowded, so it cannot be homeomorphic with a space having an isolated point.

Now, suppose that $X\in\mathcal{Z}$.  Then, by Theorem \ref{s4:t2}(v), the space $2^X[\mathcal{Z}]$ is discrete. Since $\mathcal{Z}$ is a bornology with $X\in\mathcal{Z}$ we have $\mathcal{Z}=\mathcal{P}(X)$. By Theorem \ref{s3:t2}(xi), the space $\mathbf{S}_{\mathcal{P}}(X, \mathcal{Z})$ is also discrete. This implies that $\psi$ is a homeomorphism of $\mathbf{S}_{\mathcal{P}}(X, \mathcal{Z})$ onto $2^X[\mathcal{Z}]$. Hence (ii) holds.
\smallskip

(iii) Let us show that, for every $U\in\tau_{\mathcal{A}}[\mathcal{Z}]$, $\psi[U]$ is open in the subspace $\mathbf{S}_{\mathcal{A}}$ of the space $2^X[\mathcal{Z}]$. To this aim, consider any $x\in \mathcal{A}$ and any $z\in\mathcal{Z}(x)$. It suffices to show that $\psi[B_{x,z}^{\mathcal{A}}]$ is open in the subspace $\psi[\mathcal{A}]$ of $2^X[\mathcal{Z}]$. 

 Let $g\in\psi[B_{x,z}^{\mathcal{A}}]$ and $y_g=g^{-1}[\{1\}]$. Then $\psi(y_g)=g$, so  $y_g\in B_{x, z}^{\mathcal{A}}$. We have $x\subseteq y_g$ and $y_g\cap z=\emptyset$. Let $z_g=y_g\cup z$. Since $y_g\in\mathcal{A}$, we have $y_g\in\mathcal{Z}$. Then $z_g\in\mathcal{Z}$. If $z_g=\emptyset$, then $B_{x,z}^{\mathcal{A}}=\mathcal{A}$. 
 
Assume that $z_g\neq\emptyset$. Consider any $f\in [g\upharpoonright z_g]_X\cap\psi[\mathcal{A}]$. Let $y_f=f^{-1}[\{1\}]$. Then $y_f\in \mathcal{A}$ and $f=\psi(y_f)$. Since $x\subseteq y_g$ and $f\in [g\upharpoonright y_g]_X$, we have $x\subseteq y_f$. Since $y_g\cap z=\emptyset$ and $f\in [g\upharpoonright z]_X$, we have $y_f\cap z=\emptyset$. Therefore, $y_f\in B_{x,z}^{\mathcal{A}}$, so $f\in\psi[B_{x,z}^{\mathcal{A}}]$. This completes the proof of (iii). 
\smallskip

(iv) This follows from (iii).
\smallskip

For the proofs of (v) and (vi), let us fix any $f\in 2^X[\mathcal{Z}]\setminus \mathbf{S}$. Then the set $F=f^{-1}[\{1\}]$ is infinite.

(v) Assume that every infinite subset of $X$ is Dedekind-infinite, and $[X]^{\leq\omega}\subseteq\mathcal{Z}$.  Then the set $F$ is a Dedekind-infinite subset of $X$. Hence, there exists a denumerable subset $C$ of $F$. Since $[X]^{\leq\omega}\subseteq\mathcal{Z}$, we have $C\in\mathcal{Z}$. We notice that $[f\upharpoonright C]_X\cap\mathbf{S}=\emptyset$ and, therefore,  $\mathbf{S}$ is a closed subspace of $2^X[\mathcal{Z}]$.
\smallskip

(vi) Assuming that every infinite subset of $X$ is quasi Dedekind-infinite, and $\mathcal{Z}$ is a $\sigma$-ideal, we can fix a denumerable subset $E$ of $[F]^{<\omega}$.  Since $E\subseteq\mathcal{Z}$ and $\mathcal{Z}$ is a $\sigma$-ideal, we have $\bigcup E\in\mathcal{Z}$. The set $\bigcup E$ is infinite, which implies that $[f\upharpoonright \bigcup E]_X\cap\mathbf{S}=\emptyset$. This completes the proof.  
\end{proof}

\begin{remark}
\label{s4:r4}
In general, the space $\mathbf{S}(X, \mathcal{Z})$ need not be homeomorphic with a subspace of $2^{X}[\mathcal{Z}]$, and even if $2^{X}[\mathcal{Z}]$ is a $P$-space, the space $\mathbf{S}(X, \mathcal{Z})$ need not be a $P$-space. To see this, let us fix a finite subset $c_0$ of the infinite set $X$. As in Example \ref{s3:e8}, consider the family $\mathcal{Z}=[X]^{<\omega}\cup\{(X\setminus c_0)\cup c: c\in [X]^{<\omega}\}$. Since $X\in\mathcal{Z}$, it follows from Theorem \ref{s4:t2}(v) that the space $2^X[\mathcal{Z}]$ is discrete, so it is a $P$-space and its every subspace is also discrete. We know from Example \ref{s3:e8} that the space $\mathbf{S}(X, \mathcal{Z})$ is not discrete, so it cannot be homeomorphically embedded into $2^X[\mathcal{Z}]$. In addition, it is proved in Example \ref{s3:e8} that if the set $X$ is quasi Dedekind-infinite, the space $\mathbf{S}(X, \mathcal{Z})$ is not a $P$-space and, therefore, $\mathbf{S}_{\mathcal{P}}(X, \mathcal{Z})$ is not a $P$-space.
\end{remark}
 
The following simple proposition will be applied several times in this paper. 
 
\begin{proposition}
\label{s6:p6}
For every set $X$, the following conditions are satisfied:
\begin{enumerate}
\item[(i)] $X$ is Dedekind-finite if and only if $[X]^{<\omega}=[X]^{\leq\omega}$.
\item[(ii)] $X$ is quasi Dedekind-finite if and only if $[X]^{<\omega}=[X]^{\leq\omega}$ and $[X]^{\leq\omega}$ is a $\sigma$-ideal.
\end{enumerate}
In consequence, a Dedekind-finite set $X$ is quasi Dedekind-finite if and only if $[X]^{\leq\omega}$ is a $\sigma$-ideal.
\end{proposition}
\begin{proof}
It is obvious that (i) holds. Since every quasi Dedekind-finite set is Dedekind-finite, it follows from (i) and the definition of a quasi Dedekind-finite set that (ii) holds.
\end{proof}

\begin{remark}
\label{s6:r7}
It is known that a Dedekind-finite set which is not quasi Dedekind-finite exists, for instance, in Cohen's Second Model (see \cite[Model $\mathcal{M}7$]{hr} and \cite[Section 5.4]{j}). This, together with Proposition \ref{s6:p6}, shows that it unprovable in $\mathbf{ZF}$ that, for every infinite set $X$, the equality $[X]^{<\omega}=[X]^{\leq\omega}$ implies that $[X]^{\leq\omega}$ is a $\sigma$-ideal.
\end{remark}

In view of Theorem \ref{s4:t3}(v), it is worthwhile to state the following theorem of $\mathbf{ZF}$.

\begin{theorem}
\label{s4:t5}
\begin{enumerate}
\item $\mathbf{IDI}$ is equivalent to the following statement: ``For every infinite set $X$ and every bornology $\mathcal{Z}$ on $X$ such that $[X]^{\leq\omega}\subseteq\mathcal{Z}$, the set $\{f\in 2^X: f^{-1}[\{1\}]\in [X]^{<\omega}\}$ is closed in $2^X[\mathcal{Z}]$''.
\item $\mathbf{IQDI}$ is equivalent to the following statement: ``For every infinite set $X$ and every $\sigma$-ideal $\mathcal{Z}$ of subsets of $X$ with $[X]^{<\omega}\subseteq \mathcal{Z}$,  the set $\{f\in 2^X: f^{-1}[\{1\}]\in [X]^{<\omega}\}$ is closed in $2^X[\mathcal{Z}]$''.
\end{enumerate}
\end{theorem}
\begin{proof}
For an infinite set $X$, let $\mathbf{S}=\{f\in 2^X[\mathcal{Z}]: f^{-1}[\{1\}]\in [X]^{<\omega}\}$ and let $f_1\in 2^X$ be the constant function with $f_1[X]=\{1\}$. Then $f_1\in 2^X\setminus \mathbf{S}$. If $\mathcal{Z}=[X]^{<\omega}$, then, for every $z\in\mathcal{Z}\setminus\{\emptyset\}$, we have $[f_1\upharpoonright z]\cap\mathbf{S}\neq\emptyset$, so $\mathbf{S}$ is not closed in $2^X[\mathcal{Z}]$. In particular, it follows from Proposition \ref{s6:p6} that if $X$ is an infinite (quasi) Dedekind-finite set, and $\mathcal{Z}=[X]^{\leq\omega}$, then $\mathbf{S}$ is not closed in $2^X[\mathcal{Z}]$. Moreover, if $X$ is quasi Dedekind-finite, then $[X]^{\leq\omega}$ is a $\sigma$-ideal. To complete the proofs of (1), it suffices to apply Theorem \ref{s4:t3}(v), and to conclude that (2) holds, it suffices to apply Theorem \ref{s4:t3}(vi).
\end{proof}

In the light of Theorem \ref{s4:t2}(iii), the following theorem about the homogeneity of $2^{X}[\mathcal{Z}]$ is stronger than the well-known fact that Cantor cubes are homogeneous. Some of the newest results on the homogeneity of Cantor cubes are given in \cite{ShVl}.

\begin{theorem}
\label{s4:t6}
The space $2^X[\mathcal{Z}]$ is homogeneous.
\end{theorem}
\begin{proof}
Let $f_0\in 2^X$ be the constant function with $f_0[X]=\{0\}$. We fix a function $g_0\in 2^X$ with $f_0\neq g_0$. We fix $x_0\in X$ such that $g_0(x_0)=1$. Let $z_0=\{x_0\}$. Then $z_0\in\mathcal{Z}$. Let $U_0=[f_0\upharpoonright z_0]_X$ and $V_0=[g_0\upharpoonright z_0]_X$. By Theorem \ref{s4:t2}(ii), the sets $U_0$ and $V_0$ are both clopen in $2^X[\mathcal{Z}]$. Clearly, $U_0\cap V_0=\emptyset$ and $U_0\cup V_0= 2^X$.  For an arbitrary $f\in 2^X$, we define  the function $f^{\ast}\in 2^{X}[\mathcal{Z}]$ as follows:
\[
(\forall x\in X)\quad f^{\ast}(x)=\begin{cases} 1-f(x) &\text{ if } g_0(x)=1;\\
f(x) &\text{ if } g_0(x)=0.\end{cases}
\]
\noindent One can easily check that, for every $f\in 2^X$, we have  $f^{\ast\ast}=f$. 
 We define the function $h: 2^{X}\to 2^X$ as follows:
 \[
 (\forall f\in 2^X)\quad h(f)= f^{\ast}.
\]
\noindent We observe that, for every $f\in 2^X$, $h\circ h(f)=f$. This implies that $h$ is a bijection such that $h^{-1}=h$. Since $h(f_0)=g_0$, to complete the proof, it remains to show that $h$ is $\langle \tau_2[\mathcal{Z}], \tau_2[\mathcal{Z}]\rangle$-continuous. 

Let $f\in U_0$. Then $h(f)=f^{\ast}\in V_0$.  We fix an arbitrary  non-empty set $z\in\mathcal{Z}$ such that $[f^{\ast}\upharpoonright z]_X\subseteq V_0$. We put $z_1=z\cup z_0$. Then $z_1\in\mathcal{Z}$ and $[f^{\ast}\upharpoonright z_1]_X\subseteq [f^{\ast}\upharpoonright z]_X$. We notice that $[f\upharpoonright z_1]_X\subseteq U_0$. Furthermore, for every $g\in [f\upharpoonright z_1]_X$, we have $h(g)=g^{\ast}\in [f^{\ast}\upharpoonright z_1]_X$. This proves that $h$ is $\langle \tau_2[\mathcal{Z}], \tau_2[\mathcal{Z}]\rangle$-continuous at every point of $U_0$. Using similar arguments, one can show that $h$ is $\langle \tau_2[\mathcal{Z}], \tau_2[\mathcal{Z}]\rangle$-continuous at every point of $V_0$. This completes the proof.
\end{proof}

\begin{remark}
\label{s4:r7}
Item  (b) of Theorem \ref{s3:t5} has an alternative proof. Namely, let us assume that $\mathcal{Z}$ and $\mathcal{A}$ are bornologies on $X$ such that $\mathcal{A}\subseteq\mathcal{Z}$. By Theorem \ref{s4:t3}. the space $\mathcal{S}_{\mathcal{A}}(X,\mathcal{Z})$ is homeomorphic with the subspace $\mathbf{S}_{\mathcal{A}}=\{f\in 2^X: f^{-1}[\{1\}]\in \mathcal{A}\}$ of $2^X[\mathcal{Z}]$. Let $f_0\in 2^X$ be the constant function with $f_0[X]=\{0\}$. Then $f_0\in \mathbf{S}_{\mathcal{A}}$. We fix $g_0\in \mathbf{S}_{\mathcal{A}}$ with $f_0\neq g_0$. The proof of Theorem \ref{s4:t6} shows that there exists an autohomeomorhism $h$ of $2^X[\mathcal{Z}]$ such that $h(f_0)=g_0$ and $h[\mathbf{S}_{\mathcal{A}}]=\mathbf{S}_{\mathcal{A}}$. This implies that $\mathbf{S}_{\mathcal{A}}$ is homogeneous. However, the direct proof of the homogeneity of $\mathbf{S}_{\mathcal{A}}(X, \mathcal{Z})$, given in Section \ref{s3}, is simpler.
\end{remark} 

Now, we are going to give a modification of Theorem \ref{s3:t6} for the space $2^X[\mathcal{Z}]$.

\begin{theorem}
\label{s4:t8}
\begin{enumerate}
\item[(i)] The following conditions are equivalent:
\begin{enumerate}
\item[(a)] $2^{X}[\mathcal{Z}]$ is a $P$-space;
\item[(b)] the space $2^X[\mathcal{Z}]$ has a $P$-point;
\item[(c)] the constant function $f_0\in 2^X$ such that $f_0[X]=\{0\}$ is a $P$-point of $2^X[\mathcal{Z}]$.
\end{enumerate}

\item[(ii)] If $\mathcal{Z}$ is a bornology on $X$, then  $\mathbf{S}_{\mathcal{P}}(X, \mathcal{Z})$ is a $P$-space if and only if $2^X[\mathcal{Z}]$ is a $P$-space.

\item[(iii)] If $\mathcal{Z}$ is a bornology on $X$ such that $2^X[\mathcal{Z}]$ is a $P$-space, then $\mathbf{S}(X, \mathcal{Z})$ is a $P$-space.

\item[(iv)] If $\mathcal{Z}$ is a bornology on $X$ such that $2^X[\mathcal{Z}]$ is a $P$-space, then $\mathcal{Z}$ is a $\sigma$-ideal.

\item[(v)] If $X\in\mathcal{Z}$, then $2^X[\mathcal{Z}]$ is a $P$-space, but $\mathbf{S}_{\mathcal{P}}(X, \mathcal{Z})$ may fail to be a $P$-space.  
\end{enumerate}
\end{theorem}

\begin{proof}
(i) This follows immediately from Theorem \ref{s4:t6}.
\smallskip

(ii) Assuming that $\mathcal{Z}$ is a bornology on $X$, we consider the $\langle\tau_{\mathcal{P}}[\mathcal{Z}], \tau_{2}[\mathcal{Z}]\rangle$-continuous bijection $\psi: \mathcal{P}(X)\to 2^X$ defined in Theorem \ref{s4:t3}. As in (c), let $f_0\in 2^X$ be the constant function with $f_0[X]=\{0\}$.

Suppose that $\mathbf{S}_{\mathcal{P}}(X, \mathcal{Z})$ is a $P$-space. Let $\{W_n: n\in\omega\}$ be a family of open sets in $2^X[\mathcal{Z}]$ such that $f_0\in\bigcap_{n\in\omega}W_n$. Then $\emptyset\in\bigcap_{n\in\omega}\psi^{-1}[W_n]$ and, for every $n\in\omega$, $\psi^{-1}[W_n]$ is an open set in $\mathbf{S}_{\mathcal{P}}(X, \mathcal{Z})$. Since $\emptyset$ is a $P$-point of $\mathbf{S}_{\mathcal{P}}(X, \mathcal{Z})$, there exists $z_0\in\mathcal{Z}$ such that $B_{\emptyset, z_0}^{\mathcal{P}}\subseteq\bigcap_{n\in\omega}\psi^{-1}[W_n]$. We may assume that $z_0\neq\emptyset$. Suppose that $g\in [f_0\upharpoonright z_0]_X$. Then $g^{-1}[\{1\}]\cap z_0=\emptyset$, so $g^{-1}[\{1\}]\in B_{\emptyset, z_0}^{\mathcal{P}}$. This implies that $g^{-1}[\{1\}]\in\bigcap_{n\in\omega}\psi^{-1}[W_n]$ and, therefore, $g=\psi(g^{-1}[\{1\}])\in\bigcap_{n\in\omega}W_n$. Hence $[f_0\upharpoonright z_0]_X\subseteq\bigcap_{n\in\omega}W_n$, so $f_0$ is a $P$-point of $2^X[\mathcal{Z}]$. By (i), $2^X[\mathcal{Z}]$ is a $P$-space.

Now, assume that $2^X[\mathcal{Z}]$ is a $P$-space. By Theorem \ref{s3:t6}(iv), to show that $\mathbf{S}_{\mathcal{P}}(X, \mathcal{Z})$ is s $P$-space, it suffices to prove that $\emptyset$ is a $P$-point of $\mathbf{S}_{\mathcal{P}}(X, \mathcal{Z})$. To this end, we fix a family $\{G_n: n\in\omega\}$ of open sets of $\mathbf{S}_{\mathcal{P}}(X, \mathcal{Z})$ such that $\emptyset\in\bigcap_{n\in\omega}G_n$. For every $n\in\omega$, let
\[
 A_n=\{z\in\mathcal{Z}\setminus\{\emptyset\}: B^{\mathcal{P}}_{\emptyset, z}\subseteq G_n\}\quad\text{and}\quad H_n=\bigcup\Big\{[f_0\upharpoonright z]_X: z\in A_n\Big\}.
\]
Clearly, for every $n\in\omega$, $A_n\neq\emptyset$, the set $H_n$ is open in $2^X[\mathcal{Z}]$, and $f_0\in H_n$. Since $f_0\in\bigcap_{n\in\omega}H_n$, and $2^X[\mathcal{Z}]$ is a $P$-space, there exists $z_1\in\mathcal{Z}\setminus\{\emptyset\}$ such that $[f_0\upharpoonright z_1]_X\subseteq \bigcap_{n\in\omega}H_n$. We fix $n\in\omega$ and consider any $y\in B^{\mathcal{P}}_{\emptyset, z_1}$. Let $f=\psi(y)$. Then $f[y]=\{1\}$ and $f[X\setminus y]=\{0\}$. Since $y\cap z_1=\emptyset$, we have $f\in [f_0\upharpoonright z_1]_X\subseteq H_n$. Hence $f\in H_n$. This implies that there exists $z_2\in A_n$ such that $f\in [f_0\upharpoonright z_2]_X$. Then $f[z_2]\subseteq\{0\}$, so $z_2\subseteq X\setminus y$. Therefore, $y\in B^{\mathcal{P}}_{\emptyset, z_2}\subseteq G_n$. This proves that $B^{\mathcal{P}}_{\emptyset, z_1}\subseteq\bigcap_{n\in\omega}G_n$, so $\emptyset$ is a $P$-point of $\mathbf{S}_{\mathcal{P}}(X, \mathcal{Z})$. Hence (ii) holds.
\smallskip

(iii) Since, by Theorem \ref{s3:t2}(iv), $\mathbf{S}(X, \mathcal{Z})$ is a topological subspace of $\mathbf{S}_{\mathcal{P}}(X, \mathcal{Z})$, and  every topological subspace of a $P$-space is a $P$-space, it follows from (ii) that (iii) holds.
\smallskip

(iv) This follows from (iii) and Theorem  \ref{s3:t6}(i).
\smallskip

(v) If $X\in\mathcal{Z}$, then, by Theorem \ref{s4:t2}(v), the space $2^X[\mathcal{Z}]$ is discrete, so also a $P$-space. It has been noticed in Example \ref{s3:e8} and Remark \ref{s4:r4} that if $X$ is a quasi Dedekind-infinite set, then there exists a family $\mathcal{Z}$ of subsets of $X$ such that $[X]^{<\omega}\cup\{X\}\subseteq\mathcal{Z}$, $\mathcal{Z}$ is closed under finite unions, but $\mathbf{S}_{\mathcal{P}}(X, \mathcal{Z})$ is not a $P$-space.
\end{proof}

In the forthcoming Theorem \ref{s6:t4}(e), we will show that it is relatively consistent with $\mathbf{ZF}$ that there exists an infinite set $X$ such that $\mathbf{S}(X, [X]^{\leq\omega})$ is a $P$-space, but the spaces $2^{X}[[X]^{\leq\omega}]$ and $\mathbf{S}_{\mathcal{P}}(X, [X]^{\leq\omega})$ are not $P$-spaces.
 
\begin{remark}
\label{s4:r9}
Item (iv) of Theorem \ref{s4:t8} also has the following alternative direct proof.

Suppose that $\mathcal{Z}$ is a bornology on $X$, and $2^X[\mathcal{Z}]$ is a $P$-space. By way of contradiction, we assume that $\mathcal{Z}$ is not a $\sigma$-ideal. Then there exists a family $\{t_n: n\in\omega\}$ of non-empty members of $\mathcal{Z}$ such that $\bigcup_{n\in\omega}t_n\notin\mathcal{Z}$. Let $f_0\in 2^X$ be the constant function with $f_0[X]=\{0\}$. For every $n\in\omega$, let $U_n=[f_0\upharpoonright t_n]_X$. Then, for every $n\in\omega$, $U_n\in\tau_2[\mathcal{Z}]$. Let $U=\bigcap_{n\in\omega}U_n$. Clearly, $f_0\in U$. Since $2^X[\mathcal{Z}]$ is a $P$-space, it follows from Theorem \ref{s4:t2}(ii) that there exists $z\in\mathcal{Z}\setminus\{\emptyset\}$ such that $[f_0\upharpoonright z]_X\subseteq U$. Let us fix $n\in\omega$. Suppose that $t_n$ is not a subset of $z$. Then we can fix $t\in t_n\setminus z$, and define a function $h_t\in 2^X$ as follows: $h_t(t)=1$ and, for every $x\in X\setminus\{t\}$, $h_t(x)=0$. Since $h_t\in [f_0\upharpoonright z]_X$, we deduce that $h_t\in U_n$. This contradicts the fact that $h_t\notin U_n$. Therefore, for every $n\in\omega$, $t_n\subseteq z$. Hence $\bigcup_{n\in\omega}t_n\subseteq z$. Since $\mathcal{Z}$ is an ideal and $z\in\mathcal{Z}$, we obtain that $\bigcup_{n\in\omega}t_n\in\mathcal{Z}$, which is a contradiction.
\end{remark}

\begin{theorem}
\label{s4:t10}
Each of the following statements implies the one beneath it:
\begin{enumerate}
\item[(a)] $\mathbf{CMC}(\aleph_0, \infty)$;

\item[(b)] for every infinite set $X$ and every bornology $\mathcal{Z}$ on $X$, it holds that if $\mathcal{Z}$ is a $\sigma$-ideal, then $2^X[\mathcal{Z}]$ and $\mathbf{S}_{\mathcal{P}}(X, \mathcal{Z})$ are both $P$-spaces;

\item[(c)] for every infinite set $X$ and every bornology $\mathcal{Z}$ on $X$, it holds that if $\mathcal{Z}$ is a $\sigma$-ideal, then  $\mathbf{S}(X, \mathcal{Z})$ is a $P$-space.
\end{enumerate}
\end{theorem}

\begin{proof} (a) $\rightarrow$ (b) Assuming $\mathbf{CMC}(\aleph_0, \infty)$, consider any infinite set $X$ and a $\sigma$-ideal $\mathcal{Z}$ of subsets of $X$ such that $[X]^{<\omega}\subseteq\mathcal{Z}$. In view of Theorem \ref{s4:t8}(ii), to prove (a), it suffices to check that $\mathbf{S}_{\mathcal{P}}(X, \mathcal{Z})$ is a $P$-space. To this aim, we fix a family  $\{O_n: n\in\omega\}$ of open sets in $\mathbf{S}_{\mathcal{P}}(X, \mathcal{Z})$ with $\bigcap_{n\in\omega}O_n\neq\emptyset$. Consider any $x_0\in\bigcap_{n\in\omega}O_n$. For every $n\in\omega$, let $A_n=\{z\in\mathcal{Z}(x_0): B_{x_0, z}^{\mathcal{P}}\subseteq O_n\}$. We notice that, for every $n\in\omega$, $A_n\neq\emptyset$. By $\mathbf{CMC}(\aleph_0, \infty)$, there exists a family $\{B_n: n\in\omega\}$ of non-empty countable sets such that, for every $n\in\omega$, $B_n\subseteq A_n$. Let $z^{\ast}=\bigcup_{n\in\omega}(\bigcup B_n)$. Since $\mathcal{Z}$ is a $\sigma$-ideal, we have ${z}^{\ast}\in\mathcal{Z}$. Clearly, $z^{\ast}\in \mathcal{Z}(x_0)$ and $B_{x_0, z_0}^{\mathcal{P}}\subseteq\bigcap_{n\in\omega}O_n$. Hence (a) implies (b).

(b) $\rightarrow$ (c) It follows from Theorem \ref{s4:t8}(iii) that (b) implies (c).
\end{proof}

\begin{remark}
\label{s4:r11}
It is easy to apply Theorem \ref{s3:t9} to get an alternative  proof that (a) implies (b) in Theorem \ref{s4:t10}. Indeed, assuming  $\mathbf{CMC}(\aleph_0, \infty)$, we fix an infinite set $X$ and a $\sigma$-ideal $\mathcal{Z}$ of subsets of $X$ such that $[X]^{<\omega}\subseteq\mathcal{Z}$. We fix a family $\{A_n: n\in\omega\}$ satisfying conditions ($P_1$)--($P_2$) from item (ii) of Theorem \ref{s3:t9}. For every $n\in\omega$, let $C_n=\{z\in\mathcal{Z}: (\forall x\in A_n)\quad x\cap z\neq\emptyset\}$. By  $\mathbf{CMC}(\aleph_0, \infty)$, there exists a family $\{Z_n: n\in\omega\}$ of non-empty countable sets such that, for every $n\in\omega$, $Z_n\subseteq C_n$. Let $z^{*}=\bigcup_{n\in\omega}(\bigcup Z_n)$. Since $\mathcal{Z}$ is a $\sigma$-ideal, we have $z^{*}\in\mathcal{Z}$. Clearly, for all $n\in\omega$ and $x\in A_n$, we have $z^{*}\cap x\neq\emptyset$. Hence $z^{*}\in\bigcap_{n\in\omega}C_n$. By Theorem \ref{s3:t9}, $\mathbf{S}_{\mathcal{P}}(X, \mathcal{Z})$ is a $P$-space.
\end{remark}

The following problems are interesting.
\begin{problem}
\label{s4:q12} 
\begin{enumerate}
\item Is any of the implications of Theorem \ref{s4:t10} reversible in $\mathbf{ZF}$?
\item Does the statement ``for every uncountable set $X$, if $[X]^{\leq\omega}$ is a $\sigma$-ideal, then $\mathbf{S}(X, [X]^{\leq\omega})$ is a $P$-space'' imply $\mathbf{CMC}(\aleph_0, \infty)$ in $\mathbf{ZF}$?
\end{enumerate}
\end{problem}

The forthcoming Theorem \ref{thm:Models_M1_N3}(3) will give \textbf{a strong negative answer to Problem \ref{s4:q12}(2)}. Theorems \ref{s3:t6}, \ref{s4:t8} and \ref{s4:t10} give only partial answers to the intriguing questions posed in the following problem.

\begin{problem}
\label{s4:q13} 
\begin{enumerate}
\item[(i)] If $\mathcal{Z}$ is a proper $\sigma$-ideal of subsets of $X$ with $[X]^{<\omega}\subseteq\mathcal{Z}$, when is $\mathbf{S}(X, \mathcal{Z})$ a $P$-space?
\item[(ii)] When is $2^X[\mathcal{Z}]$ a $P$-space?
\end{enumerate}
\end{problem} 

\noindent Other partial answers to these questions will be given in the forthcoming sections. Their applications to several independence results will also be shown. Although Theorem \ref{s3:t9} delivers necessary and sufficient conditions for a set $X$ and a bornology $\mathcal{Z}$ on $X$ to be such that $\mathbf{S}(X, \mathcal{Z})$ is a $P$-space, sometimes it may be hard to verify if conditions (ii) or (iii) of Theorem \ref{s3:t9} are satisfied. 

\begin{remark}
\label{s4:r14}
Suppose that $Y$ is an infinite proper subset of $X$, and $\mathcal{Z}_Y=\{z\cap Y: z\in\mathcal{Z}\}$. Then the space $2^{Y}[\mathcal{Z}_Y]$ is not a subspace of $2^{X}[\mathcal{Z}]$. However, $2^{Y}[\mathcal{Z}_Y]$ is homeomorphic with the subspace $\{f\in 2^X: f[X\setminus Y]=\{0\}\}$ of $2^X[\mathcal{Z}]$.
\end{remark}

\begin{remark}
\label{s4:r15}
We recall that an \emph{Alexandroff space} is a topological space in which the intersection of an arbitrary family of open sets is an open set (see, e.g., \cite[Chapter 8]{Rich}). A Hausdorff space is an Alexandroff space if and only if it is discrete. Therefore, by Theorems \ref{s3:t2}(v) and \ref{s4:t2}(iv),  the spaces $\mathbf{S}_{\mathcal{P}}(X, \mathcal{Z})$, $\mathbf{S}(X, \mathcal{Z})$ and $2^{X}[\mathcal{Z}]$ are Alexandroff spaces if and only if they are discrete.
\end{remark}

Let us finish this section with the following proposition, which summarizes elementary results from Sections \ref{s3} and \ref{s4} that are of significant importance in the sequel. This proposition will be invoked repeatedly in the arguments to come.

\begin{proposition}
\label{s4:p16}
Let $\mathcal{Z}$ be a bornology on an infinite set $X$. Then the following conditions are satisfied.
\begin{enumerate}
\item[(a)] The spaces $\mathbf{S}(X, \mathcal{Z})$, $\mathbf{S}_{\mathcal{P}}(X, \mathcal{Z})$ and $2^X[\mathcal{Z}]$ are all zero-dimensional and Hausdorff.
\item[(b)] If $X\notin\mathcal{Z}$, then the spaces $\mathbf{S}(X, \mathcal{Z})$ and $2^X[\mathcal{Z}]$ are  both crowded, and the space $\mathbf{S}_{\mathcal{P}}(X, \mathcal{Z})$ is neither discrete nor crowded. In particular, if $X$ is uncountable, then the spaces $\mathbf{S}(X, [X]^{\leq\omega})$ and $2^X[[X]^{\leq\omega}]$ are  both crowded, and the space $\mathbf{S}_{\mathcal{P}}(X, [X]^{\leq\omega})$ is neither discrete nor crowded.
\item[(c)] If $X\in\mathcal{Z}$, then the spaces $\mathbf{S}(X, \mathcal{Z})$, $\mathbf{S}_{\mathcal{P}}(X, \mathcal{Z})$ and $2^X[\mathcal{Z}]$ are all discrete, so they are also $P$-spaces.
\end{enumerate}
\end{proposition}
\begin{proof}
(a) This has been established in Theorems \ref{s3:t2}(v) and \ref{s4:t2}(iv).
\smallskip

(b) By Theorem \ref{s3:t2}(vi), the space $\mathbf{S}_{\mathcal{P}}(X, \mathcal{Z})$ is not crowded. Therefore, that (b) holds can be deduced from Theorem \ref{s3:t2}(xiii), Corollary \ref{s3:c3} and Theorem \ref{s4:t2}(vi).
\smallskip

(c) This follows from Theorems \ref{s3:t2}(xii) and \ref{s4:t2}(v).
\end{proof}  

\section{When can $\mathbf{S}_{\mathcal{P}}(\omega_1, [\omega_1]^{\leq\omega})$ be a $P$-space?}
\label{s5}

The main aim of this section is to prove that $\mathbf{S}_{\mathcal{P}}(\omega_1, [\omega_1]^{\leq\omega})$ is a $P$-space if and only if $\aleph_1$ is regular. Moreover, for an uncountable well-ordered cardinal $\kappa$, we show that the conjunction $\cf(\aleph_1)=\aleph_1\wedge \mathbf{CMC}(\aleph_0, \infty)$ implies that $\mathbf{S}_{\mathcal{P}}(\kappa, [\kappa]^{\leq\omega})$ is a $P$-space. We begin with the following general theorem. 

\begin{theorem}
\label{s5:t1}
Let $\kappa$ be an infinite well-ordered cardinal. Let $\mathcal{Z}$ be the family of all subsets of $\kappa$ not cofinal with $\kappa$. Then both $\mathbf{S}(\kappa, \mathcal{Z})$ and $2^{\kappa}[\mathcal{Z}]$ are crowded. The space $\mathbf{S}_{\mathcal{P}}(\kappa, \mathcal{Z})$ is neither discrete nor crowded. Furthermore, the following conditions are equivalent:
\begin{enumerate}
\item[(i)] $\mathbf{S}_{\mathcal{P}}(\kappa, \mathcal{Z})$ and $2^{\kappa}[\mathcal{Z}]$ are $P$-spaces;
\item[(ii)] $\mathbf{S}(\kappa, \mathcal{Z})$ is a $P$-space;
\item[(iii)] $\cf(\kappa)>\aleph_0$.
 \end{enumerate}
\end{theorem}

\begin{proof}
First, we notice that $\mathcal{Z}$ is a bornology on $\kappa$ such that $\kappa\notin\mathcal{Z}$. It follows from Proposition \ref{s4:p16}(b) that the spaces $\mathcal{S}(\kappa, \mathcal{Z})$ and $2^{\kappa}[\mathcal{Z}]$ are both crowded, and $\mathbf{S}_{\mathcal{P}}(\kappa, \mathcal{Z})$ is neither discrete nor crowded. 

By Theorem \ref{s4:t8}(iii), the implication (i) $\rightarrow$ (ii) is true. It remains to prove that (ii) implies (iii) and (iii) implies (i).
\smallskip

(ii) $\rightarrow$ (iii) Suppose $\mathbf{S}(\kappa, \mathcal{Z})$ is a $P$-space. Aiming for a contradiction, we assume $\cf(\kappa)=\aleph_0$. Let $A$ be a countable subset of $\kappa$ such that $A$ is cofinal with $\kappa$. For every $a\in A$, we have $a\in\mathcal{Z}$ and $B_{\emptyset, a}\in\tau[\mathcal{Z}]$.  However, $\bigcap_{a\in A}B_{\emptyset, a}=\{\emptyset\}\notin\tau[\mathcal{Z}]$, so $\mathbf{S}(\kappa, \mathcal{Z})$ is not a $P$-space. The contradiction obtained shows that (ii) implies (iii).
\smallskip 

(iii) $\rightarrow$ (i) Suppose $\cf(\kappa)>\aleph_0$. Then $\mathcal{Z}$ is a $\sigma$-ideal. For every $n\in\omega$, let $O_n\in \tau_{\mathcal{P}}[\mathcal{Z}]$. Suppose  $x\in\bigcap_{n\in\omega}O_n$. Given $n\in\omega$, there exists $z\in\mathcal{Z}(x)$ such that $B_{x,z}^{\mathcal{P}}\subseteq O_n$. Since $z$ is not cofinal with $\kappa$, there exists $\alpha\in\kappa$ such that $z\subseteq \alpha$. It follows that  $\alpha\setminus x\in\mathcal{Z}(x)$ and  $B_{x ,\alpha\setminus x}^{\mathcal{P}}\subseteq B_{x,z}^{\mathcal{P}}\subseteq O_n$. We may thus define, for every $n\in\omega$,
$$\alpha_n=\min\{\alpha\in\kappa: B_{x, \alpha\setminus x}^{\mathcal{P}}\subseteq O_{n}\}.$$
We put $\beta=\bigcup_{n\in\omega}\alpha_n$. Then, by (iv), $\beta\setminus x\in\mathcal{Z}$, so $\beta\setminus x\in\mathcal{Z}(x)$. We notice that $x\in B_{x, \beta\setminus x}^{\mathcal{P}}\subseteq\bigcap_{n\in\omega}O_n$. This implies that $x$ is a $P$-point of $\mathbf{S}_{\mathcal{P}}(\kappa, \mathcal{Z})$. Hence $\mathbf{S}_{\mathcal{P}}(\kappa, \mathcal{Z})$ is  $P$-space as required.
\end{proof}

\begin{remark}
\label{s5:r2}
For an infinite well-ordered cardinal $\kappa$, in much the same way as in the proof of Theorem \ref{s5:t1}, one can show that conditions (i)--(iii) of Theorem \ref{s5:t1} remain equivalent if  $\mathcal{Z}=[\kappa]^{<\kappa}$.
\end{remark}

\begin{remark}
\label{s5:r3}
Let $\kappa$ be a well-ordered cardinal such that $\cf(\kappa)>\aleph_0$, and let $\mathcal{Z}$ be the family of all subsets of $\kappa$ that are not cofinal with $\kappa$. Then we know from Theorem \ref{s5:t1} that $2^{\kappa}[\mathcal{Z}]$ is a $P$-space. However, the proof of it shown above is rather complicated because it is deduced from Theorem \ref{s4:t8}(ii), the proof of which involves Theorem \ref{s4:t3}. Let us include below an alternative direct proof that, under our assumptions about $\kappa$, $2^{\kappa}[\mathcal{Z}]$ is a $P$-space. 

For every $n\in\omega$, let $O_n\in \tau_2[\mathcal{Z}]$. Suppose  $f\in\bigcap_{n\in\omega}O_n$. Given $n\in\omega$, there exists a non-empty $z\in\mathcal{Z}$ such that $[f\upharpoonright z]_{\kappa}\subseteq O_n$. Since $z$ is not cofinal with $\kappa$, there exists $\alpha\in\kappa$ such that $z\subseteq \alpha$. It follows that $[f\upharpoonright \alpha]_{\kappa}\subseteq [f\upharpoonright z]_{\kappa}\subseteq O_n$. We may thus define, for every $n\in\omega$,
$$\alpha_n=\min\{\alpha\in\kappa: [f\upharpoonright \alpha]_{\kappa}\subseteq O_{n}\}.$$
We put $\beta=\bigcup_{n\in\omega}\alpha_n$. Then $\beta\neq\emptyset$ and, since $\cf(\kappa)>\aleph_0$, we have $\beta\in\mathcal{Z}$. We notice that $[f\upharpoonright \beta]_{\kappa}\subseteq\bigcap_{n\in\omega}O_n$. This implies that $2^{\kappa}[\mathcal{Z}]$ is a $P$-space.
\end{remark}

\begin{theorem}
\label{s5:t4}
The following conditions are equivalent:
\begin{enumerate}
\item[(i)] $[\omega_1]^{\leq\omega}$ is a $\sigma$-ideal;
\item[(ii)] $\cf(\aleph_1)=\aleph_1$;
\item[(iii)] $\mathbf{S}_{\mathcal{P}}(\omega_1, [\omega_1]^{\leq\omega})$ and $2^{\omega_1}[[\omega_1]^{\leq\omega}]$ are both $P$-spaces;
\item[(iv)] $\mathbf{S}(\omega_1, [\omega_1]^{\leq\omega})$ is a $P$-space.
\end{enumerate}
\end{theorem}

\begin{proof} In view of Theorem \ref{s4:t8}(iii), the implication (iii) $\rightarrow$ (iv) is true. By Theorem \ref{s3:t6}(i), the implication (iv) $\rightarrow$ (i) is also true. It is straightforward that the implication (i) $\rightarrow$ (ii) is true. Therefore, it suffices to prove that (ii) implies (iii).
\smallskip

(ii) $\rightarrow$ (iii) Suppose that (ii) holds. Consider the ideal $\mathcal{Z}$ of all subsets of $\omega_1$ not cofinal with $\omega_1$. By (ii), $\mathcal{Z}=[\omega_1]^{\leq\omega}$. Hence, by (ii) and Theorem \ref{s5:t1}, (iii) holds.
\end{proof}

\begin{theorem}
\label{s5:t5}
It is unprovable in $\mathbf{ZF}$ that $\mathbf{S}(\omega_1, [\omega_1]^{\leq\omega})$ is a $P$-space. In consequence, it is also unprovable in $\mathbf{ZF}$ that the spaces $\mathbf{S}_{\mathcal{P}}(\omega_1, [\omega_1]^{\leq\omega})$ and $2^{\omega_1}[[\omega_1]^{\leq\omega}]$ are $P$-spaces.
\end{theorem}

\begin{proof}
If $\mathbf{S}(\omega_1, [\omega_1]^{\leq\omega})$ is not a $P$-space, then it follows from Theorem \ref{s4:t8}(ii)--(iii) that $\mathbf{S}_{\mathcal{P}}(\omega_1, [\omega_1]^{\leq\omega})$ and $2^{\omega_1}[[\omega_1]^{\leq\omega}]$ are not $P$-spaces. On the other hand, if $\mathbf{S}(\omega_1, [\omega_1]^{\leq\omega})$ is a $P$-space, then, by Theorem \ref{s5:t4}, $\cf(\aleph_1)=\aleph_1$. Thus, to conclude the proof, it suffices to use the well-known fact that the equality $\cf(\aleph_1)=\aleph_1$ is unprovable in $\mathbf{ZF}$. For instance, $\aleph_1$ is singular in the Feferman--L\'{e}vy Model $\mathcal{M}9$ in \cite{hr} (see \cite[pp. 153--154]{hr} and \cite[Example 15.57, p. 259]{Je}), so  $\mathbf{S}(\omega_1, [\omega_1]^{\leq\omega})$ is not a $P$-space in $\mathcal{M}9$. (We give a detailed description of Model $\mathcal{M}9$ in the proof of Theorem \ref{s10:t1} below.)
\end{proof}

\begin{theorem}
\label{s5:t6}
Let $\kappa$ be any uncountable well-ordered cardinal. Then each of the following statements implies the one beneath it:
\begin{enumerate}
\item[(i)] $\cf(\aleph_1)=\aleph_1\wedge\mathbf{CMC}(\aleph_{0},\infty)$;
\item[(ii)] $\mathbf{S}_{\mathcal{P}}(\kappa, [\kappa]^{\leq\omega})$ and $2^{\kappa}[[\kappa]^{\leq\omega}]$ are both $P$-spaces;
\item[(iii)] $\mathbf{S}(\kappa, [\kappa]^{\leq\omega})$ is a $P$-space.
\end{enumerate}
\end{theorem}

\begin{proof}
If $\kappa=\omega_1$, then the conclusion follows from Theorem \ref{s5:t4}. 

Suppose $\kappa>\omega_1$. Note that, if $\cf(\aleph_1)=\aleph_1$, then, since $\kappa$ is well ordered, we have that $[\kappa]^{\leq\omega}$ is a $\sigma$-ideal. Hence, by Theorem \ref{s4:t10}, the implication (i) $\rightarrow$ (ii) is true. That the implication (ii) $\rightarrow$ (iii) is also true follows from Theorem \ref{s4:t8}(iii).
\end{proof}

\begin{remark}
\label{s5:r7}
It is unknown whether or not $\mathbf{CMC}(\aleph_{0},\infty)$ implies $\cf(\aleph_{1})=\aleph_1$. Moreover, if $\kappa$ is any uncountable well-ordered cardinal, then it is also unknown if $\mathbf{CMC}(\aleph_{0},\infty)$ (or $\cf(\aleph_{1})=\aleph_{1}$) implies either of (ii) or (iii) of Theorem \ref{s5:t6}.
\end{remark}

\section{An equivalent of $\mathbf{CAC}$ in terms of $\mathbf{S}_{\mathcal{P}}(X, [X]^{\leq\omega})$. The space $\mathbf{S}(X, [X]^{\leq\omega})$ for a quasi Dedekind-finite set $X$}
\label{s6}

Our goal for this section is to prove that $\mathbf{CAC}$ holds if and only if, for every infinite set $X$, $\mathbf{S}_{\mathcal{P}}(X, [X]^{\leq\omega})$ is a $P$-space. Taking the opportunity, we give a proof in $\mathbf{ZF}$ that Cantor cubes are not $P$-spaces. Another goal is to prove one of the most important results of the paper, asserting that if $X$ is an infinite quasi Dedekind-finite set, then $\mathbf{S}(X, [X]^{\leq\omega})$ is a $P$-space. This result will play a very significant role in many other theorems. 

\begin{theorem}
\label{s6:t1}
The following conditions are equivalent:
\begin{enumerate}
\item[(i)] $\mathbf{CAC}$;
\item[(ii)] $\mathbf{CMC}(\aleph_0, \infty)\wedge$``For every infinite set $X$, $[X]^{\leq\omega}$ is a $\sigma$-ideal'' (in fact, the statement ``For every infinite set $X$, $[X]^{\leq\omega}$ is a $\sigma$-ideal'' is equivalent to $\mathbf{CUC}$, see Theorem \ref{s7:t4}(a) below.);
\item[(iii)] for every infinite set $X$, $2^X[[X]^{\leq\omega}]$ is a $P$-space;
\item[(iv)] for every infinite set $X$, $\mathbf{S}_{\mathcal{P}}(X, [X]^{\leq\omega})$ is a $P$-space.
\end{enumerate}
\end{theorem}
\begin{proof} (i) $\leftrightarrow$ (ii) This is straightforward.
\smallskip

(ii) $\rightarrow$ (iii) This follows from Theorem \ref{s4:t10}.
\smallskip

(iii) $\leftrightarrow$ (iv) Since $[X]^{\leq\omega}$ is a bornology on $X$, by Theorem \ref{s4:t8}(ii), conditions (iii) and (iv) are equivalent.

(iv) $\rightarrow$ (i) Since $\mathbf{CAC}$ is equivalent to its partial version (see \cite[Form 8 A]{hr}), it suffices to show that if (iv) holds, then every denumerable family of non-empty sets has a partial choice function, that is, every such family has an infinite subfamily with a choice function. To this end, let $\mathcal{A}=\{A_{n}:n\in\omega\}$ be a denumerable family of pairwise disjoint non-empty sets. Towards a contradiction, we assume that $\mathcal{A}$ has no partial choice function.

Let $X=\bigcup\mathcal{A}$ and consider the space $\mathbf{S}_{\mathcal{P}}(X, [X]^{\leq\omega})$.  For every $n\in\omega$, we define 
\[
V_n=\bigcup\Big\{ B^{\mathcal{P}}_{\emptyset, z}: z\in [A_{n}]^{\leq\omega}\setminus\{\emptyset\}\Big\}.
\]
The sets $V_n$ are all open in $\mathbf{S}_{\mathcal{P}}(X, [X]^{\leq\omega})$, and $\emptyset\in\bigcap_{n\in\omega}V_n$. Assuming that $\mathbf{S}_{\mathcal{P}}(X, [X]^{\leq\omega})$ is a $P$-space, we deduce that there exists $z^{*}\in [X]^{\leq\omega}\setminus\{\emptyset\}$ such that $B^{\mathcal{P}}_{\emptyset, z^{*}}\subseteq\bigcap_{n\in\omega}V_n$. Since $z^{*}$ is a countable subset of $X$, and $\mathcal{A}$ has no partial choice function, it follows that, for some $n\in\omega$, $z^{*}\subseteq \bigcup_{i\in n+1}A_{i}$. Then $X\setminus z^{*}\in B^{\mathcal{P}}_{\emptyset, z^{*}}\subseteq V_{n+1}$, so there exists $z\in [A_{n+1}]^{<\omega}\setminus\{\emptyset\}$ with $X\setminus z^{*}\in B^{\mathcal{P}}_{\emptyset, z}$. For such a $z$, we have $z\subseteq z^{*}$. This is impossible because $z^{*}\cap A_{n+1}=\emptyset$. The contradiction obtained proves that (iv) implies (i).         
\end{proof}

It is well known that every compact Hausdorff $P$-space is a finite discrete space in $\mathbf{ZFC}$ (see, for example, \cite[Corollary 5.4]{gh}) and, therefore, Cantor cubes are not $P$-spaces in $\mathbf{ZFC}$. Since a Cantor cube may fail to be compact in $\mathbf{ZF}$, it is natural to ask if it is provable in $\mathbf{ZF}$ that Cantor cubes are not $P$-spaces. To give an affirmative answer to this question in our next theorem, we apply the spaces of the form $\mathbf{S}_{\mathcal{P}}(X, [X]^{<\omega})$ and $2^X[[X]^{<\omega}]$. We have already noticed in Theorem \ref{s4:t2}(iii) that, for every infinite set $X$, the space $2^X[[X]^{<\omega}]$ is the Cantor cube $\mathbf{2}^X$.

\begin{theorem}
\label{s6:t2}
For every infinite set $X$, the space $\mathbf{S}_{\mathcal{P}}(X, [X]^{<\omega})$ and the Cantor cube $\mathbf{2}^X$($=2^X[[X]^{<\omega}]$) are not $P$-spaces in $\mathbf{ZF}$.
\end{theorem}
\begin{proof}
If $X$ is a Dedekind-infinite set, then the bornology $[X]^{<\omega}$ is not a $\sigma$-ideal, so $\mathbf{S}_{\mathcal{P}}(X, [X]^{<\omega})$ and $2^X[[X]^{<\omega}]$ are not $P$-spaces by Theorem \ref{s4:t8}(ii)--(iv).

Suppose that $X$ is an infinite Dedekind-finite set.  Then $[X]^{<\omega}=[X]^{\leq\omega}$ (see Proposition \ref{s6:p6}). In view of Theorem \ref{s4:t8}(iii), to complete the proof, it suffices to check that $\mathbf{S}_{\mathcal{P}}(X, [X]^{<\omega})$ is not a $P$-space. To this aim, suppose that $\mathbf{S}_{\mathcal{P}}(X, [X]^{<\omega})$ is a $P$-space. For every $n\in\mathbb{N}$, we define 
\[
V_n=\bigcup\Big\{ B^{\mathcal{P}}_{\emptyset, z}:  z\in [X]^n\Big\}.
\]
For every $n\in\mathbb{N}$, the set $V_n$ is open in $\mathbf{S}_{\mathcal{P}}(X, [X]^{<\omega})$, and $\emptyset\in V_n$. Since $\mathbf{S}_{\mathcal{P}}(X, [X]^{<\omega})$ is a $P$-space, there exists $z^{*}\in [X]^{<\omega}$ such that $B^{\mathcal{P}}_{\emptyset, z^{*}}\subseteq\bigcap_{n\in\mathbb{N}}V_n$. Let $n^{*}=|z^{*}|$. Since $X\setminus z^{*}\in B^{\mathcal{P}}_{\emptyset, z^{*}}\subseteq V_{n^{*}+1}$, there exists $z_0\in [X]^{n^{*}+1}\setminus\{\emptyset\}$ such that $X\setminus z^{*}\in B^{\mathcal{P}}_{\emptyset, z_0}$. Then $z_0\subseteq z^{*}$, which is impossible because $|z^{*}|<|z_0|$. The contradiction obtained shows that $\mathbf{S}_{\mathcal{P}}(X, [X]^{<\omega})$ is not a $P$-space. This completes the proof.
\end{proof}

The following combinatorial lemma of Howard--Keremedis--Rubin--Rubin from \cite{hkrr} will be useful for the proof of Theorem \ref{s6:t4} below.

\begin{lemma}
\label{lem:PH}
(\cite[Lemma 5]{hkrr}) $(\mathbf{ZF})$ Let $A$ be an infinite collection of finite sets such that there is a finite set $x\subseteq\bigcup A$ with the property that, for some natural number $k$, $x$ intersects all but $k$ members of $A$ in a non-empty set. We denote this statement by $\Phi(A,k,x)$. Let $n_{0}$ be the least natural number for which there are a $k\in\omega$ and a finite $x\subseteq\bigcup A$ of size $n_{0}$ such that $\Phi(A,k,x)$. Fix such a $k$, say $k_{0}$. Then the set
$$\{x:|x|=n_{0}\text{ and }\Phi(A,k_{0},x)\}$$
is finite.
\end{lemma}

\begin{theorem}
\label{s6:t4}
The following hold:
\begin{enumerate}
\item[(a)] For every infinite set $X$, the following conditions are equivalent:
\begin{enumerate}
\item[(i)] $\mathbf{S}(X, [X]^{<\omega})$ is a $P$-space;
\item[(ii)] $[X]^{<\omega}$ is a $\sigma$-ideal;
\item[(iii)] $X$ is quasi Dedekind-finite.
\end{enumerate}
\item[(b)] Suppose $X$ is an infinite quasi Dedekind-finite set. Then $\mathbf{S}(X,[X]^{\leq\omega})$ is a $P$-space. In particular, if $X$ is  infinite and weakly Dedekind-finite or amorphous, then $\mathbf{S}(X,[X]^{\leq\omega})$ is a $P$-space.

\item[(c)] Suppose $X$ is an infinite Dedekind-finite set of reals. Then $\mathbf{S}(X,[X]^{\leq\omega})$ is a $P$-space. 

\item[(d)] Suppose $A$ is either an infinite Dedekind-finite set of reals or an infinite quasi Dedekind-finite set, and also suppose $X$ is an infinite subset of $[A]^{<\omega}$. Then $\mathbf{S}(X,[X]^{\leq\omega})$ is a $P$-space.

\item[(e)] It is relatively consistent with $\mathbf{ZF}$ that there exists an infinite Dedekind-finite set $A$ of reals such that $\mathbf{S}(A, [A]^{\leq\omega})$ is a $P$-space, but $\mathbf{S}_{\mathcal{P}}(A, [A]^{\leq\omega})$ and $2^A[[A]^{\leq\omega}]$ are not $P$-spaces.
\end{enumerate}
\end{theorem}

\begin{proof}

(a) Let $X$ be an infinite set. 
\smallskip

(i) $\rightarrow$ (ii) This follows from Theorem \ref{s3:t6}(i).
\smallskip

(ii) $\rightarrow$ (iii) This is straightforward.
\smallskip 

(iii) $\rightarrow$ (i) Suppose $X$ is quasi Dedekind-finite. Then we have $[X]^{\leq\omega}=[X]^{<\omega}$ and $[X]^{<\omega}$ is a $\sigma$-ideal.
To show that $\mathbf{S}(X,[X]^{<\omega})$ is a $P$-space, we shall apply Theorem \ref{s3:t9}. To this end, we fix a family $\{A_n: n\in\omega\}$ which has the following properties:
\begin{enumerate}
\item[($P_1$)] $(\forall n\in\omega)\quad \emptyset\neq A_n\subseteq [X]^{<\omega}\setminus\{\emptyset\}$;
\item[($P_2$)] $(\forall n\in\omega)\quad \emptyset\neq\{x\in [X]^{<\omega}: (\forall z\in A_n)\quad x\cap z\neq\emptyset\}$.
\end{enumerate}

\noindent For every $n\in\omega$, we put
\[
C_n=\left\{x\in [X]^{<\omega}: (\forall z\in A_n)\quad x\cap z\neq\emptyset\right\}\quad\text{and}\quad D_n=\Big\{x\in \left[\bigcup A_n\right]^{<\omega}: (\forall z\in A_n)\quad x\cap z\neq\emptyset\Big\}.
\]
In view of Theorem \ref{s3:t9}, to complete the proof of (a), it suffices to prove that the family $\{C_n: n\in\omega\}$ has a choice function. We notice that, for every $n\in\omega$, $D_n\subseteq C_n$ and $D_n\neq\emptyset$. To see the second assertion, fix $n\in\omega$. Then, for any $x\in C_n$ (recall that, for all $i\in\omega$, $C_{i}\neq\emptyset$), $x\cap(\bigcup A_n)\in D_n$, so $D_{n}\neq\emptyset$. In view of the latter two observations, it suffices to define a choice function for the family $\{D_n: n\in\omega\}$. That is, a function $f:\omega\to\bigcup_{n\in\omega} D_n$ such that, for every $n\in\omega$, $f(n)\in D_n$.

Let $N=\{n\in\omega: A_n \text{ is infinite}\}$. Note that, since $X$ is quasi Dedekind-finite, $N$ is a cofinite subset of $\omega$. If $n\in\omega\setminus N$, we put $f(n)=\bigcup A_n$ and notice that $f(n)\in D_n$.

Now, suppose that $n\in N$. Since $D_n\neq\emptyset$, we deduce that
$$\left\{r\in\mathbb{N}:\left(\exists x\in \left[\bigcup A_{n}\right]^{r}\right)\Phi(A_{n},0,x)\right\}\neq\emptyset,$$
\noindent where $\Phi(A_{n},0,x)$ is the formula from Lemma \ref{lem:PH} applied to $A_n$ in place of $A$. Therefore, we can define
\begin{align*}
r_{n}&=\min\left\{r\in\mathbb{N}:\left(\exists x\in \left[\bigcup A_{n}\right]^{r}\right)\Phi(A_{n},0,x)\right\},\\
F_{n}&=\left\{x\in\left[\bigcup A_{n}\right]^{r_{n}}:\Phi(A_{n},0,x)\right\}.
\end{align*}
Of course, $F_n\neq\emptyset$ and, by Lemma \ref{lem:PH}, the set $F_n$ is finite. We put $f(n)=\bigcup F_n$. Then $f(n)\in D_n$. (Note that, since $X$ is quasi Dedekind-finite, the set $\{f(n):n\in\omega\}$ is finite.) This completes the proof of (a).
\smallskip

(b) First, note that the second assertion of (b) follows from the first, from the fact that every weakly Dedekind-finite set is quasi Dedekind-finite, and from the fact that every amorphous set is weakly Dedekind-finite. The first assertion of (b) follows from (a) and the fact that if $X$ is quasi Dedekind-finite, then $[X]^{<\omega}=[X]^{\leq\omega}$.
\smallskip

(c) This follows from (b) and from the fact that any Dedekind-finite set of reals is quasi Dedekind-finite.
\smallskip

(d) Assuming that $A$ and $X$ are as in the statement of (d), it follows that $X$ is quasi Dedekind-finite. Thus, by (b), $\mathbf{S}(X,[X]^{\leq\omega})$ is a $P$-space.  
\smallskip

(e) Consider the Basic Cohen Model $\mathcal{M}1$ of \cite{hr}. (A detailed description of $\mathcal{M}1$ is given in the proof of the forthcoming Theorem \ref{thm:Models_M1_N3}(1).) Let $A$ be the set of the denumerably many added generic reals. It is known that $A$ is a Dedekind-finite set in $\mathcal{M}1$ (see \cite{j}). Hence, $[A]^{\leq\omega}=[A]^{<\omega}$ in $\mathcal{M}1$.
By (c), $\mathbf{S}(A,[A]^{\leq\omega})$ is a $P$-space in $\mathcal{M}1$, and, by Theorem \ref{s6:t2}, $\mathbf{S}_{\mathcal{P}}(A, [A]^{\leq\omega})$ and $2^A[[A]^{\leq\omega}]$ are not $P$-spaces in $\mathcal{M}1$. 
\end{proof}

\begin{definition}
\label{s6:d05}
For an infinite set $X$ and a bornology $\mathcal{Z}$ on $X$, let $\mathbf{P_0}(X,\mathcal{Z})$ be the following condition: For every family $\{A_n: n\in\omega\}$ having the following properties:
\begin{enumerate}
\item[($P_1$)] $(\forall n\in\omega)\quad \emptyset\neq A_n\subseteq [X]^{<\omega}\setminus\{\emptyset\}$,
\item[($P_2$)] $(\forall n\in\omega)\quad \emptyset\neq\{x\in \mathcal{Z}: (\forall z\in A_n)\quad x\cap z\neq\emptyset\}$,
\end{enumerate}
\noindent the family $\{\{x\in\mathcal{Z}: (\forall z\in A_n)\quad x\cap z\neq\emptyset\}: n\in\omega\}$ has a choice function.
\end{definition}

Applying Lemma \ref{lem:PH}, in much the same way as in the proof that (iii) implies (i) in Theorem \ref{s6:t4}(a), we can show that the following proposition is true in $\mathbf{ZF}$.

\begin{proposition}
\label{s6:p06}
For every infinite set $X$, condition $\mathbf{P_0}(X, [X]^{<\omega})$ is satisfied.
\end{proposition}

\begin{remark}
\label{s6:r07}
\begin{enumerate}
\item[(a)] Let $X$ be an uncountable Dedekind-infinite set and let $\mathcal{Z}=[X]^{<\omega}$. Then, by Theorem \ref{s3:t6}(i) or Theorem \ref{s6:t4}(a), $\mathbf{S}(X, \mathcal{Z})$ is not a $P$-space. However, by Proposition \ref{s6:p06}, condition $\mathbf{P_0}(X, \mathcal{Z})$, which is the second part of condition (iii) of Theorem \ref{s3:t9}, is satisfied. Therefore, the assumption that $\mathcal{Z}$ is a $\sigma$-ideal cannot be omitted in condition (iii) of Theorem \ref{s3:t9}.
\item[(b)] It follows from Theorem \ref{s6:t4}(a), taken together with Propositions \ref{s6:p6} and \ref{s6:p06}, that if for every infinite set $X$ satisfying $\mathbf{P_0}(X, [X]^{\leq\omega})$, the space $\mathbf{S}(X, [X]^{\leq\omega})$ is a $P$-space, then every infinite Dedekind-finite set is quasi Dedekind-finite.
\end{enumerate}
\end{remark}

\section{New forms $\mathbf{PS}_0-\mathbf{PS}_5$ and the equivalence of $\mathbf{CAC_{fin}}$ with $\mathbf{PS}_{2}$}
\label{s7}

In this section, we turn special attention to the open problem of the deductive strength of the new forms $\mathbf{PS}_0$--$\mathbf{PS}_5$ defined as follows.

\begin{definition}
\label{s7:d1}
\begin{enumerate}
\item $\mathbf{PS}_0$: For every uncountable set $X$, $\mathbf{S}(X,[X]^{\leq\omega})$ is a $P$-space.

\item  $\mathbf{PS}_1$: For every uncountable set $X$, there exists an uncountable set $Y\subseteq X$ such that $\mathbf{S}(Y,[Y]^{\leq\omega})$ is a $P$-space.

\item $\mathbf{PS}_2$: For every infinite set $X$, there exists an infinite set $Y\subseteq X$ such that $\mathbf{S}(Y,[Y]^{\leq\omega})$ is a $P$-space.

\item $\mathbf{PS}_3$: For every infinite set $X$, either $\mathbf{S}(X,[X]^{\leq\omega})$ is a $P$-space or $X$ has an amorphous subset.

\item  $\mathbf{PS}_4$: For every infinite set $X$, either $\mathbf{S}(X,[X]^{\leq\omega})$ is a $P$-space or $X$ has an infinite weakly Dedekind-finite subset.

\item  $\mathbf{PS}_5$: For every infinite set $X$, either $\mathbf{S}(X,[X]^{\leq\omega})$ is a $P$-space or $X$ has an infinite quasi Dedekind-finite subset. 
\end{enumerate}
\end{definition}

\begin{remark}
\label{s7:r2}
\begin{enumerate}
\item[(a)] Let us recall here that if $Y$ is an infinite subset of a set $X$ such that $\mathbf{S}(X,[X]^{\leq\omega})$ is a $P$-space, then $\mathbf{S}(Y,[Y]^{\leq\omega})$ is a $P$-space as a subspace of the $P$-space $\mathbf{S}(X,[X]^{\leq\omega})$ (see Remark \ref{s3:r10}).

\item[(b)] We note that, in view of (a) and Proposition \ref{s4:p16}(c) (and the fact that if $X$ is an infinite set, then either $|X|=\aleph_{0}$ or $|X|\nleq\aleph_{0}$), the statement ``For every infinite set $X$, $\mathbf{S}(X,[X]^{\leq\omega})$ is a $P$-space'' is equivalent to $\mathbf{PS}_0$.
The forms $\mathbf{PS}_1$--$\mathbf{PS}_5$ are formally weaker than $\mathbf{PS}_0$.

\item[(c)] By Theorems \ref{s4:t8}(iii) and \ref{s6:t1}, $\mathbf{CAC}$ implies $\mathbf{PS}_0$. The forthcoming Theorem \ref{thm:Models_M1_N3} will show that $\mathbf{PS}_0$ does not imply $\mathbf{CAC}$ in $\mathbf{ZF}$.
\end{enumerate}
\end{remark}

Let us show basic relationships between the forms $\mathbf{PS}_i$.

\begin{proposition}
\label{s7:p3}
\begin{enumerate}
\item[(i)] $(\forall i\in\{1, 2, 3, 4, 5\})\quad \mathbf{PS}_0\rightarrow \mathbf{PS}_i$.

\item[(ii)] $\mathbf{PS}_3\rightarrow \mathbf{PS}_4\rightarrow\mathbf{PS}_5\rightarrow \mathbf{PS}_1\rightarrow \mathbf{PS}_2$.
\end{enumerate}
\end{proposition}

\begin{proof}
(i) That (i) holds has been observed in Remark \ref{s7:r2}(b).
\smallskip

(ii) Since every amorphous set is an infinite weakly Dedekind-finite set, and every weakly Dedekind-finite set is quasi Dedekind finite, the implications $\mathbf{PS}_3\rightarrow \mathbf{PS}_4\rightarrow\mathbf{PS}_5$ are true. It is obvious that the implication $\mathbf{PS}_1\rightarrow \mathbf{PS}_2$ is also true, so it remains to show that $\mathbf{PS}_5$ implies $\mathbf{PS}_1$.

Assuming $\mathbf{PS}_5$, we fix an uncountable set $X$. We notice that if $Y\subseteq X$ and $Y$ is an infinite quasi Dedekind-finite set, then $Y$ is uncountable and, by Theorem \ref{s6:t4}(a), the space $\mathbf{S}(Y, [Y]^{\leq\omega})$ is a $P$-space. Hence, if $\mathbf{PS}_5$ is true, then so is $\mathbf{PS}_1$.
\end{proof}

\begin{theorem}
\label{s7:t4}
The following hold:
\begin{enumerate}
\item[(a)] $\mathbf{CUC}$ is equivalent to ``For every infinite set $X$, $[X]^{\leq\omega}$ is a $\sigma$-ideal''.

\item[(b)] $\mathbf{PS}_0 \rightarrow \mathbf{PS}_1 \rightarrow \mathbf{CUC}$.

\item[(c)] $\mathbf{IDI}\rightarrow \mathbf{PS}_2$.
\end{enumerate}
\end{theorem}

\begin{proof}
(a) This is straightforward.
\smallskip

(b) By Proposition \ref{s7:p3}(i), $\mathbf{PS}_0$ implies $\mathbf{PS}_1$. Suppose that $\mathbf{CUC}$ is not true. Then we can fix a family $\{X_n: n\in\omega\}$ of countable sets such that the set $X=\bigcup_{n\in\omega}X_n$ is uncountable. We notice that if $Y$ is an uncountable subset of $X$, then $[Y]^{\leq\omega}$ is not a $\sigma$-ideal; thus, by Theorem \ref{s3:t6}(i), $\mathbf{S}(Y, [Y]^{\leq\omega})$ is not a $P$-space. Hence, if $\mathbf{CUC}$ is false, then so is $\mathbf{PS}_1$.
\smallskip

(c) This follows directly from Proposition \ref{s4:p16}(c).
\end{proof}

Theorem \ref{s6:t4} leads to \textbf{very interesting new characterizations of $\mathbf{CAC_{fin}}$} in terms of $P$-spaces. Indeed, we have the subsequent result which also shows that $\mathbf{CAC_{fin}}$ and $\mathbf{PS}_2$ are equivalent.

\begin{theorem}
\label{s7:t5}
The following are equivalent:
\begin{enumerate}
\item[(i)] $\mathbf{CAC_{fin}}$;

\item[(ii)] Dedekind-finiteness $\equiv$ quasi Dedekind-finiteness;

\item[(iii)] for every infinite set $X$, either $\mathbf{S}(X,[X]^{\leq\omega})$ is a $P$-space or $X$ is Dedekind-infinite;

\item[(iv)] for every infinite Dedekind-finite set $X$, $\mathbf{S}(X,[X]^{\leq\omega})$ is a $P$-space;

\item[(v)] $\mathbf{PS}_2$.
\end{enumerate}
\end{theorem}

\begin{proof}
(i) $\rightarrow$ (ii) It is straightforward to see that every quasi Dedekind-finite set is Dedekind-finite. For the converse, assume $\mathbf{CAC_{fin}}$. It is known that $\mathbf{CAC_{fin}}$ is equivalent to the statement ``The union of a countable family of finite sets is countable'' (see \cite[Form 10 A]{hr}). The latter fact readily yields every Dedekind-finite set is quasi Dedekind-finite. (Recall that $\mathbf{CAC_{fin}}$ $\wedge$ ``There exists an infinite, Dedekind-finite set'' is relatively consistent with $\mathbf{ZF}$; see \cite[Basic Cohen Model $\mathcal{M}1$]{hr}.) Hence, (ii) holds as required.
\smallskip

(ii) $\rightarrow$ (iii) Assume (ii). Let $X$ be an infinite set. If $X$ is Dedekind-infinite, then the conclusion is readily obtained. If $X$ is Dedekind-finite, then, by (ii), $X$ is quasi Dedekind-finite. By Theorem \ref{s6:t4}(b), $\mathbf{S}(X,[X]^{\leq\omega})$ is a $P$-space. Hence, (iii) holds as required.
\smallskip

(iii) $\rightarrow$ (iv) This is straightforward.
\smallskip

(iv) $\rightarrow$ (v) Assume (iv). Let $X$ be an infinite set. If $X$ is Dedekind-infinite, then let $Y$ be a denumerable subset of $X$. By Proposition \ref{s4:p16}(c), $\mathbf{S}(Y,[Y]^{\leq\omega})$ is a $P$-space.

If $X$ is Dedekind-finite, then, by (iv), $\mathbf{S}(X,[X]^{\leq\omega})$ is a $P$-space.
\smallskip

(v) $\rightarrow$ (i) Using the well-known fact that $\mathbf{CAC_{fin}}$ is equivalent to the statement ``Every denumerable family of non-empty finite sets has an infinite subfamily with a choice function'' (see \cite[Form 10 E]{hr}), that $\mathbf{PS}_2$ implies $\mathbf{CAC_{fin}}$  can be established by a slight modification of the proof that $\mathbf{PS}_1$ implies $\mathbf{CUC}$ in Theorem \ref{s7:t4}(b). The details are left as an easy exercise for the readers. 
\end{proof}

\begin{remark}
\label{s7:r6}
Although it is well known that $\mathbf{IDI}$ implies $\mathbf{{CAC}_{fin}}$ in $\mathbf{ZF}$ (see \cite[p. 325]{hr}), it is worth noticing that this implication can be deduced directly from Theorems \ref{s7:t4} and \ref{s7:t5}.
\end{remark}

In view of Theorems \ref{s7:t4}(b) and \ref{s7:t5}, the \emph{vital (and central) question} that emerges is whether or not $\mathbf{CUC}$ implies $\mathbf{PS}_{0}$; and we note that the above question inspired and motivated us in producing Theorem \ref{s6:t4} and several important results depending on this theorem. We would also like to point out here that the former (and longer) version of the proof of Theorem \ref{s6:t4}(a) enlightened and directed us towards the formulation of Theorem \ref{s3:t9} and of the weak choice principles $\mathbf{PS_0}(\mathbf{I})$, $\mathbf{PS_0}(\mathbf{II})$, $\mathbf{PS}_0(\mathbf{C})$ and $\mathbf{PS}_0(\mathbf{D})$, all \emph{introduced} for the first time in Definition \ref{s7:d7} below; in particular, $\mathbf{PS_0}(\mathbf{I})$ and $\mathbf{PS_0}(\mathbf{II})$ were naturally deduced from Theorem \ref{s3:t9}. 
Unfortunately, the answer to the aforementioned (and especially) difficult question still eludes us. However, we are able to provide a \emph{non-trivial partial answer} in the forthcoming Theorems \ref{s7:t9} and \ref{s7:t12} which sheds more light on this significant open problem. 

In Theorems \ref{s7:t9} and \ref{s7:t12}(1), we show, by applying Theorem \ref{s3:t9} and using the fundamental ideas of the proof of Theorem \ref{s6:t4}(a), that $\mathbf{PS_0}$ is equivalent to the conjunction of $\mathbf{CUC}$ with any of the weak choice principles 
$\mathbf{PS_0}(\mathbf{II})$ and $\mathbf{PS}_0(\mathbf{C})$, and the conjunction of $\mathbf{CUC}$ with $\mathbf{PS}_0(\mathbf{D})$ implies $\mathbf{PS_{0}}$. Moreover, in Theorem \ref{s7:t9}, we show that $\mathbf{PS_{0}}$ is equivalent to $\mathbf{PS_{0}(I)}$. As the readers will soon realize, the results of Theorems \ref{s7:t9} and \ref{s7:t12}(1), on the one hand, provide natural and appropriate premises in order to establish $\mathbf{PS_{0}}$ and, on the other hand, point to the direction that $\mathbf{CUC}$ may not suffice to prove $\mathbf{PS_{0}}$. Lemma \ref{s7:l8} (used for the proofs of Theorems \ref{s7:t9} and \ref{s7:t12}(1)) also witnesses the relationship of $\mathbf{PS_0}(\mathbf{II})$ with $\mathbf{PS}_0(\mathbf{C})$ and $\mathbf{PS}_0(\mathbf{D})$. The independence results of Theorems \ref{s7:t10} and \ref{s7:t12}(2) supply further  information about the deductive strength and the relationships between $\mathbf{PS_0}(\mathbf{II})$, $\mathbf{PS}_0(\mathbf{C})$ and $\mathbf{PS}_0(\mathbf{D})$; none of $\mathbf{PS_0}(\mathbf{II})$ and $\mathbf{PS_0(D)}$ is provable in $\mathbf{ZF}$, whereas the status of $\mathbf{PS_0(C)}$ in $\mathbf{ZF}$ is still an \emph{open problem}. In Theorem \ref{thm:CMC_PS}, we show that $\mathbf{CMC}$ implies $\mathbf{PS_{0}(II)}$ and that $\mathbf{MC}$ implies $\mathbf{PS_{0}(C)}$, and in Theorem \ref{thm:CMC_notCUCDLO}(3) we establish that the latter implications are not reversible in $\mathbf{ZFA}$. Finally, in Theorem \ref{thm:PS0(II)_not_PS0}, we show that $\mathbf{PS_{0}(II)}\wedge\mathbf{PS_{0}(C)}$ does not imply $\mathbf{CAC_{fin}}$ in $\mathbf{ZFA}$, and thus $\mathbf{PS_{0}(II)}\wedge\mathbf{PS_{0}(C)}$ is strictly weaker than $\mathbf{PS_{0}}$ in $\mathbf{ZFA}$. 

\begin{definition}
\label{s7:d7}
\begin{enumerate}
\item $\mathbf{PS_0}(\mathbf{I})$ (respectively, $\mathbf{PS}_0(\mathbf{II})$): If $X$ is any uncountable set and  $\{A_n: n\in\omega\}$ is a denumerable family with the following two properties:
\begin{enumerate}
\item[(a)] $(\forall n\in\omega)\quad \emptyset\neq A_n\subseteq [X]^{<\omega}\setminus \{\emptyset\}$;
\item[(b)] $(\forall n\in\omega)\quad \emptyset\neq C_n=\{x\in [X]^{\leq\omega}: (\forall z\in A_n)\quad x\cap z\neq\emptyset\}$,
\end{enumerate}
then $\bigcap_{n\in\omega}C_n\neq\emptyset$ (respectively, the family $\{C_n: n\in\omega\}$ has a choice function).
\medskip

\item $\mathbf{PS}_{0}(\mathbf{C})$: If $X$ is any uncountable set and $\{A_{n}:n\in\omega\}$ is a denumerable family with the following two properties:
\begin{enumerate}
\item[(c)] for every $n\in\omega$, $A_{n}$ is an uncountable subfamily of $[X]^{<\omega}\setminus\{\emptyset\}$;

\item[(d)] for every $n\in\omega$, the family 
\[
B_{n}=\big\{\mathcal{Z}:\mathcal{Z}\text{ is a countable partition of $A_{n}$ such that }(\forall Z\in\mathcal{Z})(\exists x\in [\bigcup Z]^{<\omega})
(\forall z\in Z)(x\cap z\neq\emptyset)\big\}
\]
is non-empty,  
\end{enumerate}
then $\{B_{n}:n\in\omega\}$ has a choice function.
\medskip

\item $\mathbf{PS}_{0}(\mathbf{D})$: If $X$ is any uncountable set and $\{A_{n}:n\in\omega\}$ is a denumerable family with the following two properties:
\begin{enumerate}
\item[(e)] for every $n\in\omega$, $A_{n}$ is an uncountable subfamily of $[X]^{<\omega}\setminus\{\emptyset\}$;

\item[(f)] for every $n\in\omega$, the family 
\[
B_{n}=\big\{\mathcal{Z}:\mathcal{Z}\text{ is a denumerable partition of $A_{n}$ such that }(\forall Z\in\mathcal{Z})(\exists x\in [\bigcup Z]^{<\omega})
(\forall z\in Z)(x\cap z\neq\emptyset)\big\}
\]
is non-empty,  
\end{enumerate}
then $\{B_{n}:n\in\omega\}$ has a choice function.
\end{enumerate} 
\end{definition}

\begin{lemma}
\label{s7:l8}
The following hold:
\begin{enumerate}
\item[(a)] $\mathbf{PS}_{0}(\mathbf{C})\wedge \mathbf{CAC_{fin}}$ is equivalent to $\mathbf{\mathbf{PS}_0(\mathbf{II})}\wedge \mathbf{CAC_{fin}}$.
\item[(b)] $\mathbf{PS}_{0}(\mathbf{D})$ implies $\mathbf{PS}_{0}(\mathbf{C})$.
\item[(c)] $\mathbf{PS}_{0}(\mathbf{D})\wedge \mathbf{CAC_{fin}}$ implies $\mathbf{\mathbf{PS}_0(\mathbf{II})}$.
\end{enumerate}
\end{lemma}
 
\begin{proof}
(a) ($\rightarrow$) Suppose $\mathbf{PS}_{0}(\mathbf{C})\wedge \mathbf{CAC_{fin}}$ is true. Let $X$ be an uncountable set. Let $\{A_n: n\in\omega\}$ and $\{C_n: n\in\omega\}$ be families satisfying conditions (a) and (b) of Definition \ref{s7:d7}(1), respectively. Our goal is to define a choice function $f$ for $\{C_n: n\in \omega\}$. Let 
$$N=\{n\in\omega: A_n \text{ is uncountable}\}.$$

\noindent By $\mathbf{CAC_{fin}}$, for every $n\in\omega\setminus N$, the set $\bigcup A_n$ is countable and it belongs to $C_n$. Thus, for every $n\in\omega\setminus N$, we can define $f(n)=\bigcup A_n$. 

If the set $N$ is finite and non-empty, we can fix a family $\{f(n): n\in N\}$ such that, for every $n\in N$, $f(n)\in C_n$. Therefore, if $N$ is finite, the family $\{C_n: n\in\omega\}$ has a choice function.

Assume that the set $N$ is infinite.  For every $n\in\omega$, let $B_n$ be the family defined in item (d) of Definition \ref{s7:d7}(2). Given $n\in N$,  let us show that $B_n\neq\emptyset$. Since $C_n\neq\emptyset$, we can fix $c\in C_n$; so, for every $z\in A_n$, $z\cap c\neq\emptyset$.  Since $c$ is countable, so is $[c]^{<\omega}$. For every $F\in [c]^{<\omega}\setminus\{\emptyset\}$, we define the set $Z_F$ as follows:
$$ Z_F=\{ z\in A_n: z\cap c= F\}.$$
Then the family $\mathcal{Z}(c)=\{ Z_F: F\in [c]^{<\omega}\setminus\{\emptyset\}\}\setminus\{\emptyset\}$ is a countable partition of $A_n$ having the following property:
$$(\forall F\in [c]^{<\omega}\setminus\{\emptyset\})\quad[(Z_{F}\neq\emptyset)\rightarrow(\forall u\in Z_{F})(u\cap F=F\neq\emptyset)].$$

\noindent Therefore, $\mathcal{Z}(c)\in B_n$, which shows that, for every $n\in N$, $B_n\neq\emptyset$. By $\mathbf{PS}_{0}(\mathbf{C})$, we can fix a function $g\in\prod_{n\in N}B_n$.  For every $n\in N$, let
$$\mathcal{U}_n=\{ Z\in g(n): Z \text{ is infinite}\}.$$

\noindent If $n\in N$, then since $A_n$ is uncountable and $g(n)$ is a countable partition of $A_n$, it follows from $\mathbf{{CAC}_{fin}}$ that $\mathcal{U}_n\neq\emptyset$. Using Lemma \ref{lem:PH} and arguing in much the same way as in the proof of Theorem \ref{s6:t4}(a), for every $n\in N$ and $U\in\mathcal{U}_n$, without invoking any form of choice, we can define a finite subset $t_{n, Z}$ of $\bigcup Z$ ($\subseteq\bigcup A_n$) having the following property:
\[
(\forall z\in Z)(t_{n,Z}\cap z\neq\emptyset).
\]

Now, for every $n\in\omega$, we define the set $f(n)$ as follows:
\[
f(n)=\left(\bigcup\left\{t_{n,Z}:Z\in\mathcal{U}_{n}\right\}\right)\cup\left(\bigcup\left\{\bigcup Z:Z\in g(n)\setminus\mathcal{U}_{n}\right\}\right).
\]
It is easily seen that, for every $n\in N$, the subset $f(n)$ of $X$ has the following property:
\[
(\forall z\in A_n)\quad f(n)\cap z\neq\emptyset.
\]
It follows from $\mathbf{CAC_{fin}}$ that, for every $n\in N$, the set $f(n)$ is countable. Hence, for every $n\in N$, $f(n)\in C_n$. In this way, we have defined a function $f\in\prod_{n\in\omega}C_n$. Hence, $\mathbf{PS}_0(\mathbf{II})$ is true, as required.
\smallskip

($\leftarrow$) Suppose $\mathbf{PS}_0(\mathbf{II})\wedge\mathbf{{CAC}_{fin}}$ is true. We fix any uncountable set $X$, a denumerable family $\{A_n: n\in\omega\}$ and a family $\{B_n: n\in\omega\}$ with properties (c) and (d) of Definition \ref{s7:d7}(2), respectively. Our goal is to show that $\{B_n: n\in\omega\}$ has a choice function. 

For every $n\in\omega$, we define $C_n=\{x\in [X]^{\leq\omega}: (\forall z\in A_n)\quad x\cap z\neq\emptyset\}$. Let us fix $n\in\omega$ and show that $C_n\neq\emptyset$.  We choose any $\mathcal{Z}\in\mathcal{B}_n$ and put $\mathcal{U}_{\mathcal{Z}}=\{Z\in\mathcal{Z}: Z\text{ is infinite}\}$. Note that, by $\mathbf{CAC_{fin}}$ and the fact that $A_{n}$ is uncountable, $\mathcal{U}_{\mathcal{Z}}\neq\emptyset$ for all $\mathcal{Z}\in\mathcal{B}_n$. By Lemma \ref{lem:PH}, for every $Z\in\mathcal{U}_{\mathcal{Z}}$, we can define a finite subset $t_Z$ of $\bigcup Z$, such that, for every $u\in Z$, $u\cap t_Z\neq\emptyset$. By $\mathbf{{CAC}_{fin}}$, the set 
\[
c_{\mathcal{Z}}:=\left(\bigcup\left\{t_Z: Z\in\mathcal{U}_{\mathcal{Z}}\right\}\right)\cup\left(\bigcup\left\{\bigcup Z:Z\in \mathcal{Z}\setminus\mathcal{U}_{\mathcal{Z}}\right\}\right).
\]
\noindent is countable. Furthermore, it is easy to see that $c_{\mathcal{Z}}\in C_n$. Hence, for every $n\in\omega$, $C_n\neq\emptyset$.

By $\mathbf{PS}_0(\mathbf{II})$, there exists $h\in\prod_{n\in\omega}C_n$. For every $n\in\omega$ and $F\in [h(n)]^{<\omega}\setminus\{\emptyset\}$, we put 
$$Z_{n,F}=\{ z\in A_n: z\cap h(n)=F\}\quad\text{ and }\quad \psi(n)=\{Z_{n, F}: F\in [h(n)]^{<\omega}\setminus\{\emptyset\}\}\setminus\{\emptyset\}.$$ In this way, we have defined a function $\psi$ such that $\psi\in\prod_{n\in\omega}B_n$. Hence,  $\mathbf{PS}_{0}(\mathbf{C})$ is true, as required.
\smallskip

(b) Suppose $\mathbf{PS_{0}(D)}$ is true. Let $X$, $\{A_{n}:n\in\omega\}$ and $\{B_{n}:n\in\omega\}$ be as in Definition \ref{s7:d7}(2). If, for some $n\in\omega$, there exists a finite subset of $\bigcup A_n$ which meets non-trivially every member of $A_{n}$, then $\{A_{n}\}\in B_{n}$. So, for every such $n\in\omega$, we may choose $\{A_{n}\}$ from $B_{n}$. We thus assume, without loss of generality, that, for every $n\in\omega$, if there exists $x\subseteq\bigcup A_{n}$ which meets non-trivially every member of $A_{n}$, then $x$ is infinite. Under this assumption, it is straightforward to see that, for every $n\in\omega$,
\[
B_{n}=\big\{\mathcal{Z}:\mathcal{Z}\text{ is a denumerable partition of $A_{n}$ such that }(\forall Z\in\mathcal{Z})(\exists x\in [\bigcup Z]^{<\omega})
(\forall z\in Z)(x\cap z\neq\emptyset)\big\}.
\]
Therefore, by $\mathbf{PS_{0}(D)}$, $\{B_{n}:n\in\omega\}$ has a choice function, i.e. $\mathbf{PS_{0}(C)}$ is true, as required.
\smallskip

(c) This follows immediately from (b) and (a).
\end{proof}

\begin{theorem}
\label{s7:t9}
The following are equivalent:
\begin{enumerate}
\item $\mathbf{PS}_{0}$;

\item $\mathbf{PS_{0}(I)}$;

\item $\mathbf{CUC}\wedge\mathbf{PS_{0}(II)}$;

\item $\mathbf{CUC}\wedge\mathbf{PS}_{0}(\mathbf{C})$.
\end{enumerate}
\end{theorem}

\begin{proof}
That (1), (2) and (3) are equivalent follows from Theorem \ref{s3:t9} taken together with Theorem \ref{s7:t4}(a)--(b). Since $\mathbf{CUC}$ implies $\mathbf{{CAC}_{fin}}$, it follows from Lemma \ref{s7:l8}(a) that (3) and (4) are also equivalent.
\end{proof}

\begin{theorem}
\label{s7:t10}
It is relatively consistent with $\mathbf{ZF}$ that there exist an uncountable set $U$ and families: 
\begin{enumerate}
\item[(A)] $\mathcal{A}=\{A_{n}:n\in\omega\}\subseteq\mathcal{P}([U]^{<\omega})$ satisfying conditions (c) and (e) of Definition \ref{s7:d7}((2), (3)) (and thus also satisfying condition (a) of Definition \ref{s7:d7}(1)); 

\item[(B)] $\mathcal{B}=\{B_{n}:n\in\omega\}$ and $\mathcal{B}'=\{B_{n}':n\in\omega\}$ satisfying, respectively, conditions (d) and (f) of Definition \ref{s7:d7}((2), (3)) and being associated with $\mathcal{A}$; 

\item[(C)] $\mathcal{C}=\{C_{n}:n\in\omega\}$ satisfying condition (b) of Definition \ref{s7:d7}(1) and being associated with $\mathcal{A}$;

\item[(D)] $\mathcal{A}'=\{A_{n}':n\in\omega\}\subseteq\mathcal{P}([U]^{<\omega})$ satisfying condition (a) of Definition \ref{s7:d7}(1) and, for every $n\in\omega$, $|A_{n}'|=\aleph_{0}$; 

\item[(E)] $\mathcal{C}'=\{C_{n}':n\in\omega\}$ satisfying condition (b) of Definition \ref{s7:d7}(1) and being associated with $\mathcal{A}'$
\end{enumerate}
such that:
\begin{enumerate}
\item[(F)] $\mathbf{PS_{0}(C)}$ holds for $\langle U,\mathcal{A},\mathcal{B}\rangle$;

\item[(G)] $\mathbf{PS_{0}(D)}$ holds for $\langle U,\mathcal{A},\mathcal{B}'\rangle$;

\item[(H)] $\mathbf{PS_{0}(II)}$ fails for $\langle U,\mathcal{A},\mathcal{C}\rangle$;

\item[(I)] $\mathbf{PS_{0}(II)}$ fails for $\langle U,\mathcal{A}',\mathcal{C}'\rangle$. 
\end{enumerate}
\smallskip

In particular, $\mathbf{PS_{0}(II)}$ is unprovable in $\mathbf{ZF}$.
\end{theorem}

\begin{proof}
We first construct a \textbf{new permutation model}. We start with a model $M$ of $\mathbf{ZFA+AC}$ with a set $U$ of atoms which is a denumerable disjoint union $$U=\bigcup\{U_{n}:n\in\omega\},$$ where, for every $n\in\omega$, $U_{n}$ is a denumerable disjoint union $$U_{n}=\bigcup\{U_{n,j}:j\in\omega\},$$ where, for all $j\in\omega$, $$|U_{n,j}|=j+1.$$ Let $G$ be the group of all permutations $\phi$ of $A$ such that, for all $n,j\in\omega$, $\phi(U_{n,j})=U_{n,j}$. Let $\mathcal{F}$ be the filter of subgroups of $G$ which is generated by the pointwise stabilizers $\fix_{G}(E)$, where $E\subseteq U$ is such that:
\begin{enumerate}
\item $E\subseteq\bigcup\{U_{n}:n\in N\}$ for some $N\in [\omega]^{<\omega}$; and
\medskip

\item $\mathrm{sup}\{|E\cap U_{n,j}|:n\in N,j\in\omega\}$ is finite.  
\end{enumerate}
\medskip

It is fairly easy to verify that $\mathcal{F}$ is a normal filter on $G$. Let $\mathcal{V}$ be the permutation model which is determined by $M$, $G$ and $\mathcal{F}$. If $x\in\mathcal{V}$, then there exists a set $E\subseteq A$, defined as above, such that $\fix_{G}(E)\subseteq\sym_{G}(x)$. Under these circumstances, we call such a set $E$ a \emph{support}.

Using standard Fraenkel--Mostowski techniques, it can be easily verified that
\[
(\forall n\in\omega)(U_{n}\text{ is uncountable in $\mathcal{V}$}),
\] 
and thus
$$U\text{ is uncountable in $\mathcal{V}$}.$$

Furthermore, by the definition of supports, we obtain that
\begin{equation}
\label{eq:count_sub_U}
([U]^{\leq\omega})^{\mathcal{V}}=\big\{E\cup F:E\text{ is a support and $(\exists N\in [\omega\times\omega]^{<\omega})(F=\bigcup\{U_{n,j}:\langle n,j\rangle\in N)$}\})\big\}.
\end{equation}
We take the liberty to leave the verifications of the above observations to the readers as an easy exercise.

For every $n\in\omega$, we let
$$A_{n}=\big\{x:(\exists j\in\omega)(x\text{ is a choice set for $\{U_{n,i}:i\in j+1\}$)}\big\}\cup\big\{U_{n,j}:j\in\omega\big\}.$$
Since, for all $\phi\in G$ and $n,j\in\omega$, $\phi(U_{n,j})=U_{n,j}$, it follows that, for every $n\in\omega$, $\sym_{G}(A_{n})=G$, so $A_{n}\in\mathcal{V}$ for all $n\in\omega$. Moreover, for every $n\in\omega$, $A_{n}$ is an uncountable subset of $[U]^{<\omega}$ in $\mathcal{V}$.

For every $n\in\omega$, let
\begin{align*}
B_{n}&=\big\{\mathcal{Z}:\mathcal{Z}\text{ is a countable partition of $A_{n}$ such that }(\forall Z\in\mathcal{Z})(\exists x\in [\bigcup Z]^{<\omega})
(\forall z\in Z)(x\cap z\neq\emptyset)\big\},\\
B_{n}'&=\big\{\mathcal{Z}:\mathcal{Z}\text{ is a denumerable partition of $A_{n}$ such that }(\forall Z\in\mathcal{Z})(\exists x\in [\bigcup Z]^{<\omega})
(\forall z\in Z)(x\cap z\neq\emptyset)\big\}.
\end{align*}
As with the family $A_{n}$, $n\in\omega$, it is easy to see that, for every $n\in\omega$, $B_{n},B_{n}'\in\mathcal{V}$. Furthermore, for every $n\in\omega$, $$B_n=B_{n}'\neq\emptyset.$$ Indeed, fix $n\in\omega$. That $B_{n}=B_{n}'$ is reasonably clear. For every $j\in\omega$, define
$$Z_{n,j}=\big\{x:x\text{ is a choice set for $\{U_{n,i}:i\in j+1\}$}\big\}\cup\big\{U_{n,j}\big\}.$$
Define
$$\mathcal{Z}_{n}=\{Z_{n,j}:j\in\omega\}.$$
Clearly, $\mathcal{Z}_{n}\in\mathcal{V}$ and $\mathcal{Z}_{n}$ is denumerable in $\mathcal{V}$, since, for every $j\in\omega$, $\sym_{G}(Z_{n,j})=G$. Moreover, $\mathcal{Z}_{n}$ is a partition of $A_{n}$ such that, for every $j\in\omega$, $\bigcup Z_{n,j}$ is a finite set which meets non-trivially every member of $Z_{n,j}$. Therefore, $\mathcal{Z}_{n}\in B_{n}$. 
Now, define
$$H:=\{\langle B_{n},\mathcal{Z}_{n}\rangle:n\in\omega\}.$$

By the above arguments, we have $H\in\mathcal{V}$ and $H$ is a choice function for $\{B_{n}:n\in\omega\}$. Therefore, in $\mathcal{V}$, $\mathbf{PS_{0}(C)}$ and $\mathbf{PS_{0}(D)}$ hold for the triplet $\langle U,\{A_{n}:n\in\omega\},\{B_{n}:n\in\omega\}\rangle$. 
\smallskip

For every $n\in\omega$, we let
\[
C_{n}=\{x\in ([U]^{\leq\omega})^{\mathcal{V}}: (\forall z\in A_n)(x\cap z\neq\emptyset)\}.
\]

Then, for every $n\in\omega$, $C_{n}\in\mathcal{V}$, since $\sym_{G}(C_{n})=G$ for all $n\in\omega$. Furthermore, for every $n\in\omega$, $C_{n}\neq\emptyset$. Indeed, fix $n\in\omega$. Then, for any set $E\subseteq U_{n}$ which is a choice set for $\{U_{n,j}:j\in\omega\}$ (and thus is a support), we have that $E\cup U_{n,0}\in ([U]^{\leq\omega})^{\mathcal{V}}$ according to (\ref{eq:count_sub_U}). It is straightforward to see that $E\cup U_{n,0}$ meets non-trivially every member of $A_{n}$, so $E\cup U_{n,0}\in C_{n}$. Therefore, $C_{n}\neq\emptyset$, as required.

The family 
$$\mathcal{C}:=\{C_{n}:n\in\omega\},$$ which is denumerable in $\mathcal{V}$ (since $\sym_{G}(C_{n})=G$ for all $n\in\omega$), has no choice function in $\mathcal{V}$. Assume the contrary, and let $f\in\mathcal{V}$ be a choice function for $\mathcal{C}$. Let $E$ be a support of $f$. There exists $N\in[\omega]^{<\omega}$ such that $E\subseteq \bigcup\{U_{n}:n\in N\}$ and $\mathrm{sup}\{|E\cap U_{n,j}|:n\in N,j\in\omega\}$ is finite. Let $m=\max(N)+1$. By the following two facts:  
\begin{enumerate}
\item[(i)] $f(C_{m})\in ([U]^{\leq\omega})^{\mathcal{V}}$;

\item[(ii)] for all $j\in\omega$, $|U_{m,j}|=j+1$, 
\end{enumerate}
it readily follows that, for some $j\in\omega$, $U_{m,j}\setminus f(C_{m})\neq\emptyset$. Let $a\in U_{m,j}$ and $b\in f(C_{m})$. Consider the transposition $\phi=(a,b)$. Clearly, $\phi\in\fix_{G}(E)$, so $\phi(f)=f$, since $E$ is a support of $f$. However, by the choices of $a$ and $b$, $\phi(f(C_{m}))\neq f(C_{m})$, contradicting the fact that $E$ is a support of $f$. Therefore, $\mathbf{PS_{0}(II)}$ is false in $\mathcal{V}$ for $\langle U,\{A_{n}:n\in\omega\},\{C_{n}:n\in\omega\}\rangle$.
\medskip

Now, for every $n\in\omega$, we let
\begin{align*}
A_{n}'&=\{U_{n,j}:j\in\omega\},\\
C_{n}'&=\{x\in ([U]^{\leq\omega})^{\mathcal{V}}: (\forall z\in A_{n}')(x\cap z\neq\emptyset)\}.
\end{align*}
 
We have the following:
\begin{enumerate}
\item[(iii)] $\{A_{n}':n\in\omega\}$ and $\{C_{n}':n\in\omega\}$ are denumerable sets in $\mathcal{V}$, and, for every $n\in\omega$, $|A_{n}'|=\aleph_{0}$ in $\mathcal{V}$;

\item[(iv)] for every $n\in\omega$, $C_{n}'\neq\emptyset$; 

\item[(v)] the family $\{C_{n}':n\in\omega\}$ has no choice function in $\mathcal{V}$. 
\end{enumerate}

That (iv) and (v) hold, can be proved by using almost identical arguments to the ones for the proof of the first part that, for every $n\in\omega$, $C_{n}\neq\emptyset$, and that the family $\mathcal{C}$ has no choice function in $\mathcal{V}$; so this is left to the readers. We may thus conclude that $\mathbf{PS_{0}(II)}$ is false in $\mathcal{V}$ for $\langle U,\{A_{n}':n\in\omega\},\{C_{n}':n\in\omega\}\rangle$.
\smallskip

Since the existential formula of the statement of this theorem is a boundable statement and has a permutation model (namely the model $\mathcal{V}$), it follows from the Jech--Sochor First Embedding Theorem (see \cite[Theorem 6.1]{j}) that it also has a symmetric model of $\mathbf{ZF}$. Hence, $\mathbf{PS_{0}(II)}$ is unprovable in $\mathbf{ZF}$, finishing the proof of the theorem. 
\end{proof}

\begin{remark}
\label{s7:r11}
Note that in the model $\mathcal{V}$ of the proof of Theorem \ref{s7:t10}, we have $\{U_{n,j}:n,j\in\omega\}$ is a denumerable partition of $U$ into finite sets, but $U$ is uncountable in $\mathcal{V}$ and, therefore, $[U]^{\leq\omega}$ is not a $\sigma$-ideal. This, together with Theorem \ref{s3:t6}(i), implies that $\mathbf{S}(U,[U]^{\leq\omega})$ is not a $P$-space in $\mathcal{V}$.
\end{remark}

\begin{theorem}
\label{s7:t12}
The following hold:
\begin{enumerate}
\item $\mathbf{CUC}\wedge\mathbf{PS}_{0}(\mathbf{D})$ implies $\mathbf{PS}_{0}$.

\item $\mathbf{CUC}$ does not imply $\mathbf{PS}_{0}(\mathbf{D})$ in $\mathbf{ZF}$. In particular, $\mathbf{PS}_{0}(\mathbf{D})$ is unprovable in $\mathbf{ZF}$.
\end{enumerate}
\end{theorem}

\begin{proof}
(1) This readily follows from Lemma \ref{s7:l8}(b) and Theorem \ref{s7:t9}.
\smallskip

(2) We first show independence in $\mathbf{ZFA}$. To this end, we will use model $\mathcal{N}17$ from \cite{hr}. This model originally appeared in Brunner and Howard \cite{Brunner-Howard-1992} as a member of a class of permutation models defined therein. The description of $\mathcal{N}17$ that is given below is from Howard and Tachtsis \cite{ht}; it is different from, but equivalent to, the description given in \cite{Brunner-Howard-1992} and a correction of the corresponding one in \cite{hr}. 

We start with a model $M$ of $\mathbf{ZFA}+\mathbf{AC}$ with a set $U$ of atoms which is a denumerable disjoint union $U=\bigcup_{n\in\omega}U_{n}$, where $|U_{n}|=\aleph_{1}$ for all $n\in\omega$. Let 

\[
G=\Big{\{}\phi\in \sym(U): (\forall n\in\omega)(\phi(U_{n})=U_{n})\wedge(\exists J\subseteq\omega)(|J|{<}\aleph_{0} \text{ and $\phi$ is the identity map on $\bigcup_{i\in\omega\setminus J}U_{i}$})\Big{\}}.
\]

\noindent Let $\mathcal{F}$ be the filter of subgroups of $G$ generated by the subgroups $\fix_{G}(E)$, where $E=\bigcup_{i\in I}U_{i}$ for some finite $I\subseteq \omega$; $\mathcal{F}$ is a normal filter on $G$. 
$\mathcal{N}17$ is the permutation model determined by $M$, $G$ and $\mathcal{F}$. 

For every $x\in\mathcal{N}17$, there is a finite $I\subseteq\omega$ such that $\fix_{G}(\bigcup_{i\in I}U_{i})\subseteq\sym_{G}(x)$; any such finite union $\bigcup_{i\in I}U_{i}$ will be called a \emph{support} of $x$.

In \cite{Brunner-Howard-1992}, it was shown that
$$\mathcal{N}17\models\mathbf{CUC}.$$
(Also see \cite[Remark 6.29((1), (2))]{ht}.)

For every $n\in\omega$, we let $$X_{n}=\{n\}\cup U_{n},$$ and we also let
$$X=\bigcup_{n\in\omega}X_{n}.$$
Clearly, $X\in\mathcal{N}17$ and $X$ is uncountable in $\mathcal{N}17$. For every $n\in\omega$, let
\begin{equation}
\label{equat:A_n}
A_{n}=\big\{\{n,x\}:x\in U_{n}\big\}.
\end{equation}
Then, for every $n\in\omega$, $A_{n}$ is an uncountable family of finite subsets of $X$ (which is clearly in $\mathcal{N}17$) and the family $\{A_{n}:n\in\omega\}$ is denumerable in $\mathcal{N}17$.

For every $n\in\omega$, let $B_{n}$ be the family given by (f) of Definition \ref{s7:d7}(3), which is associated with the family $A_{n}$ defined in (\ref{equat:A_n}). We have the following facts about $B_{n}$, where $n\in\omega$.
\begin{enumerate}
\item[(a)] For every $n\in\omega$, $B_{n}\in\mathcal{N}17$. Fix $n\in\omega$. Then $\sym_{G}(B_{n})=G\in\mathcal{F}$. Indeed, let $\phi\in G$ and $\mathcal{Z}\in B_{n}$. We argue that $\phi(\mathcal{Z})\in B_{n}$. First, since $\mathcal{Z}$ is a denumerable partition of $A_{n}$ and $\phi\upharpoonright A_{n}$ is a permutation of $A_{n}$, it readily follows that $\phi(\mathcal{Z})$ is a denumerable partition of $A_{n}$. Second, since $\phi$ fixes every pure set (that is, a set for which its transitive closure contains no atoms), we have that, for all $Z\in\mathcal{Z}$ and $z\in Z$, $n\in\phi(z)\in\phi(Z)$. Therefore, $\phi(\mathcal{Z})\in B_{n}$. Since $\phi\in G$ and $\mathcal{Z}\in B_{n}$ were arbitrary, we have shown that, for any $\phi\in G$, $\phi(B_{n})\subseteq B_{n}$, which readily gives $\sym_{G}(B_{n})=G$. 

\item[(b)] For every $n\in\omega$, $B_{n}\neq\emptyset$ (in fact, $B_{n}$ is uncountable). This is straightforward; simply, consider any denumerable partition of $U_{n}$ into $\aleph_{1}$-sized sets (and note that any such partition is in $\mathcal{N}17$, since it has $U_n$ as a support), which readily yields an element of $B_{n}$ in $\mathcal{N}17$.

\item[(c)] The family $\mathcal{B}:=\{B_{n}:n\in\omega\}$ is in $\mathcal{N}17$ and is denumerable in $\mathcal{N}17$. This follows from (a).

\item[(d)] The family $\mathcal{B}$ (of (c)) has no choice function in $\mathcal{N}17$; in fact, not even a partial choice function in $\mathcal{N}17$. Assume on the contrary, that, in $\mathcal{N}17$, there exists an infinite family $\mathcal{C}\subseteq\mathcal{B}$ with a choice function, $f$ say. Since $f\in\mathcal{N}17$, let $n_{0}\in\omega$ be such that the set $E:=\bigcup\{U_{i}:i\leq n_{0}\}$ is a support of $f$; and note that, in view of the argument in (a), $E$ is also a support of every member of $\mathcal{C}$. Since $\mathcal{C}$ is infinite, there exists an integer $n>n_{0}$ such that $B_{n}\in\mathcal{C}$. Since $f(B_{n})$ is a denumerable partition of the uncountable set $A_{n}$, there exist $Z\in f(B_{n})$ with $|Z|\geq 2$ and $b\in U_{n}$ with $\{n,b\}\not\in Z$. Let $a\in U_{n}$ be such that $\{n,a\}\in Z$. Consider the transposition $\phi=(a,b)$. Then $\phi\in\fix_{G}(E)$, and thus $\phi(f)=f$. Since $Z\in f(B_{n})$ and $E$ is a support of $f$, we have $\phi(Z)\in\phi(f(B_{n}))=f(B_{n})$. Furthermore, since $\{n,b\}=\phi(\{n,a\})\in\phi(Z)\setminus Z$, we get $\phi(Z)\neq Z$. As $f(B_{n})$ is a partition of $A_{n}$, $\phi(Z)\cap Z=\emptyset$, which is a contradiction, since $|Z|\geq 2$ and $\phi$ moves only the element $\{n,a\}$ of $Z$.  
Therefore, $\mathcal{B}$ has no partial choice function in $\mathcal{N}17$, as required.
\end{enumerate}

By facts (a)--(d), we obtain 
$$\mathcal{N}17\models\neg\mathbf{PS}_{0}(\mathbf{D}).$$ 

Now, consider the following statement:
$$\Sigma=\mathbf{CUC}\wedge\neg\mathbf{PS}_{0}(\mathbf{D}).$$
Since $\Sigma$ is a conjunction of injectively boundable statements ($\mathbf{CUC}$ is injectively boundable---see \cite[Note 103]{hr}---and $\neg\mathbf{PS}_{0}(\mathbf{D})$ is boundable, and thus injectively boundable) which has a permutation model (namely $\mathcal{N}17$), it follows from a transfer theorem of Pincus (see \cite[Note 103]{hr}) that $\Sigma$ has a $\mathbf{ZF}$-model. This completes the proof of (2) and of the theorem.  
\end{proof}

\begin{remark}
\label{s7:r13} 
In the proof of Theorem \ref{thm:PS0(II)_not_PS0}, it will be shown that the Second Fraenkel Model $\mathcal{N}2$ of \cite{hr} is a model for  $\mathbf{PS}_0(\mathbf{II})\wedge\mathbf{PS}_0(\mathbf{C})\wedge\neg\mathbf{{CAC}_{fin}}\wedge\neg\mathbf{PS}_0$.
\end{remark}

\section{The statement ``$(\forall X)(\mathbf{CAC}(X)\rightarrow \mathbf{CAC}(X^{\omega}))$'' and related topics on $P$-spaces}
\label{s8}

We know from Theorem \ref{s6:t1} that $\mathbf{CAC}$ is equivalent to the statement that, for every infinite set $X$, the spaces $\mathbf{S}_{\mathcal{P}}(X, [X]^{\leq\omega})$ and $2^X[[X]^{\leq\omega}]$ are $P$-spaces. In Theorem \ref{s8:t2} below, we show that, for every uncountable set $X$, $\mathbf{CAC}(X^{\omega})$ is sufficient for $\mathbf{S}_{\mathcal{P}}(X, [X]^{\leq\omega})$, $\mathbf{S}(X, [X]^{\leq\omega})$ and $2^X[[X]^{\leq\omega}]$ to be $P$-spaces. This makes us interested in the following questions:

\begin{question}
\label{s8:q1}
\begin{enumerate}
\item If $X$ is an uncountable set, is $\mathbf{CAC}(X)$ sufficient for $\mathbf{S}(X, [X]^{\leq\omega})$ to be a $P$-space?
\item Does the statement ``$(\forall X)(\mathbf{CAC}(X)\rightarrow \mathbf{CAC}(X^{\omega}))$'' imply $\mathbf{PS}_0$?
\end{enumerate}
\end{question}

In this section, among other things, we answer Question \ref{s8:q1}(1) by showing that it is relatively consistent with $\mathbf{ZF}$ that there exists an uncountable set $X$ such that $\mathbf{CAC}(X)$ holds, but $\mathbf{S}(X, [X]^{\leq\omega})$ is not a $P$-space. We show that Question \ref{s8:q1}(2) has a negative answer in $\mathbf{ZFA}$ by proving that, in the Second Fraenkel Model $\mathcal{N}2$, ``$(\forall X)(\mathbf{CAC}(X)\rightarrow \mathbf{CAC}(X^{\omega}))$'' is true, but $\mathbf{PS}_0$ is false. We also show that $\mathcal{N}2$ is a model for $\mathbf{PS}_0(\mathbf{II})\wedge\mathbf{PS_{0}(C)}$. 

\begin{theorem}
\label{s8:t2}
For every uncountable set $X$, the following conditions are satisfied.
 \begin{enumerate} 
\item[(i)] $\mathbf{CAC}(X^{\omega})$ implies that  $\mathbf{S}(X, [X]^{\leq\omega})$, $\mathbf{S}_{\mathcal{P}}(X, [X]^{\leq\omega})$ and $2^X[[X]^{\leq\omega}]$ are all $P$-spaces.
\item[(ii)] If $|X|=|X^{\omega}|$, then $\mathbf{CAC}(X)$ implies that $\mathbf{S}(X, [X]^{\leq\omega})$, $\mathbf{S}_{\mathcal{P}}(X, [X]^{\leq\omega})$ and $2^X[[X]^{\leq\omega}]$ are all $P$-spaces.
\item[(iii)] If $\mathbf{S}(X, [X]^{\leq\omega})$ is a $P$-space, then $X$ cannot be expressed as a countable union of countable sets.
\end{enumerate}
\end{theorem}
\begin{proof}
Let $X$ be an uncountable set and let $\mathcal{Z}=[X]^{\leq\omega}$. Clearly $\mathcal{Z}$ is a bornology on $X$ such that $X\notin\mathcal{Z}$. 
\smallskip

(i) Assume that $\mathbf{CAC}(X^{\omega})$ holds. In view of Theorem \ref{s4:t8}(ii)--(iii), it suffices to prove that the space $\mathbf{S}_{\mathcal{P}}(X, \mathcal{Z})$ is a $P$-space.

Suppose that, for every $n\in\omega$, $O_n\in \tau_{\mathcal{P}}[\mathcal{Z}]$ and $\bigcap_{n\in\omega}O_n\neq\emptyset$. We fix $x\in\bigcap_{n\in\omega}O_n$. For every $n\in\omega$, let $A_n=\{z\in\mathcal{Z}(x): B_{x,z}^{\mathcal{P}}\subseteq O_n\}$. It follows from the definition of $\tau_{\mathcal{P}}[\mathcal{Z}]$ that, for every $n\in\omega$, $A_n\neq\emptyset$. For every $n\in\omega$, let $\Psi_n=\{\psi\in X^{\omega}: \psi[\omega]\in A_n\}$. By $\mathbf{CAC}(X^{\omega})$, the family $\{\Psi_n: n\in\omega\}$ has a choice function. Therefore, there exists a family $\{\psi_n: n\in\omega\}$ such that, for every $n\in\omega$, $\psi_n\in \Psi_n$. Let $z=\bigcup_{n\in\omega}\psi_n[\omega]$. Using the family $\{\psi_n: n\in\omega\}$, it is not hard to define in $\mathbf{ZF}$ a surjection $g:\omega\to z$. Therefore, $z\in\mathcal{Z}$. One can easily observe that $z\cap x=\emptyset$ and $B_{x,z}^{\mathcal{P}}\subseteq\bigcap_{n\in\omega}O_n$. This shows that $\mathbf{S}_{\mathcal{P}}(X, \mathcal{Z})$ is a $P$-space. Hence (i) holds.
\smallskip

(ii) We assume that $|X|=|X^{\omega}|$ and that $\mathbf{CAC}(X)$ holds. Since the sets $X$ and $X^{\omega}$ are equipotent, it follows from $\mathbf{CAC}(X)$ that $\mathbf{CAC}(X^{\omega})$ holds, so (ii) is a consequence of (i).
\smallskip

(iii) Suppose $X=\bigcup_{n\in\omega}C_n$ where, for every $n\in\omega$, $|C_n|\leq\aleph_{0}$. Since $X$ is uncountable, $\mathcal{Z}$ is not a $\sigma$-ideal. It follows from Theorem \ref{s3:t6}(i) that $\mathbf{S}(X, \mathcal{Z})$ is not a $P$-space.
\end{proof}

\begin{remark}
\label{s8:r3}
Assume that $X$ is an infinite set for which $\mathbf{CAC}(X^{\omega})$ holds, and let $\mathcal{Z}=[X]^{\leq\omega}$. That $\mathbf{S}_{\mathcal{P}}(X, [X]^{\leq\omega})$ is a $P$-space  can also be deduced from Theorem \ref{s3:t9}. Furthermore, without involving Theorem \ref{s4:t8}(ii), one can prove directly that $2^X[\mathcal{Z}]$ is a $P$-space by giving the following arguments.

Suppose that, for every $n\in\omega$, $O_n\in \tau_2[\mathcal{Z}]$ and $\bigcap_{n\in\omega}O_n\neq\emptyset$. Fix $f\in\bigcap_{n\in\omega}O_n$. For every $n\in\omega$, let $A_n=\{z\in\mathcal{Z}\setminus\{\emptyset\}: [f\upharpoonright z]_X\subseteq O_n\}$. It follows from the definition of $\tau_2[\mathcal{Z}]$ that, for every $n\in\omega$, $A_n\neq\emptyset$. For every $n\in\omega$, let $\Psi_n=\{\psi\in X^{\omega}: \psi[\omega]\in A_n\}$. By $\mathbf{CAC}(X^{\omega})$, the family $\{\Psi_n: n\in\omega\}$ has a choice function. Therefore, there exists a family $\{\psi_n: n\in\omega\}$ such that, for every $n\in\omega$, $\psi_n\in \Psi_n$. Let $z=\bigcup_{n\in\omega}\psi_n[\omega]$. Using the family $\{\psi_n: n\in\omega\}$, it is easy to define in $\mathbf{ZF}$ a surjection $g:\omega\to z$. Therefore, $z\in\mathcal{Z}$. One can observe that $z\neq\emptyset$ and $[f\upharpoonright z]_X\subseteq\bigcap_{n\in\omega}O_n$. This shows that $2^X[\mathcal{Z}]$ is a $P$-space. 
\end{remark}

\begin{corollary}
\label{s8:c4}
$\mathbf{CAC}(\mathbb{R})$ implies that $\mathbb{R}$ admits a topology $\tau$ such that $\langle\mathbb{R}, \tau\rangle$ is a Hausdorff, zero-dimensional, crowded $P$-space.
\end{corollary}
\begin{proof}
Assume $\mathbf{CAC}(\mathbb{R})$. Since $|\mathbb{R}|=|\mathbb{R}^{\omega}|$, we deduce from Theorem \ref{s8:t2}(ii) and Proposition \ref{s4:p16}(b) that there exists a topology $\tau$ on $[\mathbb{R}]^{<\omega}$ such that the space $\langle [\mathbb{R}]^{<\omega}, \tau\rangle$ is a Hausdorff, zero-dimensional, crowded $P$-space. The conclusion now follows from the fact that the sets $\mathbb{R}$ and $[\mathbb{R}]^{<\omega}$ are equipotent.
\end{proof}

In Theorem \ref{thm:Models_M1_N3}(2) below, we will non-trivially show that the implication in Corollary \ref{s8:c4} is \textbf{not reversible in} $\mathbf{ZF}$.

\begin{remark}
\label{s8:r5}
(a) By a slight modification of the proof of item (ii) of Theorem \ref{s8:t2} one can obtain the following result in $\mathbf{ZF}$: For every infinite set $X$ such that $[X]^{\leq\omega}$ is a $\sigma$-ideal, $\mathbf{CMC}(X^{\omega})$ implies that the spaces $\mathbf{S}(X, [X]^{\leq\omega})$, $\mathbf{S}_{\mathcal{P}}(X, [X]^{\leq\omega})$ and $2^X[[X]^{\leq\omega}]$ are all $P$-spaces.

(b) It is known that, in the Second Fraenkel Model, denoted by $\mathcal{N}2$ in \cite{hr}, $\mathbf{CMC}$ is true. However, the set $A$ of atoms of $\mathcal{N}2$ is a countable union of pairwise disjoint two-element sets, and $A$ is Dedekind-finite in $\mathcal{N}2$, so uncountable in $\mathcal{N}2$. Therefore, $[A]^{\leq\omega}$ is not a $\sigma$-ideal in $\mathcal{N}2$ and, by Theorems \ref{s3:t6}(i) and \ref{s4:t8}(ii)--(iii), the spaces $\mathbf{S}(A, [A]^{\leq\omega})$, $\mathbf{S}_{\mathcal{P}}(A, [A]^{\leq\omega})$ and $2^A[[A]^{\leq\omega}]$ are not $P$-spaces in $\mathcal{N}2$. Furthermore, it is known that $\mathbf{{CAC}_{fin}}$ is false in $\mathcal{N}2$ (see \cite[p. 178]{hr}). Therefore, it follows from Proposition \ref{s7:p3}, Theorem \ref{s7:t4} and Theorem \ref{s7:t5} that, for every $i\in\{0, 1, 2, 3, 4, 5\}$, $\mathbf{PS}_i$ is false in $\mathcal{N}2$.
\end{remark}

In regard to conditions (ii) and (iii) of Theorem \ref{s8:t2}, let us establish the following facts below.

\begin{proposition}
\label{s8:p6}
For every set $X$, $\mathbf{CAC}(X^{\omega})$ implies $\mathbf{CAC}(X)$.
\end{proposition}
\begin{proof}
 Fix an infinite set $X$. Suppose $\mathbf{CAC}(X^{\omega})$ is true. Let $\mathcal{A}=\{A_{n}:n\in\omega\}$ be a denumerable family of non-empty subsets of $X$. By $\mathbf{CAC}(X^{\omega})$, the denumerable family $\mathcal{B}=\{(A_{n})^{\omega}:n\in\omega\}$ has a choice function, $f$ say. Clearly, $h:=\{f(n)(0):n\in\omega\}$ is a choice function for $\mathcal{A}$, i.e. $\mathbf{CAC}(X)$ is true as required.
\end{proof}

\begin{theorem}
\label{s8:t7} 
$\mathbf{PW}$ implies that, for every infinite well-orderable set $X$, the following conditions are satisfied.
\begin{enumerate} 
\item[(a)] $\mathbf{CAC}(X^{\omega})$ holds.
\item[(b)] The spaces  $\mathbf{S}(X,[X]^{\leq\omega})$, $\mathbf{S}_{\mathcal{P}}(X, [X]^{\leq\omega})$ and $2^{X}[[X]^{\leq\omega}]$ are all $P$-spaces.
\item[(c)] The set $[X]^{\leq\omega}$ is a $\sigma$-ideal. 
\end{enumerate}
\end{theorem}

\begin{proof} 
Assume $\mathbf{PW}$ and fix an infinite well-orderable set $X$.  
\smallskip

(a) The set $\omega\times X$ is well orderable. It follows from $\mathbf{PW}$ that the set $\mathcal{P}(\omega\times X)$ is well orderable, so $\mathbf{CAC}(\mathcal{P}(\omega\times X))$ is true. Thus, since $X^{\omega}\subseteq\mathcal{P}(\omega\times X)$, we conclude that (a) holds.
\smallskip

(b) If $X$ is a denumerable set, then the spaces $\mathbf{S}(X,[X]^{\leq\omega})$, $\mathbf{S}_{\mathcal{P}}(X, [X]^{\leq\omega})$ and $2^{X}[[X]^{\leq\omega}]$ are all $P$-spaces by Proposition \ref{s4:p16}(c).

If $X$ is uncountable, then it follows from (a) and Theorem \ref{s8:t2}(i) that (b) holds.
\smallskip

(c) In view of Theorem \ref{s3:t6}(i), (b) implies (c).
\end{proof}

It is worth comparing the following corollary with Theorem \ref{s5:t5}.

\begin{corollary}
\label{s8:c8}
In every permutation model of $\mathbf{ZFA}$, it is true that, for every infinite well-orderable set $X$, conditions {\rm(a), (b)} and {\rm (c)} of Theorem \ref{s8:t7} are all fulfilled.
\end{corollary}
\begin{proof}
It is known that $\mathbf{PW}$ is true in every permutation model (see \cite{hr} or \cite[(4.2), p. 47]{j}). The conclusion now follows from Theorem \ref{s8:t7}.
\end{proof}

\begin{definition}
\label{s8:d9}
Let $\mathbf{\Pi}$ and $\mathbf{\Pi WO}$ be the following formulae, respectively:
\smallskip

$\mathbf{\Pi}$: There exists an infinite set $X$ such that $\mathbf{CAC}(X)$ holds, but the space $\mathbf{S}(X, [X]^{\leq\omega})$ is not a $P$-space.
\smallskip

$\mathbf{\Pi WO}$: There exists an infinite well-orderable set $X$ such that the space $\mathbf{S}(X, [X]^{\leq\omega})$ is not a $P$-space. 
\end{definition}

\begin{theorem}
\label{s8:t10}
The following hold:
\begin{enumerate}
\item[(a)] $\mathbf{\Pi WO}$ implies $\mathbf{\Pi}$.
\item[(b)] $\mathbf{\Pi}$ implies that the following formula is true:
\[ (\exists X)(\mathbf{CAC}(X)\wedge\neg\mathbf{CAC}(X^{\omega})).\]
\item[(c)] The formula $\mathbf{\Pi WO}$ is relatively consistent with $\mathbf{ZF}$. In consequence, also $\mathbf{\Pi}$ is relatively consistent with $\mathbf{ZF}$.
\item[(d)] It is relatively consistent with $\mathbf{ZF}$ that there exists a well-orderable set $X$ such that $\mathbf{CAC}(X^{\omega})$ is false.
\item[(e)] $\mathbf{\Pi WO}$ is false in every permutation model.
\end{enumerate}
\end{theorem}

\begin{proof}
(a) Since, for every well-orderable set $X$, $\mathbf{CAC}(X)$ holds, it is obvious that $\mathbf{\Pi WO}$ implies $\mathbf{\Pi}$.
\smallskip

(b) This follows from Theorem \ref{s8:t2}(i).
\smallskip

(c) We consider the Feferman--L\'{e}vy Model $\mathcal{M}9$ of \cite{hr}. We have noticed in the proof of Theorem \ref{s5:t5} that $\mathbf{S}(\omega_1, [\omega_1]^{\leq\omega})$ is not a $P$-space in $\mathcal{M}9$. Hence, $\mathcal{M}9$ is a $\mathbf{ZF}$-model in which $\mathbf{\Pi WO}$ is true. This, together with (a), shows that (c) holds.
\smallskip

(d) Since $\mathbf{S}(\omega_1, [\omega]^{\leq\omega})$ is not a $P$-space in $\mathcal{M}9$, it follows from Theorem \ref{s8:t2}(i) that  $\mathbf{CAC}(({\omega_1})^{\omega})$ is false in $\mathcal{M}9$. Hence (d) holds.
\smallskip

(e) This follows from Corollary \ref{s8:c8}.
\end{proof}

\begin{theorem}
\label{s8:t11}
The following hold:
\begin{enumerate}
\item[(i)] $(\forall X)(\mathbf{CAC}(X)\rightarrow\mathbf{CAC}(X^{\omega}))$ implies $\cf(\aleph_1)=\aleph_1 \wedge \mathbf{IWDI}$.

\item[(ii)] $\mathbf{IDI}$ does not imply $(\forall X)(\mathbf{CAC}(X)\rightarrow\mathbf{CAC}(X^{\omega}))$ in $\mathbf{ZF}$.

\item[(iii)] $\cf(\aleph_1)=\aleph_1$ does not imply $(\forall X)(\mathbf{CAC}(X)\rightarrow\mathbf{CAC}(X^{\omega}))$ in $\mathbf{ZFA}$.
\end{enumerate}
\end{theorem}

\begin{proof}
(i) Assume $(\forall X)(\mathbf{CAC}(X)\rightarrow\mathbf{CAC}(X^{\omega}))$. 

Since $\omega_1$ is a well-ordered set, $\mathbf{CAC}(\omega_1)$ holds. Therefore, by our assumption, $\mathbf{CAC}((\omega_1)^{\omega})$ holds. It follows from Theorems \ref{s8:t2}(i) and \ref{s5:t4} that $\cf(\aleph_1)=\aleph_1$.

Now, suppose $X$ is an infinite set. Aiming for a contradiction, we assume that the set $\mathcal{P}(X)$ is Dedekind-finite. Then, obviously, $\mathbf{CAC}(X)$ holds. Thus, by our hypothesis, $\mathbf{CAC}(X^{\omega})$ holds. For every $n\in\mathbb{N}$, let $$B_n=\{f\in X^{\omega}: f\upharpoonright n\text{ is an injection}\}.$$
\noindent Since $X$ is infinite, we have that, for every $n\in\mathbb{N}$, $B_n\neq\emptyset$. By $\mathbf{CAC}(X^{\omega})$, the family $\{B_n: n\in\mathbb{N}\}$ has a choice function, $f$ say. Using $f$, it is fairly easy to show that $X$ is Dedekind-infinite, so $\mathcal{P}(X)$ is Dedekind-infinite, a contradiction. Therefore, $X$ is weakly Dedekind-infinite.
\smallskip

(ii) We first show independence in $\mathbf{ZFA}$. We will use Model $\mathcal{N}18$ from \cite{hr}; $\mathcal{N}18$ was introduced by Howard in \cite{PH}. The description of the model is as follows.

We start with a model $M$ of $\mathbf{ZFA}+\mathbf{AC}$ with a denumerable set $A$ of atoms. We write $A$ as a disjoint union $A=\bigcup_{i\in\omega}A_{i}$, where each $A_{i}$ is denumerable. For each $i\in\omega$, let $F_{i}$ be a free ultrafilter on $A_{i}$, and let $I_{i}$ be the corresponding non-principal, maximal ideal, i.e. $I_{i}=\{A_{i}\setminus x:x\in F_{i}\}$. The group $G$ of permutations of $A$ is defined as follows:
\[
G=\{\phi\in\sym(A):(\forall i\in\omega)[\phi(A_{i})=A_{i}\text{ and }
(\exists u\in F_{i})(\phi\text{ fixes $u$ pointwise)]}\}.
\]
Let $\Gamma$ be the (normal) filter of subgroups of $G$ generated by the groups $G(g,e)$, where $e\in [\omega]^{<\omega}$ and $g\in(\mathcal{P}(A))^{\omega}$ is such that $(\forall i\in\omega)(g(i)\in I_{i})$, and $G(g,e)$ is defined by:
\[
G(g,e)=\{\phi\in G:(\forall i\in e)(\phi\text{ fixes $A_{i}$ pointwise) }
\text{and }
(\forall i\in\omega)(\phi\text{ fixes $g(i)$ pointwise)}\}.
\]
$\mathcal{N}18$ is the Fraenkel--Mostowski model determined by $M$, $G$ and $\Gamma$.

Tachtsis observed in \cite[Remark 6.9]{TachNDS} that
$$\mathcal{N}18\models\mathbf{IDI}.$$

\begin{claim}
\label{cl1:N18}
$\mathbf{S}(A, [A]^{\leq\omega})$ is not a $P$-space in $\mathcal{N}18$. Thus, by Theorem \ref{s8:t2}(i), $\mathbf{CAC}(A^{\omega})$ is false in $\mathcal{N}18$.
\end{claim}

\begin{proof}[Proof of claim]
In \cite{PH}, it was shown that in $\mathcal{N}18$, $A$ is a denumerable union of denumerable sets which is uncountable (and thus $\mathbf{CUC}$ is false in $\mathcal{N}18$). (The major result in \cite{PH} is that the weak choice principle ``Every denumerable family of denumerable sets has a choice function'' (\cite[Form 32]{hr}) is true in $\mathcal{N}18$.) By Theorem \ref{s3:t6}(i), $\mathbf{S}(A, [A]^{\leq\omega})$ is not a $P$-space in $\mathcal{N}18$.
\end{proof}

Next, we show that $\mathbf{CAC}(A)$ is true in $\mathcal{N}18$. We first prove a general lemma about the subsets of $A$ in $\mathcal{N}18$. We also recall the following (general) terminology. Let $M$ be a model of $\mathbf{ZFA}+\mathbf{AC}$, $x\in M$ and $H$ be a subgroup of a group $G$ of permutations of the set $A$ of atoms of $M$. (Every permutation of $A$ extends uniquely to an $\in$-automorphism of $M$ by $\in$-induction, and, for any $\phi\in G$, we identify $\phi$ with its unique extension.) We say that:
\begin{itemize}
\item $H$ \emph{fixes} $x$ if, for all $\phi\in H$, $\phi(x)=x$;
\item $H$ \emph{fixes $x$ pointwise} if, for all $\phi\in H$ and $y\in x$, $\phi(y)=y$.
\end{itemize}

\begin{lemma}
\label{lem:N18}
Let $T\in(\mathcal{P}(A))^{\mathcal{N}18}$. Let $e\in[\omega]^{<\omega}$ and $g\in(\mathcal{P}(A))^{\omega}$ be such that, for all $i\in\omega$, $g(i)\in I_{i}$, and $G(g,e)$ fixes $T$. Then the following holds:
\[
(\forall i\in\omega\setminus e)[((A_{i}\setminus g(i))\cap T\neq\emptyset)\rightarrow(A_{i}\setminus g(i)\subseteq T)].
\]
\end{lemma}

\begin{proof}
First, note that since, for every $n\in\omega$, $F_{n}$ is a free ultrafilter on $A_{n}$ and $g(n)\in I_{n}$, we have $A_{n}\setminus g(n)$ is infinite for all $n\in\omega$.

Now, aiming for a contradiction, suppose that there exists $i\in\omega\setminus e$ such that $(A_{i}\setminus g(i))\cap T\neq\emptyset$, but $A_{i}\setminus g(i)\nsubseteq T$. Choose an $a\in (A_{i}\setminus g(i))\cap T$ and a $b\in A_{i}\setminus (g(i)\cup T)$. 

We consider the permutation $\phi=(a,b)$ of $A$; that is, $\phi$ is the transposition which interchanges $a$ with $b$ and leaves all other atoms fixed. Then $\phi\in G$. Indeed, for every $n\in\omega$ with $n\neq i$, $\phi$ fixes $A_{n}$ pointwise and, since $F_{n}$ is a filter on $A_{n}$, $A_{n}\in F_{n}$. Furthermore, $\phi$ fixes $A_{i}\setminus\{a,b\}$ pointwise and, since $F_{i}$ is a free ultrafilter on $A_{i}$, $A_{i}\setminus\{a,b\}\in F_{i}$. Therefore, $\phi\in G$ and, clearly, $\phi\in G(g,e)$.
Since $G(g,e)$ fixes $T$ and $\phi\in G(g,e)$, we have $\phi(T)=T$. Then the following implications are true:
$$a\in T\rightarrow\phi(a)\in\phi(T)\rightarrow b\in T$$
This contradicts the fact that $b\not\in T$ and completes the proof of the lemma.
\end{proof}

\begin{claim}
\label{cl2:N18}
$\mathbf{CAC}(A)$ is true in $\mathcal{N}18$.
\end{claim}

\begin{proof}[Proof of claim]
In $\mathcal{N}18$, let $\mathcal{U}=\{U_{n}:n\in\omega\}$ be a denumerable family of non-empty subsets of $A$; the mapping $\omega\ni n\mapsto U_{n}$ is a bijection. There exist $e\in [\omega]^{<\omega}$ and $g\in(\mathcal{P}(A))^{\omega}$ such that $(\forall i\in\omega)(g(i)\in I_{i})$ and $G(g,e)$ fixes the mapping $\{\langle n,U_{n}\rangle: n\in\omega\}$ pointwise; that is, for all $\phi\in G(g,e)$ and $n\in\omega$, we have $\phi(\langle n,U_{n}\rangle)=\langle n,U_{n}\rangle$. It follows that, for all $n\in\omega$, $G(g,e)$ fixes $U_{n}$. 
 
For every $i\in\omega$, since $A_{i}\setminus g(i)$ is infinite (see the proof of Lemma \ref{lem:N18}), we choose an element
$$a_{i}\in A_{i}\setminus g(i).$$

We define an $h\in(\mathcal{P}(A))^{\omega}$ by
$$h(i)=g(i)\cup\{a_{i}\}\quad \text{for all $i\in\omega$}.$$
Note that, for every $i\in\omega$, $h(i)\in I_{i}$. Indeed, fix $i\in\omega$. We have $g(i)\in I_i$ and, since $F_i=\{A_i\setminus x: x\in I_i\}$ is a free ultrafilter on $A_i$, we also have $[A_i]^{<\omega}\subseteq I_i$. As $I_i$ is an ideal, we obtain $h(i)\in I_i$.

Since, for every $k\in\omega$, $G(g,e)$ fixes $U_{k}$, it follows from Lemma \ref{lem:N18} that
\begin{equation}
\label{eq:U_{k}}
(\forall k\in\omega)(\forall i\in\omega\setminus e)[((A_{i}\setminus g(i))\cap U_{k}\neq\emptyset)\rightarrow(A_{i}\setminus g(i)\subseteq U_{k})].
\end{equation}

We let
$$E=\left(\bigcup_{i\in e}A_{i}\right)\cup\left(\bigcup_{i\in\omega\setminus e}g(i)\right),$$
and we also let
$$S=\{k\in\omega: U_{k}\cap E\neq\emptyset\}.$$
\begin{enumerate}
\item[(a)] If $S\neq\emptyset$, then (in $M$) let $F$ be a choice function for the family
$$\mathcal{U}_{S}=\{U_{k}\cap E:k\in S\}.$$
We have $F\in\mathcal{N}18$, since $G(g,e)$ fixes $F$ pointwise.

\item[(b)] If $\omega\setminus S\neq\emptyset$, then, for every $k\in\omega\setminus S$, we let
$$i_{k}=\min\{i\in\omega\setminus e : (A_{i}\setminus g(i))\cap U_{k}\neq\emptyset\}.$$ 
By (\ref{eq:U_{k}}), we obtain
$$(\forall k\in\omega\setminus S)(a_{i_{k}}\in U_{k}).$$
\end{enumerate}

Now, we consider the following three cases for $S$.
\smallskip

\textit{Case 1.} $S\neq\emptyset$ and $S\neq\omega$. We define
$$H_{1}=\{\langle k,F(k)\rangle:k\in S\}\cup\{\langle k,a_{i_{k}}\rangle:k\in\omega\setminus S\}.$$
Then, by (a) and (b), $H_{1}$ is a choice function for $\mathcal{U}$ which is in $\mathcal{N}18$, since $G(h,e)$ fixes $H_{1}$ pointwise.
\smallskip

\textit{Case 2.} $S=\emptyset$. We define
$$H_{2}=\{\langle k,a_{i_{k}}\rangle:k\in\omega\}.$$
Then, by (b), $H_{2}$ is a choice function for $\mathcal{U}$ which is in $\mathcal{N}18$, since $G(h,e)$ fixes $H_{2}$ pointwise.
\smallskip

\textit{Case 3.} $S=\omega$. Then $\mathcal{U}_{S}=\{U_{k}\cap E:k\in \omega\}$ so, by (a), $F$ is a choice function for $\mathcal{U}$ which is in $\mathcal{N}18$.
\smallskip

By the above arguments, we conclude that $\mathbf{CAC}(A)$ is true in $\mathcal{N}18$, finishing the proof of the claim.
\end{proof}

Now, we consider the following statement:
$$\Psi=\mathbf{IDI}\wedge(\exists X)(\mathbf{CAC}(X)\wedge\neg\mathbf{CAC}(X^{\omega})).$$

As observed in \cite[Note 103]{hr}, $\mathbf{IDI}$ is an injectively boundable statement. Furthermore, ``$(\exists X)(\mathbf{CAC}(X)\wedge\neg\mathbf{CAC}(X^{\omega}))$'' is a boundable statement, and thus it is injectively boundable; for the definitions of the terms ``boundable'' and ``injectively boundable'' for formulae and statements, see \cite[Note 103]{hr}. Hence, $\Psi$ is a conjunction of injectively boundable statements and, since $\Psi$ has a permutation model (namely $\mathcal{N}18$), it follows from a transfer theorem of Pincus (see \cite[Note 103]{hr}) that $\Psi$ has a $\mathbf{ZF}$-model. This completes the proof of (ii).
\smallskip

(iii) Since $\cf(\aleph_{1})=\aleph_{1}$ is true in every permutation model (see \cite{hr}), it follows that $\mathcal{N}18\models\cf(\aleph_{1})=\aleph_{1}$. On the other hand, by the proof of part (ii), we know that $(\forall X)(\mathbf{CAC}(X)\rightarrow\mathbf{CAC}(X^{\omega}))$ is false in $\mathcal{N}18$. This yields the required independence result in $\mathbf{ZFA}$.   
\end{proof}

\begin{remark}
\label{s8:r15}
\begin{enumerate}
\item[(a)] Since, by Theorem \ref{s8:t10}(e), $\mathbf{\Pi WO}$ is false in every permutation model, it follows that $\mathbf{\Pi WO}$ is false in $\mathcal{N}18$. On the other hand, by the proof of Theorem \ref{s8:t11}(ii), we obtain $\mathcal{N}18\models\mathbf{\Pi}$. The latter two facts show that 
$$\mathbf{\Pi}\nrightarrow\mathbf{\Pi WO}\quad\text{in $\mathbf{ZFA}$}.$$

\item[(b)] To provide the readers with further insights, let us present a \textbf{new permutation model} which satisfies $\mathbf{\Pi}$.

 We start with a model $M$ of $\mathbf{ZFA}+\mathbf{AC}$ with a set $A$ of atoms which is a denumerable disjoint union $A=\bigcup_{n\in\omega}A_{n}$ where, for every $n\in\omega$, $A_{n}$ is denumerable. Let $G$ be the group of all permutations of $A$ which fix $A_{n}$ for all $n\in\omega$. Let $\mathcal{F}$ be the filter of subgroups of $G$ generated by the pointwise stabilizers $\fix_{G}(E)=\{\phi\in G:(\forall e\in E)(\phi(e)=e)\}$, where $E$ ranges over all subsets of $A$ of the following form: 
\begin{equation}
\label{eq:support}
E=\Big(\bigcup_{n\in F}A_{n}\Big)\cup C,\text{ $F\in [\omega]^{<\omega}$ and, for all $n\in\omega\setminus F$, $C\cap A_{n}\in [A_{n}]^{<\omega}$.}
\end{equation}
It is easy to verify that $\mathcal{F}$ is a normal filter on $G$. Let $\mathcal{N}$ be the permutation model determined by $M$, $G$, and $\mathcal{F}$. For every $x\in\mathcal{N}$, we have $\sym_{G}(x)\in\mathcal{F}$, where $\sym_{G}(x)=\{\phi\in G:\phi(x)=x\}$, and thus there exists a subset $E$ of $A$ of the form (\ref{eq:support}) such that $\fix_{G}(E)\subseteq\sym_{G}(x)$. We call such a set $E$ a \emph{support} of $x$.

Similarly to Model $\mathcal{N}18$ of the proof of Theorem \ref{s8:t11}(ii), we have that, in $\mathcal{N}$, $A$ is a denumerable union of denumerable sets which is uncountable. Therefore, $\mathbf{S}(A, [A]^{\leq\omega})$ is not a $P$-space in $\mathcal{N}$.

On the other hand, and although the setting of $\mathcal{N}$ is completely different from the one of Model $\mathcal{N}18$, the argument that $\mathbf{CAC}(A)$ is true in $\mathcal{N}$ uses fairly similar ideas to the ones employed for the proof of $\mathbf{CAC}(A)$ in $\mathcal{N}18$. We thus invite the interested readers to fill in the proof of $\mathbf{CAC}(A)$ in $\mathcal{N}$ (by making obvious modifications to the proofs of Lemma \ref{lem:N18} and of Claim \ref{cl2:N18}). 

We \emph{do not know} whether or not $\mathbf{IDI}$ is true in $\mathcal{N}$.  
\end{enumerate}
\end{remark}

The following questions are also interesting.

\begin{question}
\label{s8:q16}
\begin{enumerate}
\item Does $\cf(\aleph_1)=\aleph_1$ imply $\mathbf{CAC}((\omega_1)^{\omega})$?
\item Does the statement ``$(\forall X)(\mathbf{CAC}(X)\rightarrow\mathbf{CAC}(X^{\omega}))$'' imply $\mathbf{CAC}$ or $\mathbf{IDI}$?
\end{enumerate}
\end{question}

In Theorem \ref{s8:t20} below, we \textbf{answer Questions \ref{s8:q1}(2) and \ref{s8:q16}(2) in the negative in the setting of $\mathbf{ZFA}$}. We first establish the following theorem which provides a \textbf{general criterion} for the validity of $(\forall X)(\mathbf{CAC}(X)\rightarrow\mathbf{CAC}(X^{\omega}))$ and $\mathbf{PS_{0}(II)}$ in permutation models.

\begin{theorem}
\label{thm:general_criterion}
$[$General criterion for the validity of $(\forall X)(\mathbf{CAC}(X)\rightarrow\mathbf{CAC}(X^{\omega}))$ and $\mathbf{PS_{0}(II)}$ in permutation models$]$ Assume  $\mathcal{N}$ is a permutation model which is determined by a model $M$ of $\mathbf{ZFA}+\mathbf{AC}$, a group $G$ of permutations of the set $A$ of atoms, and a normal filter $\mathscr{F}$ on $G$ which is generated by some filter base $\mathscr{B}$ (of subgroups of $G$). If $\mathcal{N}$ satisfies the following two conditions:
\begin{enumerate}
\item[(a)] $(\forall x\in\mathcal{N})(\forall B\in\mathscr{B})(\forall \phi\in G)(B$ fixes $x\rightarrow B$ fixes $\phi(x))$;
\smallskip

\item[(b)] $(\forall x\in\mathcal{N})(\text{the $G$-orbit }\Orb(x)=\{\phi(x):\phi\in G\}$ of $x$ is countable in $\mathcal{N})$,
\end{enumerate}
then 
\[
\mathcal{N}\models(\forall X)(\mathbf{CAC}(X)\rightarrow\mathbf{CAC}(X^{\omega}))\wedge\mathbf{PS_{0}(II)}.
\] 
\end{theorem}

\begin{proof}
Let $X$ be an infinite set in $\mathcal{N}$ such that
$$\mathcal{N}\models\mathbf{CAC}(X).$$
Since $X\in\mathcal{N}$, there exists $B\in\mathscr{B}$ such that $B$ fixes $X$. Then, in $\mathcal{N}$, $X$ can be written as a well-orderable disjoint union of countable sets. Indeed, first note that
$$X=\bigcup\big\{\Orb_{B}(x):x\in X\big\},$$
where, for $x\in X$, $\Orb_{B}(x)$ is the $B$-orbit of $x$, i.e. $\Orb_{B}(x)=\{\phi(x):\phi\in B\}$. 
Second, the family $$\mathcal{U}=\{\Orb_{B}(x):x\in X\}$$ consists of pairwise disjoint subsets of $X$ and is well orderable in $\mathcal{N}$, since $B$ fixes every element of $\mathcal{U}$. Third, by condition (b), every member of $\mathcal{U}$ is a countable set in $\mathcal{N}$.

\begin{claim}
\label{claim2:s5_CAC_X}
For every countable set $\mathcal{V}$ in $\mathcal{N}$ with $\mathcal{V}\subseteq\mathcal{U}$, $\bigcup\mathcal{V}$ is countable in $\mathcal{N}$.
\end{claim}

\begin{proof}[Proof of claim]
Fix $\mathcal{V}\subseteq\mathcal{U}$ such that $\mathcal{V}$ is countable in $\mathcal{N}$. Without loss of generality, we assume that $\mathcal{V}\neq\emptyset$; otherwise, $\bigcup\mathcal{V}=\emptyset$, so $\bigcup\mathcal{V}$ is countable. 

Since $\mathcal{V}$ is a countable subset of $\mathcal{P}(X)\setminus\{\emptyset\}$ in $\mathcal{N}$ and $\mathbf{CAC}(X)$ is true in $\mathcal{N}$, $\mathcal{V}$ has a choice function in $\mathcal{N}$, $f$ say. Since $f\in\mathcal{N}$, there exists $F\in\mathscr{B}$ such that $F$ fixes $f$. We may assume that $F\subseteq B$. If $F\nsubseteq B$, then since $F,B\in\mathscr{B}$ and $\mathscr{B}$ is a filter base, there exists $I\in\mathscr{B}$ with $I\subseteq F\cap B$; clearly, $I\subseteq B$ and $I$ fixes $f$.  

Since $F\subseteq B$, $\mathcal{V}\subseteq\mathcal{U}$, and $B$ fixes every member of $\mathcal{U}$, it follows that, for all $\phi\in F$ and $V\in\mathcal{V}$, $\phi(V)=V$. Thus, since $f$ is a function and $F$ fixes $f$, we conclude that, for all $\phi\in F$ and $V\in\mathcal{V}$, $\phi(f(V))=f(V)$, i.e., for every $V\in\mathcal{V}$, $F$ fixes $f(V)$.

Since, for every $V\in\mathcal{V}$, $f(V)\in V\in\mathcal{U}$, it follows that, for every $V\in\mathcal{V}$, $V=\Orb_{B}(f(V))$. Thus, if $x\in V\in\mathcal{V}$, then $x\in \Orb_{B}(f(V))$ and, since $F$ fixes $f(V)$, it follows from condition (a) that $F$ fixes $x$. We derive that $F$ fixes every element of $\bigcup\mathcal{V}$. Therefore, $\bigcup\mathcal{V}$ is well orderable in $\mathcal{N}$.

Since $\bigcup\mathcal{V}$ is well orderable in $\mathcal{N}$ and countable in $M$ (because $\bigcup\mathcal{V}$ is a countable union of countable sets in $M$, and $M$ satisfies $\mathbf{AC}$), we conclude that $\bigcup\mathcal{V}$ is countable in $\mathcal{N}$. This completes the proof of the claim.   
\end{proof}

We are now ready to show that 
\begin{equation}
\label{eq:N_CACXom}
\mathcal{N}\models\mathbf{CAC}(X^{\omega}).
\end{equation} 
To this end, let $\mathcal{R}=\{R_{n}:n\in\omega\}$ be a family of non-empty subsets of $X^{\omega}$ which is denumerable in $\mathcal{N}$, and the mapping $\omega\ni n\mapsto R_{n}$ is a bijection in $\mathcal{N}$.

In $M$, let 
$$h=\{\langle n,f_{n}\rangle:n\in\omega\}$$
be a choice function for $\mathcal{R}$. By Claim \ref{claim2:s5_CAC_X}, it follows that, for every $n\in\omega$, $f_n\in\mathcal{N}$. To see this, fix $n\in\omega$. Since $B$ fixes $\omega\times X$, we have $\omega\times X\in\mathcal{N}$. Thus, since $f_n\subseteq \omega\times X$, and $\mathcal{N}$ is transitive, $f_n\subseteq\mathcal{N}$. Let $\mathcal{W}_n=\{U\in\mathcal{U}:U\cap f_{n}[\omega]\neq\emptyset\}$. Since $\mathcal{U}$ is a cover of $X$, and $f_{n}\in X^{\omega}$, we have $\mathcal{W}_n\neq\emptyset$. Furthermore, as $\dom(f_{n})=\omega$, it is reasonably clear that $\mathcal{W}_n$ is countable in $\mathcal{N}$. By Claim \ref{claim2:s5_CAC_X}, the set $\bigcup \mathcal{W}_n$ is countable in $\mathcal{N}$. There exists $K\in\mathscr{B}$ which fixes every element of $\bigcup\mathcal{W}_n$. Since $f_{n}[\omega]\subseteq\bigcup\mathcal{W}_n$, it follows that $K$ fixes every element of $f_{n}$. Therefore, $f_{n}\in\mathcal{N}$ as required.  

Our plan is to show that $h\in\mathcal{N}$; this will give us that (\ref{eq:N_CACXom}) is true. Since, for every $n\in\omega$, $f_{n}\in\mathcal{N}$, it suffices to show that $\fix_{G}(h)\in\mathscr{F}$. 
For every $n\in\omega$, we let 
$$\mathcal{S}_{n}=\{U\in\mathcal{U}:U\cap f_{n}[\omega]\neq\emptyset\}.$$
Similarly to the argument of the previous paragraph about $\mathcal{W}_n$, we have, for every $n\in\omega$, $\mathcal{S}_{n}\ne\emptyset$ and $\mathcal{S}_{n}$ is countable in $\mathcal{N}$. Thus, for every $n\in\omega$, $\mathcal{S}_{n}$ is countable in $M$. Put
$$\mathcal{T}=\bigcup_{n\in\omega}\mathcal{S}_{n}.$$
Since $\mathcal{T}$ is a countable union of countable sets in $M$, it follows that $\mathcal{T}$ is countable in $M$. Furthermore, since $B$ fixes every member of $\mathcal{T}$ (because $\mathcal{T}\subseteq\mathcal{U}$ and $B$ fixes every member of $\mathcal{U}$), it follows that $\mathcal{T}$ is well orderable in $\mathcal{N}$. This, together with the fact that $\mathcal{T}$ is countable in $M$, yields that $\mathcal{T}$ is countable in $\mathcal{N}$. Put
$$T=\bigcup\mathcal{T}.$$
Since $\mathcal{T}\subseteq\mathcal{U}$ and $\mathcal{T}$ is countable in $\mathcal{N}$, it follows from Claim \ref{claim2:s5_CAC_X} that $T$ is countable in $\mathcal{N}$. Moreover, as $$\bigcup_{n\in\omega}f_{n}[\omega]\subseteq T,$$ 
we conclude that $\bigcup_{n\in\omega}f_{n}[\omega]$ is countable in $\mathcal{N}$. This readily yields $\fix_{G}(h)\in\mathscr{F}$. Therefore, $h\in\mathcal{N}$, and so (\ref{eq:N_CACXom}) is true as required.
\vskip.1in

Now, we show that $\mathbf{PS_{0}(II)}$ is true in $\mathcal{N}$. In $\mathcal{N}$, let $X$ be an uncountable set and families $\mathcal{A}=\{U_{n}:n\in\omega\}$, $\mathcal{C}=\{C_{n}:n\in\omega\}$ where, for every $n\in\omega$, 
\begin{itemize}
\item $\emptyset\neq U_{n}\subseteq[X]^{<\omega}\setminus\{\emptyset\}$;
\item $\emptyset\neq C_{n}=\{x\in [X]^{\leq\omega}:(\forall z\in U_n)(z\cap x\neq\emptyset)\}$.
\end{itemize}

Without loss of generality, we assume that, for every $n\in\omega$, $U_{n}$ is infinite and, under this assumption, that every member of $C_n$ is a denumerable set. (Note that if, for some $n\in\omega$, $U_n$ is finite, then we may choose $\bigcup U_{n}$ from $C_{n}$. Also, if $U_n$ is infinite and some member of $C_n$ is finite, then applying Lemma \ref{lem:PH} to $U_{n}$ we can, without using any choice form, define a finite subset of $\bigcup U_{n}$ which is an element of $C_n$.)

Since $X\in\mathcal{N}$ and $\mathcal{A}$, $\mathcal{C}$ are denumerable sets in $\mathcal{N}$, there exists $B\in\mathscr{B}$ such that $B$ fixes $X$ and, for every $n\in\omega$, $B$ fixes  $U_{n}$ and $C_{n}$. As in the first part of the proof, we have that in $\mathcal{N}$, $X$ can be written as a well-orderable disjoint union of countable sets, namely
$$X=\bigcup\big\{\Orb_{B}(x):x\in X\big\}.$$
where, for $x\in X$, $\Orb_{B}(x)=\{\phi(x):\phi\in B\}$.  
Since $X$ is uncountable in $\mathcal{N}$, the partition
$$\mathcal{U}:=\{\Orb_{B}(x):x\in X\}$$
of $X$ (into countable sets) is infinite.

For every $n\in\omega$, put $$V_{n}=\bigcup U_{n}.$$ Since, for every $n\in\omega$, $B$ fixes $U_{n}$, $B$ also fixes $V_{n}$ for all $n\in\omega$. Indeed, if $n\in\omega$ and $\phi\in B$, then we have
$$\phi(V_{n})=\phi\big(\bigcup U_{n}\big)=\bigcup\phi(U_{n})=\bigcup U_{n}=V_{n}.$$
 
The above observation yields, for all $n\in\omega$ and $x\in V_{n}$, 
\begin{equation}
\label{eq:Orb_Vn}
\Orb_{B}(x)\cap V_{n}\neq\emptyset\rightarrow\Orb_{B}(x)\subseteq V_{n}.
\end{equation} 
To see that (\ref{eq:Orb_Vn}) holds, fix $n\in\omega$ and $x\in V_{n}$, and suppose that there exists $\phi\in B$ such that $\phi(x)\in\Orb_{B}(x)\cap V_{n}$. Let $z\in\Orb_{B}(x)$. Since $z\in\Orb_{B}(x)=\Orb_{B}(\phi(x))$, $z=\psi(\phi(x))$ for some $\psi\in B$. We have: $\phi(x)\in V_{n}\rightarrow\psi(\phi(x))\in\psi(V_{n})=V_{n}$ (the latter equality holds since $\psi\in B$ and $B$ fixes $V_{n}$), so $z=\psi(\phi(x))\in V_{n}$. Hence, $\Orb_{B}(x)\subseteq V_{n}$, and consequently (\ref{eq:Orb_Vn}) holds, as required.

By (\ref{eq:Orb_Vn}), we obtain that, for all $n\in\omega$, $$V_{n}=\bigcup\mathcal{V}_{n}\ \text{ for some infinite subfamily $\mathcal{V}_{n}$ of $\mathcal{U}$.}$$
 
In the model $M$ (which satisfies $\mathbf{AC}$), choose, for every $n\in\omega$, an element $z_{n}\in C_{n}$ (recall that, for every $i\in\omega$, $C_{i}\neq\emptyset$). Note that, for every $n\in\omega$, $z_{n}\in \mathcal{N}$, but the sequence $\langle z_{n}\rangle_{n\in\omega}$ may not be in $\mathcal{N}$. 

Since, for every $n\in\omega$, $z_n$ is denumerable and meets non-trivially every member of $U_{n}$, we can define, on the basis of $z_{n}$ and still in $M$, an infinite subfamily $\mathcal{K}_{n}$ of $\mathcal{V}_{n}$ with the following property:
\begin{equation}
\label{eq:the_K_n}
\text{if $u\in U_{n}$ and $V\in\mathcal{V}_{n}$ are such that $z_{n}\cap u\cap V\neq\emptyset$, then $V\in\mathcal{K}_{n}$}.
\end{equation}

For every $n\in\omega$, $\mathcal{K}_{n}$ is in $\mathcal{N}$, since $\mathcal{K}_{n}\subseteq\mathcal{V}_{n}\subseteq\mathcal{U}$, and thus $B$ fixes every element of $\mathcal{K}_{n}$. Furthermore, on the basis of the denumerable set $z_{n}$, we get that $\mathcal{K}_{n}$ has a denumerable choice set in $\mathcal{N}$. As $\mathcal{K}_{n}$ is disjoint (because $\mathcal{K}_{n}\subseteq\mathcal{U}$ and $\mathcal{U}$ is disjoint) and has a denumerable choice set in $\mathcal{N}$, it follows that $\mathcal{K}_{n}$ is denumerable in $\mathcal{N}$, for all $n\in\omega$. Working now exactly as in the proof of Claim \ref{claim2:s5_CAC_X}, we may conclude that 
\begin{equation}
\label{eq:K_n_denum}
(\forall n\in\omega)\big(\bigcup\mathcal{K}_{n}\text{ is denumerable in $\mathcal{N}$}\big).
\end{equation}

Note that, for every $n\in\omega$, $\bigcup\mathcal{K}_{n}\in\mathcal{N}$, because, as $B$ fixes $\mathcal{K}_{n}$ for all $n\in\omega$, $B$ fixes $\bigcup\mathcal{K}_{n}$ for all $n\in\omega$. 

Now, we define a function $f$ by
$$f=\big\{\big\langle n,\bigcup\mathcal{K}_{n}\big\rangle:n\in\omega\big\}.$$
Then $f\in\mathcal{N}$, since $B$ fixes every element of $f$. By (\ref{eq:the_K_n}), we obtain
\[
(\forall n\in\omega)\big(z_{n}\subseteq\bigcup\mathcal{K}_{n}\big),
\]
so, since $z_{n}\in C_{n}$, we infer that
\begin{equation}
\label{eq:K_n_in_C_n}
(\forall n\in\omega)(\forall z\in U_{n})\big(z\cap \big(\bigcup\mathcal{K}_{n}\big)\neq\emptyset).
\end{equation}

By (\ref{eq:K_n_denum}) and (\ref{eq:K_n_in_C_n}), we conclude that
$$(\forall n\in\omega)\big(\bigcup\mathcal{K}_{n}\in C_{n}\big),$$
and thus $f$ is a choice function for $\mathcal{C}$ in $\mathcal{N}$. Hence, $\mathbf{PS_{0}(II)}$ is true in $\mathcal{N}$, as required.
\end{proof}

On the basis of Theorem \ref{thm:general_criterion}, we are now in a position to provide \textbf{negative answers to Questions \ref{s8:q1}(2) and \ref{s8:q16}(2)} in the setting of $\mathbf{ZFA}$.

\begin{theorem}
\label{s8:t20}
For any $i\in\{0, 1, 2, 3, 4, 5\}$, the statement ``$(\forall X)(\mathbf{CAC}(X)\rightarrow \mathbf{CAC}(X^{\omega}))$'' does not imply $\mathbf{PS}_i$  in $\mathbf{ZFA}$. In particular, ``$(\forall X)(\mathbf{CAC}(X)\rightarrow \mathbf{CAC}(X^{\omega}))$'' implies neither $\mathbf{{CAC}_{fin}}$ nor $\mathbf{IDI}$ in $\mathbf{ZFA}$.
\end{theorem}

\begin{proof}
We consider the Second Fraenkel Model $\mathcal{N}2$ in \cite{hr}. The description of $\mathcal{N}2$ is as follows: We start with a model $M$ of $\mathbf{ZFA}+\mathbf{AC}$ with a set $A$ of atoms which is a denumerable disjoint union $A=\bigcup_{n\in\omega}A_{n}$, where, for each $n\in\omega$, $|A_{n}|=2$. The group $G$ of permutations of $A$ is given by
$$G=\{\phi\in\sym(A):(\forall n\in\omega)(\phi(A_{n})=A_{n})\}.$$
The normal filter $\mathcal{F}$ on $G$ is the filter of subgroups of $G$ generated by the filter base $$\mathcal{B}=\{\fix_{G}(E):E\in [A]^{<\omega}\}.$$
$\mathcal{N}2$ is the permutation model determined by $M$, $G$ and $\mathcal{F}$. By the definition of $\mathcal{F}$ and the fact that, if $\phi\in G$ is such that $\phi$ fixes an element of $A_{n}$ for some $n\in\omega$, then $\phi$ fixes $A_{n}$ pointwise (because $|A_{n}|=2$ and $\phi(A_{n})=A_{n}$), it follows that 
$$\mathcal{B}=\Big\{\fix_{G}\big(\bigcup_{s\in S}A_{s}\big):S\in [\omega]^{<\omega}\Big\}.$$ 
Therefore, if $x\in\mathcal{N}2$, then, since $\sym_{G}(x)\in\mathcal{F}$ and $\mathcal{B}$ is a base for $\mathcal{F}$, there exists $S\in[\omega]^{<\omega}$ such that $\fix_{G}(\bigcup_{s\in S}A_{s})\subseteq \sym_{G}(x)$. Under these circumstances, we call every such finite union $\bigcup_{s\in S}A_{s}$ a \emph{support} of $x$.

Let $x$ be any element of $\mathcal{N}2$. 
Then the following hold:
\smallskip

\begin{enumerate}
\item[(a)] If $E$ is a support of $x$, then, for every $\phi\in G$, $E$ is a support of $\phi(x)$. 
\smallskip

\item[(b)] $\Orb(x)=\{\phi(x):\phi\in G\}$ is finite.
\end{enumerate}
\smallskip

To see that (a) holds, let $E=\bigcup_{s\in S}A_{s}$, for some $S\in [\omega]^{<\omega}$, be a support of $x$. Let $\phi\in G$. By the definitions of $G$ and the supports, we have $\phi(E)=E$. Now, let $\pi\in\fix_{G}(E)$. Since $\phi(E)=E$, we obtain that, for every $e\in E$, $\pi(\phi(e))=\phi(e)$, or equivalently $\phi^{-1}\pi\phi(e)=e$. Therefore, $\phi^{-1}\pi\phi\in\fix_{G}(E)$ and, since $E$ is a support of $x$, we conclude that $\phi^{-1}\pi\phi(x)=x$, or equivalently $\pi\phi(x)=\phi(x)$. Hence, $E$ is a support of $\phi(x)$, and thus (a) holds.
\smallskip

That (b) holds is proved in \cite[Proof of Theorem 9.2(i)]{j}.

By the above observations, we deduce that conditions (a) and (b) of Theorem \ref{thm:general_criterion} are satisfied for $\mathcal{N}2$. Therefore, by Theorem \ref{thm:general_criterion}, $(\forall X)(\mathbf{CAC}(X)\rightarrow \mathbf{CAC}(X^{\omega}))$ is true in $\mathcal{N}2$.

By Remark \ref{s8:r5}(b), for every $i\in\{0, 1, 2, 3, 4, 5\}$, $\mathbf{PS}_i$ is false in $\mathcal{N}2$. This completes the proof of the theorem.    
\end{proof}

\begin{theorem}
\label{thm:CMC_PS}
The following hold:
\begin{enumerate}
\item $\mathbf{CMC}$ implies $\mathbf{PS_{0}(II)}$.

\item $\mathbf{MC}$ implies $\mathbf{PS_{0}(C)}$.
\end{enumerate}
\end{theorem}

\begin{proof}
(1) This is fairly easy, so we take the liberty to omit the proof.
\medskip

(2) Let $X$ be an uncountable set, let $\mathcal{A}=\{A_{n}:n\in\omega\}$ be a family of uncountable subsets of $[X]^{<\omega}\setminus\{\emptyset\}$, and let $\mathcal{B}=\{B_{n}:n\in\omega\}$ be the family which satisfies condition (d) of Definition \ref{s7:d7}(2) for $\mathcal{A}$. As in the proof of Lemma \ref{s7:l8}(b), we may assume, without loss of generality, the following:
\begin{equation}
\label{eq:infinite}
(\forall n\in \omega)(\forall x\in\mathcal{P}\big(\bigcup A_{n}\big))[((\forall z\in A_{n})\ x\cap z\neq\emptyset)\rightarrow(\text{$x$ is infinite})].
\end{equation}

By $\mathbf{MC}$, $X$ has a well-orderable partition into finite sets (see \cite[Form 67 B]{hr}), say
$$\mathcal{U}=\{F_{\alpha}:\alpha\in\kappa\}$$
for some infinite well-ordered cardinal $\kappa$ (the mapping $\alpha\mapsto F_{\alpha}$ is a bijection). 

By $\mathbf{MC}$, let $f$ be a multiple choice function for $\mathcal{B}$ (recall that, by condition (d) of Definition \ref{s7:d7}(2), $B_{n}\neq\emptyset$ for all $n\in\omega$). For every $n\in\omega$, put
$$\mathcal{Z}_{n}=\bigcup f(B_{n}).$$
Then, for every $n\in\omega$, $\mathcal{Z}_{n}$, being a finite union of countable sets, is countable. 

For every $n\in\omega$, we shall define by recursion on $\omega_{1}$ an element of $B_{n}$. Let $n\in\omega$. For every $Z\in\mathcal{Z}_{n}$, there exists a finite subset of $\bigcup Z$ which meets non-trivially every member of $Z$. Therefore, since $\mathcal{U}$ is a partition of $X$, for every $Z\in\mathcal{Z}_{n}$, there exists $I\in [\kappa]^{<\omega}$ such that:
\begin{enumerate}
\item[(a)] $I$ is of minimum cardinality;
\item[(b)] for every $z\in Z$, $z\cap\bigcup_{\alpha\in I}F_{\alpha}\neq\emptyset$; 
\item[(c)] for every $\alpha\in I$, there exists $z\in Z$ such that $z\cap F_{\alpha}\neq\emptyset$.
\end{enumerate}
For every $Z\in\mathcal{Z}_{n}$, there are only finitely many sets $I\in [\kappa]^{<\omega}$ which satisfy conditions (a), (b), and (c) for $Z$ (recall also here Lemma \ref{lem:PH}). 

\begin{definition}
For every $Z\in\mathcal{Z}_{n}$, the (finite) set
$$tr(Z):=\bigcup\big\{I\in [\kappa]^{<\omega}:I\text{ satisfies conditions (a), (b), and (c) for $Z$}\big\}$$
is called the \emph{trace} of $Z$. 
\end{definition}

Note that, since $\mathcal{Z}_{n}$ is countable, the set $\{tr(Z):Z\in\mathcal{Z}_{n}\}$ is countable.

Put $$\mathcal{M}_{n,0}=\mathcal{Z}_{n}$$ and let 
\[
I_{n,0}=\min\{tr(Z):Z\in\mathcal{M}_{n,0}\}.
\]
Since $\{tr(Z):Z\in\mathcal{M}_{n,0}\}$ is a non-empty subset of $[\kappa]^{<\omega}$ and  $[\kappa]^{<\omega}$ is well ordered (because $\kappa$ is well ordered and $|[\kappa]^{<\omega}|=\kappa$), we have $I_{n,0}$ is well defined.

Let
\[
\mathcal{Y}_{n,0}=\bigcup\big\{Z\in\mathcal{M}_{n,0}:tr(Z)=I_{n,0}\big\}.
\]
Clearly, $\mathcal{Y}_{n,0}\subseteq A_{n}$ and $(\bigcup \mathcal{Y}_{n,0})\cap (\bigcup\{F_{\alpha}:\alpha\in I_{n,0}\})$ is a finite subset of $\bigcup \mathcal{Y}_{n,0}$ (since $I_{n,0}$ is finite and every member of $\mathcal{U}$ is finite) which meets non-trivially every element of $\mathcal{Y}_{n,0}$. Furthermore, by (\ref{eq:infinite}), $\mathcal{Y}_{n,0}\neq A_{n}$.

Let
\[
\mathcal{M}_{n,1}=\{Z\in\mathcal{M}_{n,0}:Z\nsubseteq\mathcal{Y}_{n,0}\}.
\]
Since every member of $f(B_{n})$ is a partition of $A_{n}$ and $\mathcal{Y}_{n,0}\neq A_{n}$, we have $\mathcal{M}_{n,1}\neq\emptyset$. Furthermore, note that
$$\mathcal{M}_{n,1}\neq A_{n},\quad \mathcal{M}_{n,0}\supsetneq\mathcal{M}_{n,1},\quad \text{ and }\quad \mathcal{M}_{n,1}=\{Z\in\mathcal{M}_{n,0}:Z\nsubseteq\mathcal{Y}_{n,0}\wedge tr(Z)\neq I_{n,0}\}.$$ 

Let
\[
I_{n,1}=\min\{tr(Z):Z\in\mathcal{M}_{n,1}\}.
\]
Since $\mathcal{M}_{n,1}\neq\emptyset$, $\{tr(Z):Z\in\mathcal{M}_{n,1}\}$ is a non-empty subset of $[\kappa]^{<\omega}$, so $I_{n,1}$ is well defined.

Let
\[
\mathcal{Y}_{n,1}=\big(\bigcup\big\{Z\in\mathcal{M}_{n,1}:tr(Z)=I_{n,1}\big\}\big)\setminus\mathcal{Y}_{n,0}.
\] 
Then $\mathcal{Y}_{n,1}\neq\emptyset$ and $\mathcal{Y}_{n,1}\cap \mathcal{Y}_{n,0}=\emptyset$. Furthermore, $\mathcal{Y}_{n,1}\subsetneq A_{n}$ and $(\bigcup \mathcal{Y}_{n,1})\cap (\bigcup\{F_{\alpha}:\alpha\in I_{n,1}\})$ is a finite subset of $\bigcup \mathcal{Y}_{n,1}$ which meets non-trivially every element of $\mathcal{Y}_{n,1}$. 

Assume that for some $\alpha\in\omega_{1}\setminus\{0\}$, we have defined families $\{\mathcal{M}_{n,\beta}:\beta\in\alpha\}$, $\{I_{n,\beta}:\beta\in\alpha\}$, and $\{\mathcal{Y}_{n,\beta}:\beta\in\alpha\}$ such that:
\begin{itemize}
\item $\mathcal{M}_{n,0}=\mathcal{Z}_{n}$ and, for all $\beta\in\alpha\setminus\{0\}$, $\mathcal{M}_{n,\beta}=\{Z:(\forall\lambda\in\beta)(Z\in\mathcal{M}_{n,\lambda})\wedge Z\nsubseteq\bigcup_{\lambda\in\beta}\mathcal{Y}_{n,\lambda}\}$;

\item for every $\beta\in\alpha$, $I_{n,\beta}=\min\{tr(Z):Z\in\mathcal{M}_{n,\beta}\}$;

\item $\mathcal{Y}_{n,0}=\bigcup\{Z\in\mathcal{M}_{n,0}:tr(Z)=I_{n,0}\}$ and, for every $\beta\in\alpha\setminus\{0\}$, $\mathcal{Y}_{n,\beta}=\big(\bigcup\{Z\in\mathcal{M}_{n,\beta}:tr(Z)=I_{n,\beta}\}\big)\setminus\big(\bigcup_{\lambda\in\beta}\mathcal{Y}_{n,\lambda}\big)$. (Note that, for every $\beta\in\alpha$, $(\bigcup \mathcal{Y}_{n,\beta})\cap (\bigcup\{F_{\lambda}:\lambda\in I_{n,\beta}\})$ is a finite subset of $\bigcup \mathcal{Y}_{n,\beta}$ which meets non-trivially every element of $\mathcal{Y}_{n,\beta}$.)
\end{itemize}

Let
$$\mathcal{K}_{n,\alpha}=\bigcap_{\beta\in\alpha}\mathcal{M}_{n,\beta}.$$

If $\mathcal{K}_{n,\alpha}=\emptyset$, then, by the above construction, we have $\{\mathcal{Y}_{n,\beta}:\beta\in\alpha\}\in B_{n}$. Indeed, first note that since $\alpha\in\omega_{1}$, $\{\mathcal{Y}_{n,\beta}:\beta\in\alpha\}$ is countable. Second, by the above definition of $\mathcal{Y}_{n,\beta}$, $\beta\in\alpha$, it readily follows that $\{\mathcal{Y}_{n,\beta}:\beta\in\alpha\}$ is a family of pairwise disjoint sets. Third, let $u$ be any element of $A_{n}$. Since every member of $f(B_{n})$ is a partition of $A_{n}$, there exists $Z\in\mathcal{Z}_{n}$ ($=\bigcup f(B_{n})$) such that $u\in Z$. Since $\mathcal{K}_{n,\alpha}=\emptyset$, $\mathcal{M}_{n,0}=\mathcal{Z}_{n}$ and, for every $\beta\in\alpha$, $\mathcal{M}_{n,\beta}\subseteq \mathcal{M}_{n,0}$, we may let $\beta_{0}=\min\{\beta\in\alpha\setminus\{0\}:Z\not\in\mathcal{M}_{n,\beta}\}$. Since $Z\in\big(\bigcap_{\lambda\in\beta_{0}}\mathcal{M}_{n,\lambda}\big)\setminus\mathcal{M}_{n,\beta_{0}}$, it follows from the definition of $\mathcal{M}_{n,\beta_{0}}$ that $Z\subseteq \bigcup_{\lambda\in\beta_{0}}\mathcal{Y}_{n,\lambda}$. Since $u\in Z$, we obtain $u\in\bigcup_{\lambda\in\beta_{0}}\mathcal{Y}_{n,\lambda}\subseteq\bigcup_{\lambda\in\alpha}\mathcal{Y}_{n,\lambda}$. Therefore, $\{\mathcal{Y}_{n,\beta}:\beta\in\alpha\}$ is a cover of $A_{n}$. Lastly, since, for every $\beta\in\alpha$, there exists a finite subset of $\bigcup \mathcal{Y}_{n,\beta}$ which meets non-trivially every element of $\mathcal{Y}_{n,\beta}$, we conclude that $\{\mathcal{Y}_{n,\beta}:\beta\in\alpha\}\in B_{n}$, as required. 

If $\mathcal{K}_{n,\alpha}\neq\emptyset$, then we let
$$\mathcal{M}_{n,\alpha}=\{Z\in\mathcal{K}_{n,\alpha}:Z\nsubseteq\bigcup_{\beta\in\alpha}\mathcal{Y}_{n,\beta}\}.$$ 

If $\mathcal{M}_{n,\alpha}=\emptyset$, then similarly to the above arguments for the case `$\mathcal{K}_{n,\alpha}=\emptyset$' and noting that, for every $Z\in\mathcal{Z}_{n}$, either $Z\subseteq \bigcup_{\beta\in\alpha}\mathcal{Y}_{n,\beta}$ or $Z\not\in\mathcal{K}_{n,\alpha}$, we infer that $\{\mathcal{Y}_{n,\beta}:\beta\in\alpha\}\in B_{n}$.

If $\mathcal{M}_{n,\alpha}\neq\emptyset$, then we let
\begin{align*}
I_{n,\alpha}&=\min\{tr(Z):Z\in\mathcal{M}_{n,\alpha}\},\\
\mathcal{Y}_{n,\alpha}&=\big(\bigcup\big\{Z\in\mathcal{M}_{n,\alpha}:tr(Z)=I_{n,\alpha}\big\}\big)\setminus\big(\bigcup_{\lambda\in\alpha}\mathcal{Y}_{n,\lambda}\big).
\end{align*}
Clearly, $\mathcal{M}_{n,\alpha}\subsetneq\mathcal{M}_{n,\beta}$ for all $\beta\in\alpha$, $\mathcal{Y}_{n,\alpha}$ is disjoint from $\mathcal{Y}_{n,\beta}$ for all $\beta\in\alpha$, and there exists a finite subset of $\bigcup\mathcal{Y}_{n,\alpha}$ (namely $(\bigcup \mathcal{Y}_{n,\alpha})\cap (\bigcup\{F_{\lambda}:\lambda\in I_{n,\alpha}\})$) which meets non-trivially every member of $\mathcal{Y}_{n,\alpha}$.

Since $\mathcal{Z}_{n}$ is countable and at each ordinal stage $\mu$, $\mathcal{M}_{n,\mu}$ is a proper subset of $\mathcal{M}_{n,\nu}$ for all $\nu\in\mu$, it follows that the recursion must terminate at some countable ordinal stage. Otherwise, if a strictly $\subseteq$-decreasing family $\{\mathcal{M}_{n,\gamma}:\gamma\in\omega_{1}\}$ of non-empty subsets of $\mathcal{Z}_{n}$ has been constructed, then, by $\mathbf{MC}$, let $h$ be a multiple choice function for the disjoint family $\mathscr{V}=\{\mathcal{M}_{n,\gamma}\setminus\mathcal{M}_{n,\gamma+1}:\gamma\in\omega_{1}\}$. But then, $h[\mathscr{V}]$ is an $\aleph_{1}$-sized family of finite subsets of the countable set $\mathcal{Z}_{n}$, which is absurd (since $|[\omega]^{<\omega}|=\aleph_{0}$ and $\aleph_{1}$ is uncountable). Therefore, the recursion terminates at some countable ordinal stage, $\alpha_{n}$ say. This means (by the above construction) that either $\mathcal{K}_{n,\alpha_{n}}=\emptyset$ or $\mathcal{M}_{n,\alpha_{n}}=\emptyset$. In either of these two cases, we have (similarly to the above arguments) that the family
$$\Lambda_{n}:=\{\mathcal{Y}_{n,\beta}:\beta\in\alpha_{n}\}$$
is a denumerable partition of $A_{n}$ such that, for every $\beta\in\alpha_{n}$, there exists a finite subset of $\bigcup\mathcal{Y}_{n,\beta}$ (namely $(\bigcup \mathcal{Y}_{n,\beta})\cap (\bigcup\{F_{\lambda}:\lambda\in I_{n,\beta}\})$) which meets non-trivially every member of $\mathcal{Y}_{n,\beta}$. Thus, $\Lambda_{n}\in B_{n}$. 

Now, we let
$$g=\{\langle B_{n},\Lambda_{n}\rangle:n\in\omega\}.$$
By the above arguments, we conclude that $g$ is a choice function for $\mathcal{B}=\{B_{n}:n\in\omega\}$. Therefore, $\mathbf{PS_{0}(C)}$ holds, finishing the proof of (2) and of the theorem.
\end{proof}

\begin{remark}
In the forthcoming Theorem \ref{thm:CMC_notCUCDLO}(3), we will establish that the implications of Theorem \ref{thm:CMC_PS} are not reversible in $\mathbf{ZFA}$.
\end{remark}

\begin{theorem}
\label{thm:PS0(II)_not_PS0}
$\mathbf{PS_{0}(II)}\wedge\mathbf{PS_{0}(C)}$ does not imply $\mathbf{CAC_{fin}}$ in $\mathbf{ZFA}$. In particular, by Theorem \ref{s7:t5}, $\mathbf{PS_{0}(II)}\wedge\mathbf{PS_{0}(C)}$ does not imply ``for every infinite Dedekind-finite set $X$, $\mathbf{S}(X,[X]^{\leq\omega})$ is a $P$-space'' in $\mathbf{ZFA}$, and thus neither does it imply $\mathbf{PS_{0}}$ in $\mathbf{ZFA}$.
\end{theorem}

\begin{proof}
We consider the Second Fraenkel Model $\mathcal{N}2$ in \cite{hr}. It is known that 
$$\mathcal{N}2\models\mathbf{MC}\wedge\neg\mathbf{CAC_{fin}},$$
see \cite[Theorem 9.2(i)]{j}. Hence, by Theorem \ref{thm:CMC_PS}, we conclude that
$$\mathcal{N}2\models\mathbf{PS_{0}(II)}\wedge\mathbf{PS_{0}(C)}.$$
This completes the proof of the theorem.
\end{proof}

\begin{remark}
We point out that ``$\mathcal{N}2\models\mathbf{PS_{0}(II)}$'' can also be deduced from Theorem \ref{thm:general_criterion} without invoking Theorem \ref{thm:CMC_PS}. Indeed, in the proof of Theorem \ref{s8:t20}, we showed that conditions (a) and (b) of Theorem \ref{thm:general_criterion} are satisfied for $\mathcal{N}2$. Therefore, by Theorem \ref{thm:general_criterion}, $\mathbf{PS_{0}(II)}$ is true in $\mathcal{N}2$.
\end{remark}
 
\section{Other independence results and conditions under which $\mathbf{S}(X,[X]^{\leq\omega})$ and $2^{X}[[X]^{\leq\omega}]$ are $P$-spaces. A solution to Problem \ref{s4:q12}(2)}
\label{s9}

In this section, we give new non-trivial sufficient conditions for $\mathbf{S}(X, \mathcal{Z})$, $\mathbf{S}_{\mathcal{P}}(X, \mathcal{Z})$ and $2^X[\mathcal{Z}]$ to be $P$-spaces. We also establish various independence results. We give a satisfactory solution to Problem \ref{s4:q12}(2) and show that the implication in Corollary \ref{s8:c4} is not reversible in $\mathbf{ZF}$. 

\begin{proposition}
\label{s9:p1}
The following forms are equivalent:
\begin{enumerate}
\item[(i)] $\mathbf{CAC}_{\mathbf{DLO}}$;
\item[(ii)] for every linearly orderable set $X$, $\mathbf{CAC}(X^{\omega})$ holds;
\item[(iii)] for every linearly orderable set $X$, $\mathbf{CAC}(X)$ holds.
\end{enumerate}
\end{proposition}
\begin{proof}
The equivalences follow from Proposition \ref{s8:p6} and the subsequent facts:
\begin{enumerate}
\item[(a)] If $X$ is a linearly orderable set, then so is $X^{\omega}$ (notice that, given a linear order $\preceq$ on $X$, the lexicographic order $\preceq_{lex}$ on $X^{\omega}$ determined by $\preceq$ is a linear order);
\item[(b)] if $\{\langle X_n, \preceq_n\rangle: n\in\omega\}$ is a family of linearly ordered sets such that, for every pair of distinct elements $m,n$ of $\omega$, $X_m\cap X_n=\emptyset$, then the set $X=\bigcup_{n\in \omega}X_n$ is linearly orderable.
\end{enumerate}

We take the liberty to leave the details to the readers. 
\end{proof}

\begin{theorem}
\label{s9:t2}
The following hold:
\begin{enumerate}
\item Each of the following statements implies the one beneath it:
\begin{enumerate}
\item[(a)] $\mathbf{CAC_{LO}}$;

\item[(b)] $\mathbf{CAC_{DLO}}$;

\item[(c)] for every infinite linearly orderable set $X$, the spaces $\mathbf{S}(X, [X]^{\leq\omega})$, $\mathbf{S}_{\mathcal{P}}(X, [X]^{\leq\omega})$ and $2^X[[X]^{\leq\omega}]$ are all $P$-spaces;

\item[(d)] for every infinite linearly orderable set $X$, $\mathbf{S}(X, [X]^{\leq\omega})$ is a $P$-space;

\item[(e)] $\mathbf{CUC}_{\mathbf{DLO}}$;

\item[(f)] $\mathbf{vDCP}(\omega)$ $\wedge \cf(\aleph_1)=\aleph_1$.
\end{enumerate}

\item {\rm(b)} does not imply {\rm(a)} in $\mathbf{ZFA}$.
\end{enumerate}
\end{theorem}

\begin{proof}
(1) It is obvious that the implications (a) $\rightarrow$ (b), (c) $\rightarrow$ (d) and (e) $\rightarrow$ (f) are true.  
It follows from Theorem \ref{s8:t2}(i), Proposition \ref{s4:p16}(c) and Proposition \ref{s9:p1} that (b) implies (c). 

To show that (d) implies (e), we fix an arbitrary family $\{\langle X_n, \preceq_n\rangle: n\in\omega\}$ of linearly ordered sets such that, for every pair $m,n$ of distinct elements of $\omega$, $X_m\cap X_n=\emptyset$. We put $X=\bigcup_{n\in\omega}X_n$. Then $X$ is linearly orderable. Suppose $X$ is uncountable. Assuming (d), we obtain that $\mathbf{S}(X, [X]^{\leq\omega})$ is a $P$-space. It follows from Theorem \ref{s3:t6}(i) that $[X]^{\leq\omega}$ is a $\sigma$-ideal. This implies that $X$ is countable. Therefore, (d) implies (e). 
\smallskip

(2) We consider a permutation model that was introduced by Howard and Rubin in \cite{phjr}. We start with a model $M$ of $\mathbf{ZFA}+\mathbf{AC}$ with a set $A$ of atoms which is a denumerable disjoint union $A=\bigcup_{n\in\omega}A_{n}$ where,  for every $n\in\omega$, $|A_{n}|=\aleph_{0}$. Let $G$ be the group of all permutations $\pi$ of $A$ such that, for each $n\in\omega$, $\pi(A_n)=A_n$ and the set $\{a\in A: \pi(a)\neq a\}$ is finite. Let $\mathcal{F}$ be the normal filter of subgroups of $G$ generated by the pointwise stabilizers $\fix_{G}(E)$, where $E$ is a finite union of members of the family $\{A_n: n\in\omega\}$. Let $\mathcal{N}$ be the permutation model determined by $M$, $G$, and $\mathcal{F}$.

In \cite{phjr}, it was shown that $\mathbf{AC_{DLO}}$ holds in $\mathcal{N}$, but the family $\{A_n: n\in\omega\}$, which is denumerable in $\mathcal{N}$ and consists of denumerable (and thus linearly orderable) sets in $\mathcal{N}$, does not have a choice function in $\mathcal{N}$. Thus $\mathbf{CAC_{LO}}$ is false in $\mathcal{N}$ (see \cite[Lemma 2.1, Theorem 3.6]{phjr}).
\end{proof}

\begin{remark}
\label{s9:r3} 
Arguing in much the same way as in the well-known proof that $\mathbf{CAC}$ implies $\mathbf{IDI}$, and using the fact that if a set $X$ is linearly orderable, then so is $X^{\omega}$, one can prove that $\mathbf{CAC_{DLO}}$ implies Form 185 of \cite{hr} stating that every infinite linearly orderable set is Dedekind-infinite.
\end{remark}

In view of Theorems \ref{s5:t6} and \ref{s9:t2}(1), as well as of Remark \ref{s5:r7}, the question whether or not $\mathbf{CMC}(\aleph_{0},\infty)$ implies $\mathbf{CAC_{DLO}}$ is unavoidable. To the best of our knowledge, this question was, until now, open for $\mathbf{ZFA}$, and we note that it remains open for $\mathbf{ZF}$. We \textbf{settle the former question} by proving (in Theorem \ref{thm:CMC_notCUCDLO}(1) below) that $\mathbf{Z}(\omega)$ (and thus $\mathbf{CMC}(\aleph_{0},\infty)$) does not imply $\mathbf{CUC_{DLO}}$ in $\mathbf{ZFA}$. (The status of each of ``$\mathbf{Z}(\omega)\rightarrow\mathbf{vDCP}(\omega)$'' and ``$\mathbf{CMC}(\aleph_{0},\infty)\rightarrow\mathbf{vDCP}(\omega)$'' is unknown in either $\mathbf{ZF}$ or $\mathbf{ZFA}$.) We recall that $\mathbf{CMC}(\aleph_{0},\infty)$ is strictly weaker than $\mathbf{Z}(\omega)$ in $\mathbf{ZF}$; for example, in Feferman's Model $\mathcal{M}2$ of \cite{hr}, $\mathbf{CMC}(\aleph_{0},\infty)$ is true, but $\mathbf{Z}(\omega)$ is false (see \cite{hr}).

Also, in Theorem \ref{thm:CMC_notCUCDLO}(2), we show that $\mathbf{PS_{0}(II)}\wedge\mathbf{CAC_{fin}}$ (equivalently, by Lemma \ref{s7:l8}(a), $\mathbf{PS_{0}(C)}\wedge\mathbf{CAC_{fin}}$) does not imply $\mathbf{PS_{0}}$ in $\mathbf{ZFA}$ (compare with Theorem \ref{s7:t9}). In Theorem \ref{thm:CMC_notCUCDLO}(3), we establish that the implications of Theorem \ref{thm:CMC_PS} are not reversible in $\mathbf{ZFA}$.  

For the proof of Theorem \ref{thm:CMC_notCUCDLO}(1), we will need the following result of Tachtsis from \cite{TachNDS}, which is a \textbf{general criterion} for the validity of $\mathbf{Z}(\omega)$ in permutation models.

\begin{theorem}
\label{thm:3D}
(\cite[Theorem 5.4]{TachNDS}) Assume $\mathcal{N}$ is a permutation model determined by a model $M$ of $\mathbf{ZFA+AC}$ with a set $A$ of atoms, a group $G$ of permutations of $A$ and a normal filter $\mathcal{F}$ on $G$. Assume that:
\begin{enumerate}
\item For every $x\in\mathcal{N}$, $\Orb(x)=\{\phi(x):\phi\in G\}$ is countable in $\mathcal{N}$;
\item $\mathbf{WOAC_{fin}}$ is true in $\mathcal{N}$.
\end{enumerate}
Then $\mathbf{Z}(\omega)$ is true in $\mathcal{N}$.
\end{theorem}

\begin{theorem}
\label{thm:CMC_notCUCDLO}
The following hold:
\begin{enumerate}
\item $\mathbf{Z}(\omega)$ does not imply $\mathbf{CUC_{DLO}}$ in $\mathbf{ZFA}$. In particular, $\mathbf{CMC}(\aleph_{0},\infty)$ does not imply $\mathbf{CUC_{DLO}}$ in $\mathbf{ZFA}$.

\item $\mathbf{PS_{0}(II)}\wedge\mathbf{CAC_{fin}}$ does not imply $\mathbf{CUC_{DLO}}$ in $\mathbf{ZFA}$, and thus neither does it imply $\mathbf{PS_{0}}$ 
in $\mathbf{ZFA}$. Hence, by Lemma \ref{s7:l8}(a), $\mathbf{PS_{0}(C)}\wedge\mathbf{CAC_{fin}}$ does not imply $\mathbf{CUC_{DLO}}\vee\mathbf{PS_{0}}$ 
in $\mathbf{ZFA}$ either.

\item $\mathbf{PS_{0}(II)}\wedge \mathbf{PS_{0}(C)}$ does not imply $\mathbf{CMC}$ in $\mathbf{ZFA}$.
\end{enumerate}
\end{theorem}

\begin{proof}
We will use the Good--Tree--Watson Model $\mathcal{N}53$ in \cite{hr}. We start with a model $M$ of $\mathbf{ZFA+AC}$ with a set $A$ of atoms which is a denumerable disjoint union $A=\bigcup_{n\in\omega}Q_{n}$, where $Q_{n}=\{a_{n,q}:q\in\mathbb{Q}\}$ and $\mathbb{Q}$ is the set of rational numbers. Suppose $\prec$ is the lexicographic ordering on $A$. That is,
$$a_{n,q}\prec a_{m,r}\text{ if }n< m,\text{ or }n=m\text{ and }q<r.$$
Let $G$ be the group of all permutations of $A$ which are rational translations on $Q_{n}$ with $n\in\omega$. That is, if $\phi\in G$, then, for every $n\in\omega$, there exists $r_{n}\in\mathbb{Q}$ such that $\phi\upharpoonright Q_{n}(a_{n,q})=a_{n,q+r_{n}}$ for all $q\in\mathbb{Q}$. Let $\mathcal{F}$ be the filter of subgroups of $G$ generated by the pointwise stabilizers $\fix_{G}(E)$ with $E\in [A]^{<\omega}$. $\mathcal{N}53$ is the permutation model determined by $M$, $G$ and $\mathcal{F}$. By the definition of $\mathcal{F}$ and the fact that if $\phi\in G$ fixes an element of $Q_{n}$ for some $n\in\omega$, then $\phi$ fixes the whole $Q_{n}$ pointwise, it easily follows that if $x\in\mathcal{N}53$, then for some $n\in\omega$, $\fix_{G}(\bigcup_{i\leq n}Q_{i})\subseteq\sym_{G}(x)$. Under these circumstances, we call $\bigcup_{i\leq n}Q_{i}$ a \emph{support} of $x$.
\smallskip

(1) We note that ${\prec}\in\mathcal{N}53$, since $\sym_{G}(\prec)=G\in\mathcal{F}$. So $\langle A,\prec\rangle$ is a linearly ordered set in $\mathcal{N}53$. Furthermore, using standard Fraenkel--Mostowski techniques, it can be shown that the family $\mathcal{Q}=\{Q_{n}:n\in\omega\}$, which is denumerable in $\mathcal{N}53$ and consists of denumerable sets in $\mathcal{N}53$, has no multiple choice function in $\mathcal{N}53$; in particular, $\mathbf{CMC}$ is false in $\mathcal{N}53$. Therefore, $$\mathcal{N}53\models\neg\mathbf{CUC_{DLO}}.$$

We will now prove that $\mathbf{Z}(\omega)$ is true in $\mathcal{N}53$ by verifying that conditions (1) and (2) of Theorem \ref{thm:3D} are satisfied in $\mathcal{N}53$.

\begin{claim}
\label{cl1:N53}
For every $x\in\mathcal{N}53$, $\Orb(x)$ is countable in $\mathcal{N}53$.
\end{claim}

\begin{proof}[Proof of claim]
We first observe that, for every $x\in\mathcal{N}53$, $\Orb(x)$ is well orderable in $\mathcal{N}53$. Indeed, fix $x\in\mathcal{N}53$. There exists $n\in\omega$ such that the set  $E=\bigcup_{i\leq n}Q_{i}$ is a support of $x$. In much the same way as with the establishment of (a) in the proof of Theorem \ref{s8:t20}, it can be shown that, for any $\phi\in G$, $E$ is a support of $\phi(x)$. Consequently, $E$ is a support of every element of $\Orb(x)$, so $\Orb(x)$ is well orderable in $\mathcal{N}53$.  

Now, we show that $\Orb(x)$ is countable in $\mathcal{N}53$. We let $$V=A\setminus E=\bigcup_{i>n}Q_{i}.$$ 
We assert that 
\begin{equation}
\label{eq:Orb_N53}
\Orb(x)=\Orb_{V}(x),
\end{equation}
where $\Orb_{V}(x)$ is the $\fix_{G}(V)$-orbit of $x$, i.e. $\Orb_{V}(x)=\{\phi(x):\phi\in\fix_{G}(V)\}$. It is clear that $\Orb_{V}(x)\subseteq\Orb(x)$. Conversely, let $\phi\in G$. Let $\eta$ be the permutation of $A$ such that $\eta\upharpoonright E=\phi\upharpoonright E$ and $\eta$ is the identity on $V$ (so $\eta\in\fix_{G}(V)$). Since $\phi,\eta$ agree on $E$, it follows that $\eta^{-1}\phi\in\fix_{G}(E)$ and, since $E$ is a support of $x$, $\eta^{-1}\phi(x)=x$, or equivalently $\phi(x)=\eta(x)$. As $\eta(x)\in\Orb_{V}(x)$, it follows that $\phi(x)\in\Orb_{V}(x)$. Hence, $\Orb(x)\subseteq \Orb_{V}(x)$, and therefore (\ref{eq:Orb_N53}) is true, as asserted.

The group $\fix_{G}(V)$ is isomorphic to the group $K=\prod_{i\leq n}\mathsf{Tr}(Q_{i})$, where, for any $i\in\omega$, $\mathsf{Tr}(Q_{i})$ denotes the group of all rational translations on $Q_{i}$. Since, for every $i\in\omega$, $\mathsf{Tr}(Q_{i})$ is isomorphic to $\mathsf{Tr}(\mathbb{Q})$ (the group of all rational translations on $\mathbb{Q}$) and $|\mathsf{Tr}(\mathbb{Q})|=\aleph_{0}$, we deduce that, for every $i\in\omega$, $|\mathsf{Tr}(Q_{i})|=\aleph_{0}$. Since $K$ is a finite product of $\mathsf{Tr}(Q_{i})$s, we get $|K|=\aleph_{0}$, and so $|\fix_{G}(V)|=\aleph_{0}$. Hence, in $M$, $|\Orb_{V}(x)|\leq \aleph_{0}$, and thus, by (\ref{eq:Orb_N53}), $|\Orb(x)|\leq \aleph_{0}$ in $M$. Since, by the first part of the proof, $\Orb(x)$ is well orderable in $\mathcal{N}53$, it follows that $\Orb(x)$ is countable in $\mathcal{N}53$. This completes the proof of the claim.
\end{proof}

\begin{claim}
\label{cl2:N53}
$\mathbf{WOAC_{fin}}$ is true in $\mathcal{N}53$.
\end{claim}

\begin{proof}[Proof of claim]
Let $\mathcal{U}=\{U_{\alpha}:\alpha\in\kappa\}$, where $\kappa$ is an aleph, be a well-ordered family in $\mathcal{N}53$ consisting of non-empty finite sets; the map $\kappa\ni\alpha\mapsto U_{\alpha}$ is a bijection in $\mathcal{N}53$. There exists $n\in\omega$ such that the set $E=\bigcup_{i\leq n}Q_{i}$ is a support of $\{\langle\alpha,U_{\alpha}\rangle:\alpha\in\kappa\}$. It follows that, for every $\alpha\in\kappa$, $E$ is a support of $U_{\alpha}$. 

We assert that $E$ is a support of every element of $\bigcup\mathcal{U}$. This will give us that $\bigcup\mathcal{U}$ is well orderable in $\mathcal{N}53$, and hence that $\mathcal{U}$ has a choice function in $\mathcal{N}53$. By way of contradiction, we assume that there exist $\alpha\in\kappa$ and $x\in U_{\alpha}$ such that $E$ is not a support of $x$. Thus, there exists $\phi\in\fix_{G}(E)$ such that $\phi(x)\neq x$.

Since $x\in\mathcal{N}53$, there exists $m\in\omega$ such that the set $E'=\bigcup_{i\leq m}Q_{i}$ is a support of $x$.  
As $E$ is not a support of $x$, we have $m>n$.
Let $\eta$ be the permutation of $A$ which agrees with $\phi$ on $E'$ (and thus $\eta\in\fix_{G}(E)$) and is the identity map on $A\setminus E'$. Since $\eta$ agrees with $\phi$ on $E'$ and $E'$ is a support of $x$, we have $\eta(x)=\phi(x)$, and thus $\eta(x)\neq x$, since $\phi(x)\neq x$. Furthermore, as $\eta=\prod_{n<i\leq m}\eta_{i}$, where for $n<i\leq m$, $\eta_{i}\upharpoonright Q_{i}=\phi\upharpoonright Q_{i}$ and $\eta_{i}\upharpoonright (A\setminus Q_{i})$ is the identity map, it follows that for some $n<i_{0}\leq m$, $\eta_{i_{0}}(x)\neq x$. Let $$\mathcal{G}=\fix_{G}(A\setminus Q_{i_{0}}).$$ We have $\eta_{i_{0}}\in\mathcal{G}$ and $\mathcal{G}$ is isomorphic to $\mathsf{Tr}(Q_{i_{0}})$. Let
$$Y=\{\psi(x):\psi\in\mathcal{G}\}.$$
Since $x\in U_{\alpha}$, $\mathcal{G}\subseteq\fix_{G}(E)$ and $E$ is a support of $U_{\alpha}$, it follows that $Y\subseteq U_{\alpha}$. As $U_{\alpha}$ is finite, so is $Y$. Let
$$H=\{\psi\in\mathcal{G}:\psi(x)=x\}.$$
Since $|Y|>1$ (because $x,\eta_{i_{0}}(x)\in Y$ and $\eta_{i_{0}}(x)\neq x$), $H$ is a proper subgroup of $\mathcal{G}$. Since $|Y|=(\mathcal{G}:H)$ and $Y$ is finite, it follows that the quotient group $\mathcal{G}/H$ has finite order, say $k>1$. 

There exists $q_0\in\mathbb{Q}$ such that $\eta_{i_{0}}\upharpoonright Q_{i_{0}}$ is a translation on $Q_{i_{0}}$ by $q_{0}$. Let $\sigma$ be the permutation of $A$ which is such that $\sigma\upharpoonright Q_{i_{0}}$ is a translation on $Q_{i_{0}}$ by $q_{0}/k$ and $\sigma\upharpoonright (A\setminus Q_{i_{0}})$ is the identity map. It is clear that $\sigma\in\mathcal{G}$ and that
\begin{equation}
\label{e:1}
\eta_{i_{0}}=\underbrace{\sigma\circ\sigma\circ\cdots\circ\sigma}_{k\text{-times}}
\end{equation}
Since $\mathcal{G}/H$ has order $k$, we get
\begin{equation}
\label{e:2}
H=\underbrace{(\sigma H)\oplus(\sigma H)\oplus\cdots\oplus(\sigma H)}_{k\text{-times}},
\end{equation}
where $\oplus$ denotes the group operation on $\mathcal{G}/H$. By (\ref{e:1}) and (\ref{e:2}), we deduce that
$$H=\eta_{i_{0}}H,$$
so $\eta_{i_{0}}\in H$. But this contradicts the fact that $\eta_{i_{0}}(x)\neq x$.
Therefore, $E$ is a support of every element of $\bigcup\mathcal{U}$, so $\bigcup\mathcal{U}$ is well orderable in $\mathcal{N}53$. It follows that $\mathbf{WOAC_{fin}}$ is true in $\mathcal{N}53$, as required.
\end{proof}

By Claims \ref{cl1:N53}, \ref{cl2:N53} and Theorem \ref{thm:3D}, we conclude that
$$\mathcal{N}53\models\mathbf{Z}(\omega),$$
and thus $\mathcal{N}53\models\mathbf{CMC}(\aleph_{0},\infty)$. 
\smallskip

(2) By Claim \ref{cl2:N53}, $\mathbf{CAC_{fin}}$ is true in $\mathcal{N}53$. By the opening paragraph of the proof of the theorem, we see that 
$$\mathcal{B}=\Big\{\fix_{G}\big(\bigcup_{i\leq n}Q_{i}\big):n\in\omega\Big\}$$ 
is a filter base for the normal filter $\mathcal{F}$ on $G$ which was used for the construction of $\mathcal{N}53$. This, together with Claim \ref{cl1:N53} and the argument in the first paragraph of the proof of Claim \ref{cl1:N53}, yields that conditions (a) and (b) of Theorem \ref{thm:general_criterion} are satisfied for $\mathcal{N}53$. Therefore, by Theorem \ref{thm:general_criterion}, $\mathbf{PS_{0}(II)}$ is true in $\mathcal{N}53$. Since $\mathbf{PS_{0}(II)}\wedge\mathbf{CAC_{fin}}$ is true in $\mathcal{N}53$, we obtain, by Lemma \ref{s7:l8}(a), that $\mathbf{PS_{0}(C)}\wedge\mathbf{CAC_{fin}}$ is also true in $\mathcal{N}53$.

By (1), we know that $\mathbf{CUC_{DLO}}$ is false in $\mathcal{N}53$ and, by Theorem \ref{s9:t2}(1), we obtain that the statement ``For every infinite linearly orderable set $X$, $\mathbf{S}(X,[X]^{\leq\omega})$ is a $P$-space'' is false in $\mathcal{N}53$. Consequently, $\mathbf{PS}_{0}$ is false in $\mathcal{N}53$.
\smallskip

(3) By the proofs of (1) and (2), we know that $\mathbf{PS_{0}(II)}\wedge\mathbf{PS_{0}(C)}$ is true in $\mathcal{N}53$, whereas $\mathbf{CMC}$ is false in $\mathcal{N}53$. Hence, (3) holds as required.
\end{proof}

\begin{remark}
\label{s9:r8}
We note that the status of each of $\mathbf{Z}(\omega)$, $\mathbf{CMC}(\aleph_{0},\infty)$ and $\mathbf{WOAC_{fin}}$ in Model $\mathcal{N}53$ (of the proof of Theorem \ref{thm:CMC_notCUCDLO}) is not specified in \cite{hr}. Thus, the proof of Theorem \ref{thm:CMC_notCUCDLO} \textbf{fills the gap in information} in \cite{hr}.

We also observe that, by Claim \ref{cl2:N53} (of the proof of Theorem \ref{thm:CMC_notCUCDLO}) and Theorem \ref{s7:t5}, 
$$\mathcal{N}53\models\mathbf{PS}_{2}.$$ 
\end{remark}
   
On the basis of Theorem \ref{s6:t4}, we will prove, in Theorem \ref{thm:Models_M1_N3} below, that the conjunction $\mathbf{BPI}\wedge \mathbf{PS}_0$ does not imply $\mathbf{CAC}(\mathbb{R})\vee\mathbf{CMC}(\aleph_0, \infty)$ in $\mathbf{ZF}$. This will provide \textbf{strong negative answers to Problem \ref{s4:q12}}(2) and will also show that \textbf{``(b) $\rightarrow$ (d)'' of Theorem \ref{s9:t2}(1) is not reversible in $\mathbf{ZF}$}. Moreover, in Theorem \ref{thm:Models_M1_N3}(2), we will establish that \textbf{the implication in Corollary \ref{s8:c4} is not reversible in} $\mathbf{ZF}$.
 
\begin{theorem}
\label{thm:Models_M1_N3}
The following hold:
\begin{enumerate}
\item $\mathbf{PS}_0$ is true in the Basic Cohen Model $\mathcal{M}1$ of \cite{hr}.

\item ``$\mathbb{R}$ admits a topology $\tau$ such that $\langle\mathbb{R}, \tau\rangle$ is a Hausdorff, zero-dimensional, crowded $P$-space'' is strictly weaker than $\mathbf{CAC}(\mathbb{R})$ in $\mathbf{ZF}$. 

\item The conjunction $\mathbf{BPI}\wedge \mathbf{PS}_0$  does not imply $\mathbf{CAC}(\mathbb{R})\vee\mathbf{CMC}(\aleph_0, \infty)$ in $\mathbf{ZF}$. Thus, the above conjunction does not imply $\mathbf{CAC_{DLO}}\vee\mathbf{CMC}(\aleph_0, \infty)$ in $\mathbf{ZF}$ either.

\item $\mathbf{PS}_0$ is true in the Mostowski Linearly Ordered Model $\mathcal{N}3$ of \cite{hr}.

\item $\mathbf{PS}_{0}$ is strictly weaker than $\mathbf{CAC}$, and thus than ``For every infinite set $X$, $\mathbf{S}_{\mathcal{P}}(X, [X]^{\leq\omega})$  and $2^{X}[[X]^{\leq\omega}]$ are $P$-spaces'', in $\mathbf{ZF}$. 
\end{enumerate}
\end{theorem}

\begin{proof}
(1) First, let us recall that $\mathcal{M}1$ is a symmetric submodel of a generic extension $M[G]$ constructed by forcing over a countable transitive model $M$ of $\mathbf{ZFC}$ with the set $\mathbb{P}:=\Fn(\omega\times\omega,2)$ partially ordered by reverse inclusion, that is, for all $p,q\in\mathbb{P}$, $p\leq q$ if and only if $p \supseteq q$. More specifically, note first that every $\pi\in\sym(\omega)$ induces an automorphism of $\langle \mathbb{P},\leq\rangle$ by
\begin{align*}
\dom(\pi p)&=\{\langle \pi(n),i\rangle:\langle n,i\rangle\in\dom(p)\},\\
(\pi p)(\pi(n),i)&=p(n,i).
\end{align*}
Let $\GG$ be the group of all automorphisms of $\langle \mathbb{P},\leq\rangle$ induced by permutations $\pi\in\sym(\omega)$ as above. Let $\F$ be the filter of subgroups of $\GG$ generated by the filter base $\{\fix_{\GG}(E):E\in [\omega]^{<\omega}\}$; $\F$ is a normal filter on $\GG$ (see \cite[Chapter 5]{j}). $\M1$ is the symmetric submodel of $M[G]$, where $G$ is a $\mathbb{P}$-generic filter over $M$, which is determined by $\GG$ and $\F$. That is, $$\M1=\{\tau_{G}:\tau\in\HS\},$$
where $\HS$ denotes the class of all hereditarily symmetric $\mathbb{P}$-names in $M$, and, for $\tau\in \HS$, $\tau_{G}$ denotes the value of the name $\tau$ by $G$  (see \cite[Definition 2.7, p. 189]{kun}).

Let $A=\{x_{n}:n\in\omega\}$ be the set of the denumerably many added generic reals (see \cite[(5.11), p. 66]{j}). It is well known that $A$ is Dedekind-finite in $\mathcal{M}1$ (see \cite[Proof of Lemma 5.15]{j}). Furthermore, by \cite[Lemma 5.25]{j}, for every $X\in \mathcal{M}1$, there are, in $\mathcal{M}1$, an ordinal $\gamma$ and an injection $f:X\rightarrow [A]^{<\omega}\times\gamma$.  

Let $X$ be an uncountable set in $\mathcal{M}1$. Fix an ordinal $\gamma$ and an injection $f:X\rightarrow [A]^{<\omega}\times\gamma$ which is in $\mathcal{M}1$. There are two cases for $X$.
\smallskip

\textit{Case 1.} $X$ is well orderable in $\mathcal{M}1$. Since $|X|=|f[X]|$ in $\mathcal{M}1$, we get $f[X]$ is well orderable in $\mathcal{M}1$. Put
$$\mathcal{S}=\{b\in [A]^{<\omega}:(\exists \alpha\in\gamma)(\langle b,\alpha\rangle\in f[X])\}.$$
Since $f[X]$ is well orderable in $\mathcal{M}1$ and $A$ is quasi Dedekind-finite in $\mathcal{M}1$, it follows that $\mathcal{S}$ is finite. Therefore, the set
$$s=\bigcup\mathcal{S}$$
is a finite subset of $A$. It is fairly easy to see that $f[X]$ has a hereditarily symmetric $\mathbb{P}$-name, using the canonical $\mathbb{P}$-names of the reals $x_{n}$ (see \cite[(5.12), p. 66]{j}) which belong to $s$, and the $\mathbb{P}$-names $\check{\alpha}$, $\alpha\in\gamma$. Let
$$m_{s}=\max\{m\in\omega:x_{m}\in s\}+1.$$
We also let $$P_{m_{s}}=\{p\in\mathbb{P}:\dom(p)\subseteq m_{s}\times\omega\}\quad \text{ and }\quad G_{m_{s}}=G\cap \mathbb{P}_{m_{s}}.$$
Then $G_{m_{s}}$ is a $\mathbb{P}_{m_{s}}$-generic filter over $M$ and, since every $\mathbb{P}_{m_{s}}$-name is a hereditarily symmetric name (if $\dot{x}$ is a $\mathbb{P}_{m_{s}}$-name, then $\fix_{\GG}(m_{s})\subseteq\sym_{\GG}(\dot{x})$), we have
$$M[G_{m_{s}}]\subseteq\mathcal{M}1.$$
Furthermore, in view of the above observations, $f[X]\in M[G_{m_{s}}]$ and
\begin{equation}
\label{eq:L_s}
(\mathbf{S}(f[X],[f[X]]^{\leq\omega}))^{\mathcal{M}1}=(\mathbf{S}(f[X],[f[X]]^{\leq\omega}))^{M[G_{m_{s}}]}.
\end{equation}

We now show that $(\mathbf{S}(f[X],[f[X]]^{\leq\omega}))^{\mathcal{M}1}$ is a $P$-space in $\mathcal{M}1$. This, together with the fact that $|X|=|f[X]|$ in $\mathcal{M}1$, will give us that $(\mathbf{S}(X,[X]^{\leq\omega}))^{\mathcal{M}1}$ is a $P$-space in $\mathcal{M}1$. Let $\langle O_{n}\rangle_{n\in\omega}$ be a sequence of open sets in $(\mathbf{S}(f[X],[f[X]]^{\leq\omega}))^{\mathcal{M}1}$ such that $\emptyset\in\bigcap_{n\in\omega}O_{n}$. By (\ref{eq:L_s}), we have, for every $n\in\omega$, $O_{n}\in M[G_{m_{s}}]$ and $O_{n}$ is open in $(\mathbf{S}(f[X],[f[X]]^{\leq\omega}))^{M[G_{m_{s}}]}$. Since $M[G_{m_{s}}]$ satisfies $\mathbf{AC}$ (see \cite[Theorem 4.2, p. 201]{kun}), we choose, for every $n\in\omega$, a $z_{n}\in ([f[X]]^{\leq\omega})^{M[G_{m_{s}}]}$ such that $(B_{\emptyset,z_{n}}\subseteq O_{n})^{M[G_{m_{s}}]}$. We let 
$$z=\bigcup_{n\in\omega}z_{n}.$$
Then $z \in ([f[X]]^{\leq\omega})^{M[G_{m_{s}}]}$ (since $M[G_{m_{s}}]$ satisfies $\mathbf{CUC}$) and, clearly, $(B_{\emptyset,z}\subseteq \bigcap_{n\in\omega}O_{n})^{M[G_{m_{s}}]}$. Since $M[G_{m_{s}}]\subseteq\mathcal{M}1$, it follows that $z \in ([f[X]]^{\leq\omega})^{\mathcal{M}1}$. Furthermore, by (\ref{eq:L_s}) and the fact that $(B_{\emptyset,z}\subseteq \bigcap_{n\in\omega}O_{n})^{M[G_{m_{s}}]}$, we deduce that $(B_{\emptyset,z}\subseteq \bigcap_{n\in\omega}O_{n})^{\mathcal{M}1}$. Therefore, $(\mathbf{S}(f[X],[f[X]]^{\leq\omega}))^{\mathcal{M}1}$ is a $P$-space in $\mathcal{M}1$, and hence so is $(\mathbf{S}(X,[X]^{\leq\omega}))^{\mathcal{M}1}$. 
\smallskip

\textit{Case 2.} $X$ is not well orderable in $\mathcal{M}1$. On the basis of the injection $f:X\rightarrow [A]^{<\omega}\times\gamma$, there is, in $\mathcal{M}1$, a well-ordered partition $\{X_\alpha:\alpha\in I\}$, where $I\subseteq\gamma$, of $X$ into Dedekind-finite sets. Without loss of generality, for each $\alpha\in I$, we may view $X_{\alpha}$ as a subset of $[A]^{<\omega}$. Since $X$ is an infinite subset of $[A]^{<\omega}$ and $A$ is a Dedekind-finite set of reals in $\mathcal{M}1$, Theorem \ref{s6:t4}(c) yields $\mathbf{S}(X,[X]^{\leq\omega})$ is a $P$-space in $\mathcal{M}1$.

The above arguments complete the proof of (1).
\smallskip

(2) This follows from the subsequent facts: part (1); $|\mathbb{R}|=|[\mathbb{R}]^{<\omega}|$; Corollary \ref{s8:c4}; and $\mathbf{CAC}(\mathbb{R})$ is false in the Basic Cohen Model (see \cite{hr}).  
\smallskip

(3) By part (1), we know that $\mathbf{PS}_{0}$ is true in $\mathcal{M}1$. Thus, by Theorem \ref{s7:t4}(b), $\mathbf{CUC}$ holds in $\mathcal{M}1$. On the other hand, it is known that $\mathbf{BPI}$ is true in $\mathcal{M}1$ (see \cite{hr}, \cite{j}). Furthermore, in $\mathcal{M}1$, the set $A$ of the added generic reals is dense in $\mathbb{R}$ with its usual linear ordering. Since $A$ is Dedekind-finite in $\mathcal{M}1$, we conclude that $\mathbf{CMC}(\aleph_0, \infty)$ is false in $\mathcal{M}1$ (and $\mathbf{CAC}(\mathbb{R})$ is false in $\mathcal{M}1$) for the denumerable family $\mathcal{A}=\{A\cap (n,n+1):n\in\omega\}$, where $(n,n+1)$ is the open interval in the usual ordering of $\mathbb{R}$. Moreover, since $\mathbf{CAC}(\mathbb{R})$ is false in $\mathcal{M}1$, so is $\mathbf{CAC_{DLO}}$.
\smallskip

(4) Let $A$ be the set of atoms of $\mathcal{N}3$. (For a detailed description of $\mathcal{N}3$, see the Appendix.) It is well known that $A$ is weakly Dedekind-finite (and thus quasi Dedekind-finite) in $\mathcal{N}3$ (see \cite[Problem 5, p. 52]{j}). Furthermore, by \cite[Lemma 4.6]{j}, for every $X\in \mathcal{N}3$, there are, in $\mathcal{N}3$, an ordinal $\gamma$ and an injection $f:X\rightarrow [A]^{<\omega}\times\gamma$. 

Let $X$ be an uncountable set in $\mathcal{N}3$. If $X$ is well orderable in $\mathcal{N}3$, then, by Corollary \ref{s8:c8}, $\mathbf{S}(X,[X]^{\leq\omega})$ is a $P$-space in $\mathcal{N}3$. If $X$ is not well orderable in $\mathcal{N}3$, then working in much the same way as in the proof of Case 2 of part (1), we may conclude, by virtue of Theorem \ref{s6:t4}(c), that $\mathbf{S}(X,[X]^{\leq\omega})$ is a $P$-space in $\mathcal{N}3$. (For the readers' further insight and deeper understanding, a crude proof of $\mathbf{PS}_0$ in $\mathcal{N}3$, without using Theorem \ref{s6:t4}, will be given in the Appendix.)
\smallskip

(5) The ``weaker'' part follows from Theorems \ref{s4:t8}(iii) and \ref{s6:t1}. The ``strictly weaker'' part follows from (1), Theorem \ref{s6:t1}, and from the fact that $\mathbf{CAC}$ is false in the Basic Cohen Model $\mathcal{M}1$ of \cite{hr}.
\end{proof}

Let us have a look at the formula ``$(\forall X)(\mathbf{CAC}(X)\rightarrow\mathbf{CAC}(X^{\omega}))$'' again to show its independence from $\mathbf{PS}_2$ in $\mathbf{ZFA}$.

\begin{theorem}
\label{s9:t10}
The following hold:
\begin{enumerate}
\item The implication $\mathbf{PS}_1\rightarrow \mathbf{PS}_2$ is true in $\mathbf{ZF}$, but not reversible in $\mathbf{ZF}$.

\item ``For every infinite well-orderable set $X$, there exists an infinite set $Y\subseteq X$ such that $\mathbf{S}(Y, [Y]^{\leq\omega})$ is a $P$-space'' is provable in $\mathbf{ZF}$.

\item ``$(\forall X)(\mathbf{CAC}(X)\rightarrow\mathbf{CAC}(X^{\omega}))$'' does not imply $\mathbf{PS}_2$ in $\mathbf{ZFA}$.

\item $\mathbf{PS}_2$ does not imply ``$(\forall X)(\mathbf{CAC}(X)\rightarrow\mathbf{CAC}(X^{\omega}))$'' in $\mathbf{ZF}$.
\end{enumerate}
\end{theorem}

\begin{proof}
(1) We have noticed in Proposition \ref{s7:p3}(ii) that $\mathbf{PS}_1$ implies $\mathbf{PS}_2$ in $\mathbf{ZF}$. 

To show that $\mathbf{PS}_2$ does not imply  $\mathbf{PS}_1$ in $\mathbf{ZF}$, we consider any model $\mathcal{M}$ of $\mathbf{ZF}$ in which $\mathbf{CAC_{fin}}$ is true and $\mathbf{CUC}$ is false. For example, in Sageev's Model I, denoted by $\mathcal{M}6$ in \cite{hr}, $\mathbf{CAC_{fin}}\wedge\neg\mathbf{CUC}$ is true (see \cite[p. 152]{hr}). By Theorem \ref{s7:t5}, $\mathbf{PS}_2$ is true in $\mathcal{M}$. By Theorem \ref{s7:t4}(b), $\mathbf{PS}_1$ is false in $\mathcal{M}$. 
\smallskip

(2) This follows from the fact that every infinite well-orderable set is Dedekind-infinite and from Proposition \ref{s4:p16}(c).
\smallskip

(3) This follows from Theorems \ref{s8:t20} and \ref{s7:t5}.
\smallskip

(4) By Theorem \ref{s8:t11}(ii), there is a model $\mathcal{U}$ of $\mathbf{ZF}$ such that
$$\mathcal{U}\models [\mathbf{IDI}\wedge (\exists X)(\mathbf{CAC}(X)\wedge\neg \mathbf{CAC}(X^{\omega}))].$$
\noindent By Theorem \ref{s7:t4}(c), we have $\mathcal{U}\models \mathbf{PS}_2$.  This completes the proof of the theorem.    
\end{proof}

For our next result, Theorem \ref{s9:t12} below, we will need the following lemma.

\begin{lemma}
\label{lem:WOAM_CACDLO} 
$\mathbf{WOAM}$ implies $\mathbf{CAC_{DLO}}$.
\end{lemma}

\begin{proof}
The proof of the lemma follows from the subsequent facts:
\begin{enumerate}
\item No infinite linearly ordered set is amorphous. This was shown by Truss in \cite[Theorem 3]{jtr}.

\item $\mathbf{WOAM}$ implies ``Every linearly ordered set can be well ordered'' (\cite[Form 90]{hr}). This follows from (1).

\item The union of a countable family of linearly ordered sets can be linearly ordered.
\end{enumerate}

The details are left to the readers.   
\end{proof}

\begin{theorem}
\label{s9:t12}
The following hold:
\begin{enumerate}
\item $\mathbf{WOAM}$ implies that, for every infinite well-orderable set $X$, $\mathbf{S}_{\mathcal{P}}(X, [X]^{\leq\omega})$, $2^X[[X]^{\leq\omega}]$ and $\mathbf{S}(X, [X]^{\leq\omega})$ are all $P$-spaces. In consequence, the following implications are true:

 $$\mathbf{WOAM}\rightarrow \mathbf{PS}_3\rightarrow \mathbf{PS}_4\rightarrow\mathbf{PS}_5\rightarrow \mathbf{PS}_1\rightarrow (\mathbf{CUC}\wedge\mathbf{PS}_2).$$

\item $\mathbf{PS}_0$ does not imply $\mathbf{WOAM}$ in $\mathbf{ZF}$. Thus, if $i\in\{1, 2, 3, 4, 5\}$, then $\mathbf{PS}_i$ does not imply $\mathbf{WOAM}$ in $\mathbf{ZF}$.
  
\item $\mathbf{WOAM}$ does not imply $\mathbf{NAS}$ in $\mathbf{ZFA}$.  Thus, if $i\in\{1, 2,3, 4, 5\}$, then $\mathbf{PS}_i$ does not imply $\mathbf{NAS}$ in $\mathbf{ZFA}$.

\item $\mathbf{W}_{\aleph_1}\wedge \mathbf{CAC_{DLO}}$ implies $\mathbf{W}_{\aleph_1}\wedge \cf(\aleph_{1})=\aleph_{1}$ which in turn implies ``For every uncountable set $X$, there exists an uncountable set $Y\subseteq X$ such that  $\mathbf{S}_{\mathcal{P}}(Y, [Y]^{\leq\omega})$,  $2^Y[[Y]^{\leq\omega}]$ and $\mathbf{S}(Y,[Y]^{\leq\omega})$ are all $P$-spaces''.

\item In every permutation model, $\mathbf{W}_{\aleph_{1}}$ implies ``For every uncountable set $X$, there exists an uncountable set $Y\subseteq X$ such that $\mathbf{S}_{\mathcal{P}}(Y, [Y]^{\leq\omega})$,  $2^Y[[Y]^{\leq\omega}]$ and $\mathbf{S}(Y,[Y]^{\leq\omega})$ are all $P$-spaces''.

\item $\mathbf{PS}_0$ does not imply $\mathbf{W}_{\aleph_{1}}\vee\mathbf{CAC_{DLO}}$ in $\mathbf{ZF}$.

\item $\mathbf{W}_{\aleph_{1}}\wedge\mathbf{CAC_{DLO}}$ and $\mathbf{WOAM}$ are mutually independent in $\mathbf{ZFA}$.
\end{enumerate}
\end{theorem}

\begin{proof}
(1) Assume $\mathbf{WOAM}$. Let $X$ be an infinite set. By $\mathbf{WOAM}$, $X$ is either well orderable or has an amorphous subset. If $X$ is well orderable, then, by Lemma \ref{lem:WOAM_CACDLO} and Theorem \ref{s9:t2}(1), the spaces $\mathbf{S}_{\mathcal{P}}(X, [X]^{\leq\omega})$, $2^X[[X]^{\omega}]$ and $\mathbf{S}(X,[X]^{\leq\omega})$ are all $P$-spaces. Proposition \ref{s7:p3}(ii) and Theorem \ref{s7:t4}(b) complete the proof of (1).
\smallskip

(2) This follows from Theorem \ref{thm:Models_M1_N3}(1) and the fact that $\mathbf{WOAM}$ is false in the Basic Cohen Model $\mathcal{M}1$ of \cite{hr} (see \cite[p. 147]{hr}).
\smallskip

(3) This follows from (1) and the fact that the Basic Fraenkel Model $\mathcal{N}1$ of \cite{hr} satisfies  $\mathbf{WOAM}\wedge\neg\mathbf{NAS}$`` (see \cite[p. 177]{hr}).
\smallskip

(4) The first implication follows from Theorem \ref{s9:t2}(1). 

For the second implication, assume that $\mathbf{W}_{\aleph_1}\wedge \cf(\aleph_{1})=\aleph_{1}$ is true. Let $X$ be an uncountable set. Since $X$ is uncountable, $\mathbf{W}_{\aleph_1}$ yields the existence of a subset $Y$ of $X$ with $|Y|=\aleph_{1}$. By ``$\cf(\aleph_{1})=\aleph_{1}$'' and Theorem \ref{s5:t4}, we obtain that $\mathbf{S}_{\mathcal{P}}(Y, [Y]^{\leq\omega})$,  $2^Y[[Y]^{\leq\omega}]$ and $\mathbf{S}(Y,[Y]^{\leq\omega})$ are all $P$-spaces.
\smallskip

(5) This follows from second implication in part (4) and the fact that ``$\cf(\aleph_{1})=\aleph_{1}$'' is true in every permutation model (see \cite{hr}).
\smallskip

(6) By Theorem \ref{thm:Models_M1_N3}(1), we know that  $\mathbf{PS}_0$ is true in the Basic Cohen Model $\mathcal{M}1$ from \cite{hr}. We have shown in the proof of Theorem \ref{thm:Models_M1_N3}(3) that $\mathbf{CAC_{DLO}}$ is false in $\mathcal{M}1$. It is known that also $\mathbf{W}_{\aleph_1}$ is false in $\mathcal{M}1$ (see the proof of Theorem \ref{thm:Models_M1_N3}(1)). This shows that (6) holds.
\smallskip

(7) We first consider the permutation model $\mathcal{N}16$ from \cite{hr}. Tachtsis showed in \cite[Proof of Theorem 5.4]{TachAPAL} that ``Every linearly ordered set can be well ordered'' is true in $\mathcal{N}16$. Since the latter statement (clearly) implies $\mathbf{CAC_{DLO}}$, we obtain that $\mathbf{CAC_{DLO}}$ is true in $\mathcal{N}16$. On the other hand, $(\forall n\in\omega)(\mathbf{W}_{\aleph_{n}})\wedge\neg\mathbf{WOAM}$ is true in $\mathcal{N}16$, as shown by Jech in \cite[Proof of Theorem 8.6]{j}. Therefore, $\mathbf{W}_{\aleph_{1}}\wedge\mathbf{CAC_{DLO}}$ does not imply $\mathbf{WOAM}$ in $\mathbf{ZFA}$ as required.

Now, $\mathbf{WOAM}\wedge\neg\mathbf{IDI}$ is true in the Basic Fraenkel Model $\mathcal{N}1$. Since $\mathbf{W}_{\aleph_{1}}$ implies $\mathbf{IDI}$ (which is equivalent to $\mathbf{W}_{\aleph_{0}}$), we get $\mathbf{WOAM}$ does not imply $\mathbf{W}_{\aleph_{1}}\wedge\mathbf{CAC_{DLO}}$ in $\mathbf{ZFA}$ as required.
\end{proof}

\begin{remark}
\label{rem:W_aleph}
We note that $\mathbf{AC_{DLO}}$ (and thus $\mathbf{CAC_{DLO}}$) does not imply $\mathbf{PS}_1$ in $\mathbf{ZFA}$. This follows from Theorem \ref{s7:t4}(b) and from the fact that $\mathbf{AC_{DLO}}\wedge \neg\mathbf{CUC}$ is true in the permutation model $\mathcal{N}$ of the proof of Theorem \ref{s9:t2}(2).

We \emph{do not know} whether or not $\mathbf{W}_{\aleph_{1}}$ implies either of $\mathbf{PS}_1$  or ``For every uncountable set $X$, there exists an uncountable set $Y\subseteq X$ such that $2^{Y}[[Y]^{\leq\omega}]$ is a $P$-space'' in $\mathbf{ZF}$. Moreover, it is also \emph{unknown} if $\mathbf{W}_{\aleph_{1}}$ implies $\mathbf{CUC}$ in $\mathbf{ZF}$. However, we observe that $\mathbf{W}_{\aleph_{1}}\wedge\cf(\aleph_{1})=\aleph_{1}$ implies $\mathbf{CUC}$, as it can be easily verified. Since ``$\cf(\aleph_{1})=\aleph_{1}$'' is true in every permutation model (see \cite{hr}), it follows that, in every permutation model, $\mathbf{W}_{\aleph_{1}}$ implies $\mathbf{CUC}$. 

Now, since $\mathbf{AC_{DLO}}\wedge \neg\mathbf{CUC}$ is true in the permutation model $\mathcal{N}$ of the proof of Theorem \ref{s9:t2}(2) and ``$\cf(\aleph_{1})=\aleph_{1}$'' is true in every permutation model, it follows from the observation of the second paragraph that $\mathbf{W}_{\aleph_{1}}$ is false in $\mathcal{N}$. Therefore, $\mathbf{AC_{DLO}}$ (and thus $\mathbf{CAC_{DLO}}$) does not imply $\mathbf{W}_{\aleph_{1}}$ in $\mathbf{ZFA}$.

In Remark \ref{s9:r19}, we will observe, on the basis of the proof of Theorem \ref{s9:t15} below, that $\mathbf{W}_{\aleph_{1}}$ does not imply $\mathbf{CAC_{DLO}}$ in $\mathbf{ZFA}$. This, together with the above considerations, yields $\mathbf{W}_{\aleph_{1}}$ and $\mathbf{CAC_{DLO}}$ are mutually independent in $\mathbf{ZFA}$. Moreover, since ``$\cf(\aleph_{1})=\aleph_{1}$'' is true in every permutation model, we have that the first implication in part (4) of Theorem \ref{s9:t12} is not reversible in $\mathbf{ZFA}$.  
\end{remark}
\smallskip

In view of the following:
\begin{enumerate}
\item Theorems \ref{s4:t8}(ii)--(iii),  \ref{s7:t4}(b) and \ref{s9:t12}(1);

\item $\mathbf{WOAM}\rightarrow\mathbf{UT}(\mathbf{WO,WO,WO})$ (the latter implication can be proved similarly to Lemma \ref{lem:WOAM_CACDLO});

\item $\mathbf{CAC}\rightarrow \mathbf{PS}_0 \rightarrow \mathbf{CUC}$ (by Theorems \ref{s8:t2}(i) and \ref{s7:t4}(b));

\item $\mathbf{CAC}\rightarrow\mathbf{UT}(\aleph_{0},\mathbf{WO},\mathbf{WO})$ and $\mathbf{CAC}\nrightarrow\mathbf{UT}(\mathbf{WO},\mathbf{WO},\mathbf{WO})$ in $\mathbf{ZF}$ (see \cite[Model $\mathcal{M}1(\langle\omega_{1}\rangle)$]{hr}), 
\end{enumerate}
it is reasonable to ask whether or not ``For every uncountable set $X$, there exists an uncountable set $Y\subseteq X$ such that $\mathbf{S}_{\mathcal{P}}(Y, [Y]^{\leq\omega})$ is a $P$-space'' implies $\mathbf{UT}(\aleph_{0},\mathbf{WO},\mathbf{WO})$. We \textbf{answer this question in the setting of $\mathbf{ZFA}$} by showing in the following theorem that ``For every uncountable set $X$, there exists an uncountable set $Y\subseteq X$ such that $\mathbf{S}_{\mathcal{P}}(Y, [Y]^{\leq\omega})$ is a $P$-space'' does not imply $\mathbf{UT}(\aleph_{0},\mathbf{WO},\mathbf{WO})$ in $\mathbf{ZFA}$. We will need the following group-theoretic lemma.

\begin{lemma}
\label{lem:tr}
(\cite{dt}, \cite{tr}) The group $A(\mathbb{R})$ of order automorphisms of $\langle \mathbb{R},\leq\rangle$, where $\leq$ is the usual ordering of $\mathbb{R}$, has no proper subgroups with index $<2^{\aleph_{0}}$.
\end{lemma}

\begin{theorem}
\label{s9:t15}
``For every uncountable set $X$, there exists an uncountable set $Y\subseteq X$ such that $\mathbf{S}_{\mathcal{P}}(Y, [Y]^{\leq\omega})$, $2^Y[[Y]^{\leq\omega}]$ and $\mathbf{S}(Y,[Y]^{\leq\omega})$ are all $P$-spaces'' does not imply $\mathbf{UT}(\aleph_{0},\mathbf{WO},\mathbf{WO})$ in $\mathbf{ZFA}$. In particular, $\mathbf{PS}_1$ does not imply $\mathbf{UT}(\aleph_{0},\mathbf{WO},\mathbf{WO})$ in $\mathbf{ZFA}$.
\end{theorem}

\begin{proof}
We will use a permutation model which was introduced by Keremedis and Tachtsis in \cite{kt}. The description of the model is as follows.

 We start with a model $M$ of $\mathbf{ZFA}+ \mathbf{AC}$ with a set $A$ of atoms which is a denumerable disjoint union $A=\bigcup\{A_{n}:n\in\omega\}$, where each $A_{n}$, $n\in\omega$, is continuum sized and ordered like the reals with their usual ordering by $\leq_{n}$. Let $G$ be the group of all permutations of $A$ such that, for all $n\in\omega$ and $\phi\in G$, $\phi\upharpoonright A_{n}$ is an order automorphism of $\langle A_{n},\leq_{n}\rangle$. Let $\mathcal{F}$ be the (normal) filter of subgroups of $G$ generated by the pointwise stabilizers $\fix_{G}(E)$, where $E$ is a finite union of $A_{n}$s. Let $\mathcal{N}$ be the permutation model determined by $M$, $G$ and $\mathcal{F}$. If $x\in\mathcal{N}$, then $\sym_{G}(x)\in\mathcal{F}$, so there exists an $S\in [\omega]^{<\omega}$ such that $\fix_{G}(\bigcup_{s\in S}A_{s})\subseteq\sym_{G}(x)$. Under these circumstances, we call $\bigcup_{s\in S}A_{s}$ a \emph{support} of $x$.

\begin{claim}
\label{cl1:NKT}
Every uncountable set in $\mathcal{N}$ has an uncountable well-orderable subset in $\mathcal{N}$. In particular, $\mathbf{W}_{\aleph_{1}}$ is true in $\mathcal{N}$.
\end{claim}

\begin{proof}[Proof of claim]
Let $X$ be an uncountable set in $\mathcal{N}$. If $X$ is well orderable in $\mathcal{N}$, then the conclusion of the claim is straightforward. Suppose $X$ is not well orderable in $\mathcal{N}$. Since $X\in\mathcal{N}$, let $E=\bigcup_{s\in S}A_{s}$, for some $S\in [\omega]^{<\omega}$, be a support of $X$. Since $X$ is not well orderable in $\mathcal{N}$, there exists $x\in X$ such that $E$ is not a support of $x$. Let $S'\in [\omega]^{<\omega}$ be such that $S'\cap S=\emptyset$ and the set $E\cup E'$, where $E'=\bigcup_{s\in S'}A_{s}$, is a support of $x$. 

Since $E$ is not a support of $x$, there exists $\phi\in\fix_{G}(E)\setminus\sym_{G}(x)$. Let $\eta$ be the permutation of $A$ which agrees with $\phi$ on $E'$ and $\eta$ is the identity map on $A\setminus E'$. Clearly, $\eta\in G$. Moreover, $\eta\in\fix_{G}(E)$, since $E\cap E'=\emptyset$. As $\eta$ agrees with $\phi$ on $E\cup E'$ and $E\cup E'$ is a support of $x$, we have $\eta(x)=\phi(x)$, and thus $\eta(x)\neq x$, since $\phi(x)\neq x$. Furthermore, as $S'$ is finite and $\eta=\prod_{s\in S'}\eta_{s}$, where for $s\in S'$, $\eta_{s}\upharpoonright A_{s}=\phi\upharpoonright A_{s}$ and $\eta_{s}\upharpoonright (A\setminus A_{s})$ is the identity map, it follows that for some $s_{0}\in S'$, $\eta_{s_{0}}(x)\neq x$.
We let $$\mathbf{G}=\fix_{G}(A\setminus A_{s_{0}}).$$ 
Then $\mathbf{G}$ is isomorphic to the group $A(\mathbb{R})$. Furthermore,  
$$\Orb_{A\setminus A_{s_{0}}}(x)=\{\pi(x):\pi\in \mathbf{G}\}$$ is a subset of $X$ (since $x\in X$ and $\mathbf{G}\subseteq\fix_{G}(E)\subseteq\sym_{G}(X)$) which is well orderable in $\mathcal{N}$. Indeed, $\fix_{G}(E\cup E')\subseteq\fix_{G}(\Orb_{A\setminus A_{s_{0}}}(x))$. To see this, let $\rho\in\fix_{G}(E\cup E')$ and $\pi\in\mathbf{G}$. By the definition of $G$, it follows that $\pi(E\cup E')=E\cup E'$. Furthermore, since $E\cup E'$ is a support of $x$, $\pi(E\cup E')$ is a support of $\pi(x)$ and, since $\pi(E\cup E')=E\cup E'$, we conclude that $E\cup E'$ is a support of $\pi(x)$. Therefore, $\rho(\pi(x))=\pi(x)$, i.e. $\rho\in\fix_{G}(\Orb_{A\setminus A_{s_{0}}}(y))$ as required.

Now, we show that $\Orb_{A\setminus A_{s_{0}}}(x)$ is uncountable in $\mathcal{N}$. Let
$$\mathbf{H}=\{\pi\in\mathbf{G}:\pi(x)=x\}.$$
Then $\mathbf{H}$ is a proper subgroup of $\mathbf{G}$ (since $\eta_{s_{0}}\in\mathbf{G}\setminus\mathbf{H}$) and $|\Orb_{A\setminus A_{s_{0}}}(x)|=(\mathbf{G}:\mathbf{H})$. As $\mathbf{G}$ is isomorphic to $A(\mathbb{R})$ and $\mathbf{H}$ is a proper subgroup of $\mathbf{G}$, it follows from Lemma \ref{lem:tr} that $(\mathbf{G}:\mathbf{H})\not\leq\aleph_{0}$. Therefore, $|\Orb_{A\setminus A_{s_{0}}}(x)|\not\leq\aleph_{0}$. This completes the proof of the claim.
\end{proof}

\begin{claim}
\label{cl2:NKT}
The statement ``For every uncountable set $X$, there exists an uncountable set $Y\subseteq X$ such that $\mathbf{S}_{\mathcal{P}}(Y, [Y]^{\leq\omega})$, $2^Y[[Y]^{\leq\omega}]$ and $\mathbf{S}(Y,[Y]^{\leq\omega})$ are all $P$-spaces'' is true in $\mathcal{N}$. Hence, $\mathbf{PS}_1\wedge \mathbf{CUC}$ is also true in $\mathcal{N}$. (We note that ``$\mathcal{N}\models\mathbf{CUC}$'' was also shown in \cite{kt}.)
\end{claim}

\begin{proof}[Proof of claim]
The first assertion follows from Claim \ref{cl1:NKT} and Theorem \ref{s9:t12}(5). The second assertion follows from the first and Theorem \ref{s7:t4}(b).
\end{proof}

\begin{claim}
\label{cl3:NKT}
$\mathbf{UT}(\aleph_{0},\mathbf{WO},\mathbf{WO})$ is false in $\mathcal{N}$.
\end{claim}

\begin{proof}[Proof of claim]
Consider the family $\mathcal{A}=\{A_{n}:n\in\omega\}$. Then $\mathcal{A}\in\mathcal{N}$ and $\mathcal{A}$ is denumerable in $\mathcal{N}$, since $\fix_{G}(\mathcal{A})=G\in\mathcal{F}$. Furthermore, for every $n\in\omega$, $A_{n}$ is well orderable in $\mathcal{N}$, since $\fix_{G}(A_{n})\in\mathcal{F}$.

Now, using standard Fraenkel--Mostowski techniques, it can be shown that $\mathcal{A}$ has no choice function in $\mathcal{N}$, and thus $\bigcup\mathcal{A}$ is not well orderable in $\mathcal{N}$. We take the liberty to leave the details to the readers.
\end{proof}

The above arguments complete the proof of the theorem.    
\end{proof}

\begin{remark}
\label{s9:r19}
Let $\mathcal{N}$ be the permutation model of the proof of Theorem \ref{s9:t15}. By Claim \ref{cl1:NKT} of this proof, we know that $\mathbf{W}_{\aleph_{1}}$ is true in $\mathcal{N}$. On the other hand, $\mathbf{CAC_{DLO}}$ is false in $\mathcal{N}$. Indeed, first note that, for every $n\in\omega$, the linear ordering $\leq_{n}$ of $A_{n}$ is in $\mathcal{N}$, since $\sym_{G}(\leq_{n})=G\in\mathcal{F}$ for all $n\in\omega$. Second, from the proof of Claim \ref{cl3:NKT}, we know that the family $\{A_{n}:n\in\omega\}$ (which is denumerable in $\mathcal{N}$) has no choice function in $\mathcal{N}$. So $\mathbf{CAC_{DLO}}$ is false in $\mathcal{N}$. We thus conclude that $\mathbf{W}_{\aleph_{1}}$ does not imply $\mathbf{CAC_{DLO}}$ in $\mathbf{ZFA}$. 
\end{remark}

\section{On $\mathbf{P}(WO, 0-dim, T_2, crow)$, $\mathbf{P}(0-dim, T_2, crow)$, and a solution to Problem \ref{s1:q2}((2), (3))}
\label{s10}

We investigate the set-theoretic strength of $\mathbf{P}(WO, 0-dim, T_2, crow)$ and $\mathbf{P}(0-dim, T_2, crow)$ (see Definition \ref{s2forms}((22), (23))), and establish some independence results concerning these principles. We begin with the non-trivial result that $\mathbf{P}(WO, 0-dim, T_2,crow)$ (and thus $\mathbf{P}(0-dim, T_2, crow)$) is true in the Feferman-L\'{e}vy Model $\mathcal{M}9$ in \cite{hr}. We will deduce this from Theorem \ref{s5:t1} and the fact that $\mathbf{G}(\cf(\aleph)>\aleph_0)$ (see Definition \ref{s2forms}(15)) is true in $\mathcal{M}9$. This, together with the fact that $\omega_{1}$ is singular in $\mathcal{M}9$, will \textbf{completely settle Problem \ref{s1:q2}(3)}. Perhaps, it is known that $\mathbf{G}(\cf(\aleph)>\aleph_0)$ is true in $\mathcal{M}9$. However, the status of $\mathbf{G}(\cf(\aleph)>\aleph_0)$ in $\mathcal{M}9$ is not specified in \cite{hr} and neither in \cite{j} nor in \cite{Je}. Therefore, for completeness and for the readers' convenience, we include the relatively short argument that $\mathbf{G}(\cf(\aleph)>\aleph_0)$ is true in $\mathcal{M}9$ in the proof of Theorem \ref{s10:t1}(2) below. 

Finally, in Corollary \ref{s10:c2}(6), we will show that $\mathbf{P}(WO, 0-dim, T_2,crow)$ is true in the Basic Cohen Model $\mathcal{M}1$ in \cite{hr}; this provides an \textbf{affirmative answer to Problem \ref{s1:q2}(2)}.  

\begin{theorem}
\label{s10:t1}
The following hold:
\begin{enumerate}
\item[(1)] Each of the following statements implies the one beneath it:
\begin{enumerate}
\item[(i)] $\cf(\aleph_{1})=\aleph_{1}$;

\item[(ii)] $\mathbf{G}(\cf(\aleph)>\aleph_0)$;

\item[(iii)] $\mathbf{P}(WO, 0-dim, T_2, crow)$;

\item[(iv)] $\mathbf{P}(0-dim, T_2, crow)$.
\end{enumerate}

\item[(2)] $\mathbf{G}(\cf(\aleph)>\aleph_0)$ (and thus $\mathbf{P}(WO, 0-dim, T_2, crow)$ and $\mathbf{P}(0-dim, T_2, crow)$) is true in the Feferman--L\'{e}vy Model $\mathcal{M}9$ in \cite{hr}.

\item[(3)] $\mathbf{G}(\cf(\aleph)>\aleph_0)$ (and thus $\mathbf{P}(WO, 0-dim, T_2, crow)$ and $\mathbf{P}(0-dim, T_2, crow)$) is strictly weaker than $\cf(\aleph_1)=\aleph_1$ in $\mathbf{ZF}$.
\end{enumerate}
\end{theorem}

\begin{proof}
(1) The implications ``(i) $\rightarrow$ (ii)'' and ``(iii) $\rightarrow$ (iv)'' are straightforward. The implication ``(ii) $\rightarrow$ (iii)'' follows from Theorem \ref{s5:t1}.
\smallskip

(2) For the readers' convenience, we first recall the description of $\mathcal{M}9$.
We start with a countable transitive model $M$ of $\mathbf{ZFC}$. The notion of forcing is the set 
\[
\P=\{p\in\Fn(\omega\times\omega,\omega_{\omega}):(\forall \langle n,i\rangle\in\dom(p))(p(\langle n,i\rangle)\in\omega_{n})\}
\]
endowed with the partial order of reverse inclusion. Let $G$ be a $\P$-generic filter over $M$ and let $M[G]$ be the generic extension model of $M$. Then all cardinals smaller than $\omega_{\omega}^{M}$ are collapsed to $\omega$ in $M[G]$, since if $f=\bigcup G$, then $f$ is a function (in $M[G]$) on $\omega\times\omega$, and, for every $n\in\omega$, the function $f_{n}$ defined on $\omega$ by $f_{n}(i)=f(n,i)$ for all $i\in\omega$, maps $\omega$ onto $\omega_{n}$ (see Jech \cite[Example 15.57, p. 259]{Je}).

Let $\GG$ be the group of all permutations $\pi$ of $\omega\times\omega$ such that 
\[(\forall n\in\omega)(\forall i\in\omega)(\exists j\in\omega)(\pi(n,i)=\langle n,j\rangle).\]
Every $\pi\in\GG$ induces an automorphism of $\langle P,\le\rangle$ by
\begin{align*}
\dom(\pi p)&=\{\pi(n,i):\langle n,i\rangle\in\dom(p)\},\\
(\pi p)(\pi(n,i))&=p(n,i).
\end{align*}
Let $\F$ be the (normal) filter of subgroups of $\GG$ generated by $\{H_{n}:n\in\omega\}$, where
 \[H_{n}=\{\pi\in\GG:(\forall k\le n)(\forall i\in\omega)(\pi(k,i)=\langle k,i\rangle)\}.\]
  
\noindent Model $\M9$ is the symmetric submodel of $M[G]$ which is determined by the group $\GG$ and the normal filter $\F$. That is, $$\M9=\{\tau_{G}:\tau\in\HS\},$$
where $\HS$ is the class of all hereditarily symmetric names in $M$ and, for $\tau\in \HS$, $\tau_{G}$ is the interpretation of $\tau$ by $G$  (see \cite[Chapter 5]{j}).

The following are known about $\mathcal{M}9$ (see \cite[Example 15.57, p. 259]{Je}):
\begin{enumerate}

\item[(a)] For all $n\in\omega$, $f_{n}\in\M9$, so $\omega_{n}^{M}$ is countable in $\M9$ for all $n\in\omega$;

\item[(b)] $\omega_{\omega}^{M}$ is countable in $M[G]$ (recall that $M[G]$ satisfies $\mathbf{AC}$ and in $M[G]$, $\omega_{\omega}^{M}$ is a countable union of countable sets);

\item[(c)] $\omega_{\omega}^{M}$ is a cardinal in $\M9$; 

\item[(d)] $\omega_{1}^{\M9}=\omega_{\omega}^{M}$ (by (a) and (c)). So, in $\M9$, $\cf(\omega_{1}^{\M9})=\omega$.
\end{enumerate}

We now show that $\mathbf{G}(\cf(\aleph)>\aleph_0)$ is true in $\M9$. First, note that $\P\subseteq \Fn(\omega\times\omega,\omega_{\omega})$. Consider the poset $\mathbb{S}=\langle\Fn(\omega\times\omega,\omega_{\omega}),\supseteq\rangle$. By \cite[Lemma 6.10, p. 213]{kun}, $\mathbb{S}$ has the $((\omega_{\omega})^{<\omega})^{+}$-cc property in $M$. Since $((\omega_{\omega})^{<\omega})^{+}=(\omega_{\omega})^{+}=\omega_{\omega+1}$, it follows that $\mathbb{S}$ has the $\omega_{\omega+1}$-cc property in $M$. As $\P\subseteq\mathbb{S}$, $\P$ has the $\omega_{\omega+1}$-cc property in $M$ also. By \cite[Lemma 6.9, p. 213]{kun}, we conclude that $\P$ preserves cofinalities $\geq(\omega_{\omega+1})^{M}$ and, since $(\omega_{\omega+1})^{M}$ is regular, $\P$ preserves cardinals $\geq(\omega_{\omega+1})^{M}$ also. This, together with (b), gives us $$\omega_{1}^{M[G]}=(\omega_{\omega+1})^{M}=\omega_2^{\mathcal{M}9}.$$ 
Thus, in $\M9$, $\cf(\omega_{2}^{\M9})>\omega$. This shows that (2) holds, as required.
\smallskip

(3) This follows from (1), (2), and the fact that $\omega_{1}^{\mathcal{M}9}$ is singular in $\M9$ (see \cite[pp. 153--154]{hr},  \cite[Example 15.57, p. 259]{Je}.  
\end{proof}

\begin{corollary}
\label{s10:c2}
The following hold:
\begin{enumerate}
\item[(1)] $\mathbf{P}(WO, 0-dim, T_2, crow)$ is true in every permutation model of $\mathbf{ZFA}$.

\item[(2)] There is a model of $\mathbf{ZF}$ in which $\mathbf{P}(WO, 0-dim, T_2, crow)$ is true and $\mathbb{R}$ is  a countable union of countable sets.

\item[(3)] If $\mathbf{S}(\mathbb{R}, [\mathbb{R}]^{\leq\omega})$ is a $P$-space, then $\mathbb{R}$ is not a countable union of countable sets, and the implication is not reversible in $\mathbf{ZF}$.

\item[(4)] There is a model of $\mathbf{ZF}$ in which $\mathbf{P}(WO, 0-dim, T_2, crow)$ is true and $\mathbf{S}(\mathbb{R}, [\mathbb{R}]^{\leq\omega})$ is not a $P$-space.

\item[(5)] $\mathbf{PS}_1$ implies $\mathbf{P}(WO, 0-dim, T_2, crow)$, but this implication is not reversible in $\mathbf{ZF}$.

\item[(6)] $\mathbf{P}(WO, 0-dim, T_2, crow)$ is true in the Basic Cohen Model $\mathcal{M}1$ of \cite{hr}.
\end{enumerate}
\end{corollary}

\begin{proof}
(1) This follows from Proposition \ref{s4:p16} and Corollary \ref{s8:c8}.  
\smallskip

(2) In the Feferman--L\'{e}vy Model $\mathcal{M}9$ in \cite{hr}, $\mathbb{R}$ is a countable union of countable sets (see \cite[Theorem 10.6]{j}). The conclusion now follows from Theorem \ref{s10:t1}(2).
\smallskip

(3) The validity of the implication follows from Theorem \ref{s8:t2}(iii). For the second assertion of (3), we note that, in \cite[Appendix]{Miller}, Arnold Miller constructed a variant of the Feferman--L\'{e}vy Model $\mathcal{M}9$, denoted by $V$, in which $\mathbb{R}=\bigcup_{n\in\omega}B_n$ where, for every $n\in\omega$, $B_{n}$ is a countable union of countable sets, but $\mathbb{R}$ is not a countable union of countable sets in $V$. Thus, for some $n\in\omega$, $B_{n}$ is uncountable in $V$, and so, by Theorem \ref{s3:t6}(i), $\mathbf{S}(B_{n},[B_{n}]^{\leq\omega})$ is not a $P$-space in $V$. This, together with Remark \ref{s3:r10} and the fact that subspaces of $P$-spaces are $P$-spaces, yields $\mathbf{S}(\mathbb{R},[\mathbb{R}]^{\leq\omega})$ is not a $P$-space in $V$.   
\smallskip

(4) Since $\mathbb{R}$ is a countable union of countable sets in the Feferman--L\'{e}vy Model $\mathcal{M}9$, it follows from (3) that $\mathbf{S}(\mathbb{R}, [\mathbb{R}]^{\leq\omega})$ is not a $P$-space in $\mathcal{M}9$. On the other hand, by Theorem \ref{s10:t1}(2), $\mathbf{P}(WO, 0-dim, T_2, crow)$ is true in $\mathcal{M}9$.
\smallskip

(5) It follows from Theorem \ref{s5:t4} (also from Theorem \ref{s9:t2}) that $\mathbf{PS}_1$ implies the equality $\cf(\aleph_1)=\aleph_1$. Thus, by Theorem \ref{s10:t1}, (5) holds.
\smallskip

(6) This follows from Theorem \ref{s10:t1}(1) and the fact that ``$\cf(\aleph_{1})=\aleph_{1}$'' is true in $\mathcal{M}1$ (see \cite{hr}).
\end{proof}

\begin{remark}
\label{s10:r3}
\begin{enumerate}
\item We recall that Form 38 in \cite{hr} (i.e. ``$\mathbb{R}$ is not the union of a countable family of countable sets'') is strictly weaker than $\cf(\aleph_1)=\aleph_1$ in $\mathbf{ZF}$ (see \cite[Problem 2, p. 148]{j} and \cite[Model $\mathcal{M}12(\aleph)$]{hr}). In view of this observation, note in particular that Corollary \ref{s10:c2}(2) is a \textbf{proper strengthening} of Theorem \ref{s10:t1}(4).

\item By Theorem \ref{s10:t1}(1) and the fact that $\mathbf{CUC}$ (clearly) implies $\cf(\aleph_{1})=\aleph_{1}$, one readily obtains that $\mathbf{CUC}$ implies $\mathbf{P}(WO, 0-dim, T_2, crow)$. It is worth noting here that a weaker union form also implies $\mathbf{P}(WO, 0-dim, T_2, crow)$. Indeed, Tachtsis \cite{TachFM} considered the following union form, denoted by $\mathbf{UT}(\aleph_{0},\cuc,\cuc)$ in \cite{TachFM}: ``If $\mathcal{A}$ is a denumerable family of sets each of which can be expressed as a denumerable union of denumerable sets, then $\bigcup\mathcal{A}$ can be expressed as a denumerable union of denumerable sets''.

In \cite[Theorem 5.2(1)]{TachFM}, it was shown (among other results) that the principle $\mathbf{UT}(\aleph_{0},\cuc,\cuc)$ implies $\mathbf{G}(\cf(\aleph)>\aleph_0)$. Thus, by Theorem \ref{s10:t1}(1), $$\mathbf{UT}(\aleph_{0},\cuc,\cuc)\rightarrow\mathbf{P}(WO, 0-dim, T_2, crow)\rightarrow\mathbf{P}(0-dim, T_2, crow).$$ On the other hand, by \cite[Theorem 5.7(1)]{TachFM}, $\mathbf{UT}(\aleph_{0},\cuc,\cuc)$ does not imply $\mathbf{CUC}$ in $\mathbf{ZFA}$.
\end{enumerate} 
\end{remark}

\section{Questions}
\label{s11}

\begin{question}
\label{s11:q1}
Does $\mathbf{CUC}$ imply $\mathbf{PS}_0$, or $\mathbf{PS_{0}(II)}$, or $\mathbf{PS_{0}(C)}$? (Compare with Theorem \ref{s7:t9}.)
\end{question}

\begin{question}
\label{s11:q2}
\begin{enumerate}
\item Is there a model of $\mathbf{ZF}$ in which $\mathbf{PS_{0}(C)}$ is false?

\item Is there a model of $\mathbf{ZFA}$ or of $\mathbf{ZF}$ in which $\mathbf{PS_{0}}$, or $\mathbf{PS_{0}(II)}$ or $\mathbf{PS_{0}(C)}$, is true, but $\mathbf{PS_{0}(D)}$ is false?

\item What is the status of each of $\mathbf{PS}_0$, $\mathbf{PS_{0}(II)}$, and $\mathbf{PS_{0}(C)}$ in Model $\mathcal{N}17$ of the proof of Theorem \ref{s7:t12}(2)?

\item Is $\mathbf{PS_{0}(D)}$ false in Model $\mathcal{N}53$ of the proof of Theorem \ref{thm:CMC_notCUCDLO}? 
\end{enumerate}
\end{question}

\begin{question}
\label{s11:q3}
\begin{enumerate}
\item Does $\cf(\aleph_1)=\aleph_1$ imply ``For every uncountable well-orderable set $X$, $\mathbf{S}(X, [X]^{\leq\omega})$ is a $P$-space'' in $\mathbf{ZF}$? (Compare with Theorems \ref{s5:t4} and \ref{s5:t6}.)

\item Does $\mathbf{CMC}(\aleph_{0},\infty)$ imply ``For every uncountable well-orderable set $X$, $\mathbf{S}(X, [X]^{\leq\omega})$ is a $P$-space'' in $\mathbf{ZF}$? (Compare with Theorem \ref{s5:t6}.)
\end{enumerate}
\end{question}

\begin{question}
\label{s11:q4}
Does $\mathbf{CUC}_{\mathbf{DLO}}$ imply ``For every uncountable linearly orderable set $X$, $\mathbf{S}(X, [X]^{\leq\omega})$ is a $P$-space'' in $\mathbf{ZF}$?
\end{question}

\begin{question}
\label{s11:q5}
\begin{enumerate}
\item Are $\mathbf{P}(WO, 0-dim, T_2, crow)$ and $\mathbf{P}(0-dim, T_2, crow)$ provable in $\mathbf{ZF}$?

\item Is $\mathbf{P}(WO, 0-dim, T_2, crow)$, or $\mathbf{P}(0-dim,T_2, crow)$, strictly weaker than $\mathbf{G}(\cf(\aleph)>\aleph_0)$ in $\mathbf{ZF}$? (Compare with Theorem \ref{s10:t1}.)

\item Is $\mathbf{P}(WO, 0-dim, T_2, crow)$, or $\mathbf{P}(0-dim, T_2, crow)$, true in Gitik's Model $\mathcal{M}17$ from \cite{hr} in which all uncountable cardinals are singular?
\end{enumerate}
\end{question}

\begin{remark}
\label{s11:r6}
Of course, a positive answer to Question \ref{s11:q5}(1) will give a positive answer to Question \ref{s11:q5}(3) and a negative answer to Question \ref{s11:q5}(2).
\end{remark}

\begin{question}
\label{s11:q7}
Does the statement ``$\mathbb{R}$ is not a countable union of countable sets'' (Form 38 in \cite{hr}) imply $\mathbf{P}(WO, 0-dim, T_2, crow)$, or $\mathbf{P}(0-dim, T_2, crow)$, in $\mathbf{ZF}$? (Compare with Corollary \ref{s10:c2}(2).)
\end{question}

\begin{question}
\label{s11:q8}
Is there a model of $\mathbf{ZFA}$ or of $\mathbf{ZF}$ in which ``For every infinite linearly ordered set $\langle X,\leq\rangle$, $[X]^{<\omega}$ admits a 0-dimensional, crowded Hausdorff topology $\tau$ such that $\langle [X]^{<\omega},\tau\rangle$ is a $P$-space'' is true, but $\mathbf{CAC_{DLO}}$ is false? (Compare with Theorem \ref{s9:t2}. Note also that, if in the former statement ``infinite'' is replaced by ``uncountable'', then the resulting question receives an affirmative answer in $\mathbf{ZF}$ due to Theorem \ref{thm:Models_M1_N3}(1) and the proof of Theorem \ref{thm:Models_M1_N3}(3).)
\end{question}

\begin{question}
\label{s11:q9}
Does $\mathbf{BPI}$ imply $\mathbf{PS}_0$?
\end{question}

\begin{question}
\label{s11:q10}
Does $\mathbf{WOAM}$ imply $\mathbf{PS}_0$?
\end{question}

\begin{question}
\label{s11:q11}
Is $\mathbf{PS}_0$ true in the Basic Fraenkel Model $\mathcal{N}1$ of \cite{hr}?
\end{question}

\begin{question}
\label{s11:q12}
Is any of the implications in Theorem \ref{s7:t4}(b) reversible in $\mathbf{ZF}$ or $\mathbf{ZFA}$?
\end{question}

\begin{question}
\label{s11:q13}
Does $\mathbf{W}_{\aleph_{1}}$ imply either of ``For every uncountable set $X$, there exists an uncountable set $Y\subseteq X$ such that $\mathbf{S}_{\mathcal{P}}(X, [X]^{\leq\omega})$ is a $P$-space'' or $\mathbf{PS}_1$?
\end{question}

\section*{Appendix}
Below, we give a crude proof of $\mathbf{PS}_0$ in the Mostowski Linearly Ordered Model, using only the properties of the model rather than Theorem \ref{s6:t4}. The proof is interesting in its own right and we include it because we consider it important to provide the readers with further insight and information.
\medskip

\noindent\textbf{Theorem.} \emph{$\mathbf{PS}_0$ is true in the Mostowski Linearly Ordered Model $\mathcal{N}3$ of \cite{hr}}.

\begin{proof}
The description of $\mathcal{N}3$ is as follows.
We start with a model $M$ of $\mathbf{ZFA} + \mathbf{AC}$ with a denumerable set $A$ of atoms using an ordering $\le$ of $A$ chosen so that $\langle A,\le\rangle$ is order-isomorphic to the set $\mathbb{Q}$ of rational numbers with the usual ordering. Let $G$ be the group of all order automorphisms of $\langle A,\le\rangle$.  
Let $\mathcal{F}$ be the (normal) filter of subgroups of $G$ generated by the subgroups $\fix_{G}(E)$, $E\in [A]^{<\omega}$. $\mathcal{N}3$ is the permutation model determined by $M$, $G$, and $\mathcal{F}$. If $x\in\mathcal{N}3$, then $\sym_{G}(x)\in\mathcal{F}$, so there exists $E\in [A]^{<\omega}$ such that $\fix_{G}(E)\subseteq\sym_{G}(x)$. Under these circumstances, we call $E$ a \emph{support} of $x$.
 
Let $X$ be an uncountable set in $\mathcal{N}3$. If $X$ is well orderable in $\mathcal{N}3$, then, by Corollary \ref{s8:c8}, we get that $\mathbf{S}(X,[X]^{\leq\omega})$ is a $P$-space in $\mathcal{N}3$.
 
Suppose $X$ is not well orderable in $\mathcal{N}3$. Let $E\in [A]^{<\omega}$ be a support of $X$. By the proof of Theorem \ref{thm:Models_M1_N3}(4), we may view $X$ in $\mathcal{N}3$ as a disjoint union $\bigcup\{X_{i}:i\in\gamma\}$, where $\gamma$ is a well-ordered cardinal number, such that, for all $i\in\gamma$, $X_{i}\subseteq [A]^{<\omega}$. Since $A$ is weakly Dedekind-finite in $\mathcal{N}3$ and $\mathbf{UT}(\mathbf{WO},\mathbf{WO},\mathbf{WO})$ is true in $\mathcal{N}3$ (see \cite{hr}),  it follows that $\{i\in\gamma: X_{i}$ is infinite$\}\ne\emptyset$. Without loss of generality, we assume that $X_{i}$ is infinite for all $i\in\gamma$. Moreover, we have 
\begin{equation}
\label{eq:sigma_ideal}
([X]^{\leq\omega}=[X]^{<\omega})^{\mathcal{N}3}\text{ and $([X]^{<\omega})^{\mathcal{N}3}$ is a $\sigma$-ideal.}
\end{equation}
Thus, $$(\mathbf{S}(X,[X]^{\leq\omega}))^{\mathcal{N}3}=(\mathbf{S}(X,[X]^{<\omega}))^{\mathcal{N}3}.$$ Furthermore, 

\begin{equation}
\label{eq:orb_basic_open}
(\forall \phi \in G)(\forall x,z\in [X]^{<\omega})(x\cap z=\emptyset\rightarrow\phi(B_{x,z})=B_{\phi(x),\phi(z)})
\end{equation}
and
$$(\forall x,z\in [X]^{<\omega})(x\cap z=\emptyset\rightarrow(B_{x,z})^{M}=(B_{x,z})^{\mathcal{N}3})$$
since, by (\ref{eq:orb_basic_open}), $(\bigcup x)\cup(\bigcup z)$ is a support of $B_{x,z}$.

To establish that $\mathbf{S}(X,[X]^{<\omega})$ is a $P$-space in $\mathcal{N}3$, it suffices to show, by the second clause of (\ref{eq:sigma_ideal}) and by Theorem \ref{s3:t6}(ii), that $\emptyset$ is a $P$-point of $\mathbf{S}(X,[X]^{<\omega})$ in $\mathcal{N}3$. To this end, let, in $\mathcal{N}3$, $\langle O_{n}\rangle_{n\in\omega}$ be a $\subseteq$-decreasing sequence of open neighborhoods of $\emptyset$ in $\mathbf{S}(X,[X]^{<\omega})$. Put 
$$U=\bigcap_{n\in\omega}O_{n}.$$
 
If $\langle O_{n}\rangle_{n\in\omega}$ has a finite range, then $U$ is open in $\mathbf{S}(X,[X]^{<\omega})$. So we assume that the range of $\langle O_{n}\rangle_{n\in\omega}$ is infinite and, without loss of generality, we further assume that $\langle O_{n}\rangle_{n\in\omega}$ is strictly $\subseteq$-decreasing. Let $E'\in [A]^{<\omega}$ be a support of $\langle O_{n}\rangle_{n\in\omega}$; so $E'$ is a support of $O_{n}$ for all $n\in\omega$. We can assume that $E\subseteq E'$; otherwise, consider the set $E\cup E'$ which is a support of $\langle O_{n}\rangle_{n\in\omega}$.  
Note that $U\in \mathcal{N}3$, since $E'$ is a support of $U$. Indeed, if $\phi\in\fix_{G}(E')$, then we have
$$\phi(U)=\phi\left(\bigcap_{n\in\omega}O_{n}\right)=\bigcap_{n\in\omega}\phi(O_{n})=\bigcap_{n\in\omega}O_{n}=U.$$

In $M$ (which satisfies $\mathbf{AC}$), pick an $\subseteq$-increasing sequence $\langle z_{n}\rangle_{n\in\omega}$ of elements of $[X]^{<\omega}$ such that
\begin{equation}
\label{eq:B_1}
(\forall n\in\omega)(B_{\emptyset,z_{n}}\subseteq O_{n}).
\end{equation}
(Note that the sequence $\langle z_{n}\rangle_{n\in\omega}$ may not be in $\mathcal{N}3$, but, for every $n\in\omega$, $z_{n}\in\mathcal{N}3$.) Since $\langle z_{n}\rangle_{n\in\omega}$ is $\subseteq$-increasing, we clearly have
\begin{equation}
\label{eq:B_2}
(\forall n\in\omega)(B_{\emptyset,z_{n+1}}\subseteq B_{\emptyset,z_{n}}).
\end{equation}
   
If $\langle z_{n}\rangle_{n\in\omega}$ has a finite range, then it is straightforward to find (in $\mathcal{N}3$) a basic open neighborhood of $\emptyset$ which is contained in $U$; in particular, take $B_{\emptyset,\bigcup_{n\in\omega}z_{n}}$. So we assume $\langle z_{n}\rangle_{n\in\omega}$ is strictly $\subseteq$-increasing. It follows that, for all but finitely many $n\in\omega$, $(\bigcup z_{n})\setminus E'\neq\emptyset$. Without loss of generality, we assume that 
$$(\forall n\in\omega)\left(\left(\bigcup z_{n}\right)\setminus E'\neq\emptyset\right).$$ 

For every $n\in\omega$, we define
\begin{align*}
V_{n}&=\bigcup\big\{\phi(B_{\emptyset,z_{n}}):\phi\in\fix_{G}(E')\big\}\\
&=\bigcup\big\{B_{\emptyset,\phi(z_{n})}:\phi\in\fix_{G}(E')\big\},
\end{align*}
\noindent where the second equality follows from (\ref{eq:orb_basic_open}) and from the fact that $\phi(\emptyset)=\emptyset$. We observe the following:
\begin{enumerate}
\item[(i)] For every $n\in\omega$, $V_{n}\in\mathcal{N}3$, since $E'$ is a support of $V_{n}$ for all $n\in\omega$;

\item[(ii)] for every $n\in\omega$, $V_{n}$ is open in $\mathbf{S}(X,[X]^{<\omega})$, since it is a union of basic open sets in $\mathbf{S}(X,[X]^{<\omega})$;

\item[(iii)] for every $n\in\omega$, we have
$$\emptyset\in V_{n+1}\subseteq V_{n}\subseteq O_{n}.$$ 
This follows from (\ref{eq:B_1}), (\ref{eq:B_2}), the definition of $V_{n}$ ($n\in\omega$), and from the fact that $E'$ is a support of $O_{n}$ for all $n\in\omega$.
\end{enumerate}

We let
$$V=\bigcap_{n\in\omega}V_{n}.$$
By (i), $E'$ is a support of $V$, so $V\in\mathcal{N}3$ (the argument is similar to the one for `$U\in\mathcal{N}3$').  
Furthermore, we have
\begin{equation}
\label{eq:V_subset_U}
\emptyset\in V\subseteq U.
\end{equation}

\begin{claim}
\label{cl:V_open}
In $\mathcal{N}3$, there exists $z\in [X]^{<\omega}$ such that $B_{\emptyset,z}\subseteq V$.
\end{claim}

\begin{proof}[Proof of claim]
Pick a $z\in [X]^{<\omega}\setminus\{\emptyset\}$ such that $\mathcal{P}(E')\cap X\subseteq z$. 
We assert that $B_{\emptyset,z}\subseteq V$. To this end, let $w\in B_{\emptyset,z}$ and also let $n\in\omega$. We will show that $w\in V_{n}$, that is, we will show that $w\in B_{\emptyset,\phi(z_{n})}$ for some $\phi\in\fix_{G}(E')$ (recall the definition of $V_{n}$). There are the following two cases. 
\smallskip

\textit{Case 1.} $w\cap z_{n}=\emptyset$. Then $w\in B_{\emptyset,z_{n}}\subseteq V_{n}$; the latter inclusion follows from the definition of $V_{n}$, taking $\phi$ therein to be the identity permutation of $A$. 
\smallskip

\textit{Case 2.} $w\cap z_{n}\ne\emptyset$. Then, in view of the fact that $(\bigcup z_{n})\setminus E'\neq\emptyset$, of the choice of $z$, and of the fact that $w\cap z=\emptyset$, it is not hard to construct a $\phi\in\fix_{G}(E')$ such that 
\begin{equation}
\label{eq:phi}
\phi(z_{n})\cap w=\emptyset.
\end{equation} 
Indeed, first note that, since $w\cap z=\emptyset$ and $\mathcal{P}(E')\cap X\subseteq z$, we have that, for every $u\in w\cap z_{n}$, $u\setminus E'\neq\emptyset$. (And note that it is possible to have $(\mathcal{P}(E')\setminus\{\emptyset\})\cap z_{n}\neq\emptyset$.) So now, since  $z_{n}\subseteq [A]^{<\omega}$ and $\langle A,\leq\rangle$ is order-isomorphic to $\mathbb{Q}$ with the usual ordering, it is fairly easy to construct a $\phi\in \fix_{G}(E')$ such that, for every $u\in z_{n}$ with $u\setminus E'\neq\emptyset$, $\phi(u)\not\in w$. Note also that, for every $u\in\mathcal{P}(E')\cap z_{n}$, $\phi(u)=u\not\in w$, due to the choices of $z$ and $w$. But then, (\ref{eq:phi}) holds as required.  

Since $\phi\in\fix_{G}(E')$, (\ref{eq:phi}) together with the definition of $V_{n}$ gives us  
$$w\in B_{\emptyset,\phi(z_{n})}\subseteq V_{n}.$$ 
The above arguments complete case 2 and the proof of the claim. 
\end{proof}

By (\ref{eq:V_subset_U}) and Claim \ref{cl:V_open}, we deduce that $\emptyset$ is a $P$-point of $\mathbf{S}(X,[X]^{<\omega})$ in $\mathcal{N}3$. By Theorem \ref{s3:t6}(ii), $\mathbf{S}(X,[X]^{<\omega})$ is a $P$-space in $\mathcal{N}3$ as required.  
\end{proof}

\end{document}